\documentclass[11pt]{amsart} 
\usepackage{amsmath,amsfonts,amssymb,amsthm,amscd}
\usepackage{hyperref}
\usepackage{graphics}
\usepackage[latin1]{inputenc}
\usepackage{color}

\newcommand{\suchthat}{\;\ifnum\currentgrouptype=16 \middle\fi|\;}

\allowdisplaybreaks

\DeclareMathOperator{\id}{id}

\DeclareMathOperator{\diag}{diag}
\DeclareMathOperator{\Diag}{Diag}
\DeclareMathOperator{\Aut}{Aut}
\DeclareMathOperator{\rank}{rank}
\DeclareMathOperator{\tr}{tr.deg.}

\DeclareMathOperator{\GL}{GL}

\newcommand{\Q}{\mathbb{Q}}
\newcommand{\R}{\mathbb{R}}       
\newcommand{\C}{\mathbb{C}}  

\newcommand{\qc}{\mathfrak{c}}

\newcommand{\K}{\mathbb{K}} 	
\theoremstyle{plain}
\newtheorem{theorem}{Theorem}[section]

\newtheorem{prop}[theorem]{Proposition}

\newtheorem{fact}[theorem]{Fact}
\newtheorem{corollary}[theorem]{Corollary}
\newtheorem{proposition}[theorem]{Proposition}
\newtheorem{lemma}[theorem]{Lemma}

\theoremstyle{definition}

\newtheorem{defnprop}[theorem]{Definition-Proposition}
\newtheorem{definition}[theorem]{Definition}

\theoremstyle{remark}

\newtheorem{remark}[theorem]{Remark}

\title{Two-dimensional locally $\mathbb{K}$-Nash groups}
\author{El\'ias Baro}
\address{Departamento de \'Algebra, Facultad de Matem\'aticas, Universidad Complutense de Madrid, 
28040 Madrid, Spain. 
E-mail address: ebaro@ucm.es}
\author{Juan de Vicente}
\address{Departamento de Matem\'aticas, Universidad Aut\'onoma de Madrid, 28049 Madrid, Spain. 
E-mail address: juan.devicente@uam.es}
\author{Margarita Otero}
\address{Departamento de Matem\'aticas, Universidad Aut\'onoma de Madrid, 28049 Madrid, Spain. 
E-mail address: margarita.otero@uam.es}

\thanks{The first and third authors are partially supported by  Spanish MTM2014-55565-P and Grupos UCM 910444. Second author also supported by a grant of the International
Program of Excellence in Mathematics at Universidad Aut\'onoma de Madrid.}


\date{31 October 2017}

\begin{document}
 
\begin{abstract} We give a classification of 
connected abelian locally (real) Nash  groups of dimension two. We first  consider Painlevé's description of meromorphic maps admitting an Algebraic Addition Theorem and analyse the algebraic dependence of  such maps. We then give a classification of connected abelian locally  complex Nash   groups of dimension two,  from which we deduce the  corresponding real classification.
\end{abstract}
\maketitle

\section{Introduction}\label{FCintroduction} 

The original motivation of this work  was to provide a classification of the two-dimensional abelian \emph{real} algebraic groups, \emph{i.e.}, the real points of  complex algebraic groups defined over $\mathbb{R}$. Abelian irreducible \emph{complex} algebraic groups have  already been classified by J.P.\,Serre \cite{Serre} and therefore a natural approach could be to study the \emph{descent datum} of the latter (see \S\,\ref{finalcomments}, below). Instead, in this paper, we will follow  a different  path: we will classify two-dimensional  real algebraic abelian groups -- up to isogeny -- via  a description of   simply connected two-dimensional abelian locally Nash groups. 

M.\,Shiota \cite{Shiota1} introduced locally Nash manifolds as a framework to study covers of Nash manifolds. On the other hand, A.\,Pillay \cite{Pillay} proved that every semialgebraic group is a Nash group. Therefore, locally Nash groups are the appropriated category to consider coverings of semialgebraic groups, in particular, coverings of real algebraic groups (moreover, E.\,Hrushovski and A.\,Pillay \cite{Hrushovski_Pillay} proved that  simply connected locally Nash groups are precisely  universal coverings of the connected components of  real algebraic groups). Our classification of the two dimensional simply-connected abelian locally Nash groups goes via their complex analogues  and  gives  an alternative proof -- by analytic methods -- of the results of Serre,  in dimension two. In particular, we can apply our results to obtain a classification of the mentioned descent datum.

In \cite{BDOLCN} we have introduced the category of locally $\mathbb{C}$-Nash groups, which is the complex version of the Nash one. We will make a strong use of the $\C$-Nash category  in the present (real) classification problem. Henceforth, we will denote by $\mathbb{K}$ either $\mathbb{R}$ or  $\mathbb{C}$, and we will speak of locally $\mathbb{K}$-Nash groups. Madden and Stanton already noted that meromorphic maps admitting an algebraic addition theorem play an important role in the study of abelian 	locally Nash groups. Recall that a meromorphic map $f:\mathbb{C}^n\dashrightarrow \mathbb{C}^n$ admits an \emph{algebraic addition theorem}  (AAT) if its coordinate functions $f_1,\ldots,f_n$ are 
algebraically independent over $\mathbb{C}$ and each $f_i(u+v)$ is algebraic over $\mathbb{K}(f_1(u),\ldots,f_n(u),f_1(v),\ldots,f_n(v)))$. Invariance with respect to complex conjugation will play an important role when studying the real case via the complex one. We say that a map $f:\mathbb{C}^n\dashrightarrow \mathbb{C}^m$ is \emph{$\mathbb{C}$-meromorphic} 
if it is just meromorphic and \emph{$\mathbb{R}$-meromorphic} if it is invariant meromorphic, \emph{i.e.}, if $\overline{f(\overline{u})}=f(u)$, for each $u\in \mathbb{C}^n$ where $f$ is defined. 
It can be checked (see the discussion before \cite[Rmk.\,4.3]{BDOLCN}) that if $f:\mathbb{C}^n\dashrightarrow \mathbb{C}^n$ is a $\mathbb{K}$-meromorphic map admitting an AAT and there exist $a\in \mathbb{K}^n$ and an open neighbourhood $U\subseteq \mathbb{K}^n$ of $0$ such that
$$
U\rightarrow \mathbb{K}^n: u\mapsto f(u+a) 
$$
is an analytic diffeomorphism, then we get a $\K$-Nash atlas  on $(\K^n,+)$ by taking translates of the above chart.  Hereafter, \emph{when we write $(\mathbb{K}^n,+,f)$, is a locally $\K$-Nash group  we are assuming that the map $f$ has the above properties and the relevant atlas is given as above}. In \cite{BDOLCN} we prove that all  simply connected abelian locally $\mathbb{K}$-Nash groups are (up to isomorphism) of this form.

In this paper we classify  two-dimensional   simply connected abelian locally $\mathbb{K}$-Nash groups and compute their  automorphism groups. 
The latter computation -- in the complex case -- will be specially useful  for the real classification.
More precisely, the content of this paper is the following. Sections\,\ref{one-dimensional} to \ref{Sproofs} are  dedicated to the complex case.  In \S\,\ref{one-dimensional}, we give a classification of one-dimensional simply connected locally $\mathbb{C}$-Nash groups 
(Theorem\,\ref{1dim C-classification}) and of their automorphism groups (Proposition\,\ref{C-automorphisms}) and in \S\,\ref{Sproofs} we give the corresponding two-dimensional  abelian results  (Theorem\,\ref{2dim C-classification} and Proposition\,\ref{C-automorphisms 2}). To compute the automorphism groups, it will be specially useful  a constant $\mathfrak c$ -- associated to the various Weierstrass functions -- which we introduce  in Definition\,\ref{defci}. In \S\,\ref{SPainleve}
 we consider  Painlev\'e's description  of the meromorphic  maps admitting an AAT -- which will be essential for the classification  given in \S\,\ref{Sproofs} -- and we characterise  algebraic dependence of the maps involved in such  description (Propositions\,\ref{Z general} and \ref{S general}).  In Section\,\ref{2dlng}  we give the main result of the paper: the classification of the simply connected two-dimensional abelian locally Nash groups (Theorem\,\ref{2dim R-classification}).

Finally, in Section\,\ref{finalcomments} we will apply the previous results to study two-dimensional algebraic groups. In \cite{BDOLCN} we proved that abelian simply-connected locally $\C$-Nash groups are universal coverings of irreducible complex algebraic groups. In Proposition\,\ref{algebraic groups} we identify these algebraic groups for the locally $\C$-Nash groups already obtained in the classification of Theorem\,\ref{2dim C-classification}. As a consequence, in Corollary\,\ref{isogenous corollary}, we give a classification of two-dimensional irreducible abelian algebraic groups up to isogeny (already obtained by J.P. Serre). In Proposition\,\ref{realalgebraicgroups} and Corollary\,\ref{real isogenous corollary} we give the real counterpart of these results, for which we need to study first the \emph{descent datum} of the algebraic groups of Corollary\,\ref{isogenous corollary}.
 
The results of this paper are part of the second author's Ph.D. thesis \cite{tesis}.

\section{Preliminaries}
We begin  by stating two results  which will be used in the rest of this paper without  further reference to them.

\begin{fact}\cite[Props.\,4.9 and 4.5, and Rem.\,4.3]{BDOLCN}\label{Kabelian} \emph{(1)} Every simply connected $n$-dimensional abelian locally $\mathbb{K}$-Nash group is isomorphic to some $(\mathbb{K}^n,+, f)$.

\noindent \emph{(2)} 
The locally $\mathbb{K}$-Nash groups $(\mathbb{K}^n,+,f )$ and $(\mathbb{K}^n,+,g)$ are isomorphic if and only if 
there exists an $\alpha \in \GL_n(\mathbb{K})$ such  that   $g\circ \alpha$ is algebraic over $\mathbb{K}(f)$.

\noindent \emph{(3)}  If $(\mathbb{R}^n,+, f)$ is a locally Nash group then $(\mathbb{C}^n,+, f)$ is a locally $\C$-Nash group. Conversely, if $(\mathbb{C}^n,+, f)$ is a locally $\C$-Nash group and $f$ is $\R$-meromorphic then  $(\mathbb{R}^n,+, f)$  is a locally Nash group.
\end{fact}

\begin{fact}\label{Lemma 3}\cite[Lem.\,3]{BDOAAT} If $f:\C^n\dashrightarrow\C^n$  is a  meromorphic map which admits an AAT then $f(u+a)$ is algebraic over $\C(f(u))$, for each $a\in\C^n$.
\end{fact}

Given a set of maps $S$ from $\K^m$ to $\K^n$, we say that a map $f:\K^m\to\K^n$ is \emph{algebraic over} $\K(S)$ if each coordinate function of $f$ is algebraic over the field generated -- over $\K$ -- by the coordinates functions of the maps in $S$. In that case we write $$f\in \K(S)^{alg}.$$

Recall that given a meromorphic map $f:\mathbb{C}^n\dashrightarrow \mathbb{C}^m$,  the \emph{group of periods of $f$} is
$\Lambda _f:=\{  a \in \mathbb{C}^n \suchthat f(u)=f(u+a)\}$. If $f:\mathbb{C}^n\dashrightarrow \mathbb{C}^n$  is a local diffeomorphism  at some point then $\Lambda _f$ is a discrete subgroup of $(\mathbb{C}^n,+)$, see \cite[Lem.\,4.6]{BDOLCN}, its \emph{rank} is its dimension as a free $\mathbb{Z}$-module. 
 
\begin{fact}\label{BDOrank}\emph{1)} \cite[Lem.\,4.7]{BDOLCN} 
Let $f,g:\mathbb{C}^n\dashrightarrow \mathbb{C}^n$ be meromorphic maps such that  $\Lambda _g$ is a discrete subgroup of $(\mathbb{C}^n,+)$ and  $g\in\mathbb{C}(f)^{alg}$.  Then, $\Lambda _f$ is also discrete and there exists $N\in \mathbb{N}\setminus \{0\}$ such that $N\Lambda _f\leq\Lambda _g$,  
and so  $\rank \Lambda _f \leq \rank \Lambda _g$.  Moreover, if the coordinate functions of $g$ are algebraically independent over $\mathbb{C}$ then $\rank \Lambda _f = \rank \Lambda _g$.

\noindent \emph{2)} \cite[Prop.\,4.8]{BDOLCN} Let $(\mathbb{K}^n,+,f)$ and $(\mathbb{K}^n,+,g)$ be isomorphic locally $\K$-Nash groups. Then $\rank \Lambda _f=\rank \Lambda _g$. 
\end{fact}
The latter fact allows us to extend the notion of rank to certain families of fields as follows -- we will use this invariant in Section\,\ref{SPainleve} --.
\begin{definition}\label{Z-rank}
Let $\mathbb{L}$ be a field of meromorphic functions from $\mathbb{C}^n$ to $\mathbb{C}$.
Let $f=(f_1,\dots,f_n):\mathbb{C}^n\dashrightarrow \mathbb{C}^n$ be  such that $\{ f_1,\ldots ,f_n\}$ is a 
transcendence basis of $\mathbb{L}$ over $\mathbb{C}$ and $\Lambda _f$ is a discrete subgroup of $\mathbb{C}^n$.
We define $\rank\mathbb{L}:=\rank  \Lambda _f$. If $\mathcal{P}$ is a family of such fields,  all of the same rank, then we define $\rank\mathcal{P}$ as this common rank.
\end{definition}

We will also need  some basic properties of the Weierstrass $\wp$,  $\zeta$ and $\sigma$ functions that we now recall.
Given a lattice $\Omega$ of $\mathbb{C}$, we will consider the Weierstrass functions $\sigma _{\Omega}$, $\zeta _{\Omega}$ and 
$\wp _{\Omega}$.
Recall that $\zeta _{\Omega}(u) = \frac{\sigma '_{\Omega}(u)}{\sigma _{\Omega}(u)}$ and $\wp _{\Omega}(u)  = -\zeta '_{\Omega}(u)$.
Given $\xi \in \mathbb{C}$, we introduce  the following  notation $$\widetilde{\sigma}_{\Omega,\xi}(u):=\frac{\sigma _{\Omega}(u-\xi)}{\sigma _{\Omega}(u)}.$$
Also, given $\omega\in \mathbb{C}\setminus \mathbb{R}$, we denote by $\wp _{\omega}$, $\zeta _{\omega}$, $\sigma _{\omega}$, 
$\widetilde{\sigma} _{\omega ,\xi}$ the functions, $\wp _{{\langle 1,\omega \rangle}_\mathbb{Z}}$, $\zeta _{{\langle 1,\omega \rangle}_\mathbb{Z}}$, 
$\sigma _{{\langle 1,\omega \rangle}_\mathbb{Z}}$ and $\widetilde{\sigma} _{{\langle 1,\omega \rangle}_\mathbb{Z},\xi}$ respectively.

\begin{fact}\label{Chandrasekharan} Let $\Omega$ be a lattice of $\mathbb{C}$ and $\xi \in \mathbb{C}$.
Then, for each $a\in \mathbb{C}^*$ we have that: ${\wp _{a\Omega}(au)=a^{-2}\wp _{\Omega}(u)},  {\zeta _{a\Omega}(au)=a^{-1}\zeta _{\Omega}(u)},  \sigma _{a\Omega}(au)=a\sigma _{\Omega}(u)$ and $\widetilde{\sigma} _{a\Omega,a\xi}(au)=\widetilde{\sigma} _{\Omega,\xi }(u)$.
\end{fact}
\begin{proof}
See \cite[Ch.III eq. $(2.2)$, Ch.IV eqs. (1.4)  and (2.6)]{Chandrasekharan}. The last equality then follows by definition.
\end{proof}

\begin{lemma}\label{cosets} Let $\Omega _1$ be a sublattice  of index $n$ of  a lattice  $\Omega _2$ of  $\mathbb{C}$.
Let $\xi \in \mathbb{C}$ and 
$\Omega _2=\bigcup _{i=1}^n (\Omega _1 +a_i)$.
Then, there exist $\mathfrak{c},C,C'\in \mathbb{C}$ such that:

\emph{(1)} $\displaystyle  \wp _{\Omega _2}(u)=\sum _{i=1}^n \wp _{\Omega _1}(u+a_i)-\mathfrak{c}$.

\emph{(2) }$\displaystyle  \zeta _{\Omega _2}(u)=\sum _{i=1}^n \zeta _{\Omega _1}(u+a_i)+\mathfrak{c}u+C$.

\emph{(3) }$\displaystyle  \sigma _{\Omega _2}(u)=e^{(\mathfrak{c}/2)u^{2}+Cu+C'}\prod _{i=1}^n \sigma _{\Omega _1}(u+a_i)$.

\emph{(4)} $\displaystyle  \widetilde{\sigma} _{\Omega _2,\xi}(u)=
	e^{-\xi \mathfrak{c}u+(\mathfrak{c}/2)\xi^{2}-C\xi}\prod _{i=1}^n \widetilde{\sigma} _{\Omega _1,\xi}(u+a_i)$.

\end{lemma}
\begin{proof}All follow easily from the identity $\wp '_{\Omega _2}(u)=\sum _{i=1}^n \wp '_{\Omega _1}(u-a_i)$, noting that $\zeta _{\Omega} '=-\wp _{\Omega}$ and $(\ln (\sigma _{\Omega}(u)))'=\zeta _{\Omega}(u)$, for any lattice $\Omega$ of $\mathbb{C}$.
\end{proof}
Note that the constant $\mathfrak{c}$ of the lemma  only depends on $\Omega _1$ and $\Omega _2$, 
and not on the $a_1,\ldots ,a_n$ chosen.

\begin{definition}\label{defc}Let $\Omega _1$ and $\Omega _2$ be lattices of $\mathbb{C}$ with 
$\Omega _1\leq\Omega _2$. We define  the \emph{residue of $\Omega_1$ in $\Omega_2$} -- denoted by $\mathfrak{c}(\Omega _2,\Omega _1)$ --  the constant $\mathfrak{c}$  given by  the last lemma. 
\end{definition}

\begin{lemma}\label{C-wp} Let $\Lambda _1$ and $\Lambda _2$ be lattices of $\mathbb{C}$. Then, 

$1)$ $\wp _{\Lambda _1}$ and $\wp _{\Lambda _2}$ are algebraically dependent over $\mathbb{C}$ if and only if there exists a 
 lattice $\Lambda$ of $\mathbb{C}$ such that $\Lambda \leq \Lambda _1$ and $\Lambda \leq \Lambda _2$, and 
 
$2)$ $\wp _{\Lambda _1} \circ \alpha$ and $\wp _{\alpha ^{-1}(\Lambda _1)}$ are algebraically dependent over $\mathbb{C}$,   for any  $\alpha$ in $ \GL_1(\mathbb{C})$.
\end{lemma}
\begin{proof}$1)$ By Fact\,\ref{BDOrank} it is enough to prove that if $\Lambda_1\leq \Lambda_2$ then $\wp _{\Lambda _1}$ and $\wp _{\Lambda _2}$ are algebraically dependent over $\mathbb{C}$. Let $n=[\Lambda _2:\Lambda _1]$. Then, there exist $a_1,\ldots ,a_n,C\in \mathbb{C}$ such that 
$\wp _{\Lambda _2}(u)=\sum _{i=1}^n \wp _{\Lambda _1}(u+a_i)+C$.
By Fact\,\ref{Lemma 3} and since $\wp _{\Lambda _1}$ admits an AAT, $\wp _{\Lambda _1}(u+a)\in\mathbb{C}(\wp _{\Lambda _1}(u))^{alg}$, for all $a\in \mathbb{C}$, and so
 $\wp _{\Lambda _2}\in\mathbb{C}(\wp _{\Lambda _1})^{alg}$. Finally, $2)$ follows from the identity $\wp _{\Lambda _1} (u)=b^{2}\wp _{b\Lambda _1}(bu)$ for all $b\in \mathbb{C}^*$ (see \emph{e.g.},\,\cite[Ch.III]{Chandrasekharan}).
\end{proof}

We will make use of the following characterisation of algebraic independence of Weierstrass functions. For the rest of the paper, we will also make use of the following notation: Let $\omega\in\C\setminus\R$. $K_{\omega}:=\Q(\omega)$ is $\omega$ is quadratic over $\mathbb{Q}$, and $K_{\omega}:= \mathbb{Q}$, otherwise.

\begin{fact}[{\cite[Thm.\,3]{Brownawell_Kubota}}] \label{Weierstrass algebraicity}
Fix $m\in \mathbb{N}$ and, for each $i\in \{1,\ldots ,m\}$, let $\Omega _i:={\langle 1,\omega _i\rangle}_{\mathbb{Z}}$, with 
$\omega _i\in \mathbb{C}\setminus \mathbb{R}$.
Suppose that $\omega _j\neq (a+b\omega _i)(c+d\omega _i)^{-1}$ whenever $i\neq j$ and $a,b,c,d\in \mathbb{Z}$ with $ad-bc\neq 0$.
For each $l\in \{1,\ldots ,m+1\}$, let $a_{l,1},\ldots ,a_{l,n_l}\in \mathbb{C}^*$.
Then, the functions
\[
 \{ u, \wp _{\Omega _i}(a_{i,j}u),\zeta _{\Omega _i}(a_{i,j}u), \exp (a_{m+1,k}u) \suchthat i,j,k\}, 
\]
where $i\in \{1,\ldots ,m\}$, $j\in \{1,\ldots ,n_i\}$, $k\in \{1,\ldots ,n_{m+1}\}$,
are algebraically independent over $\mathbb{C}$ if and only if  both the set $\{a_{i,1},\ldots ,a_{i,n_i}\}$
is linearly independent over $K_{\omega_i}$, for each $i \in \{1,\ldots ,m\}$, and $\{a_{m+1,1},\ldots ,a_{m+1,n_{m+1}}\}$ is linearly independent over $\Q$.
\end{fact}

We finish this section with some results and comments concerning invariance with respect to complex conjugation.

\begin{lemma}\label{wp_conjugate}\label{zeta_conjugate} 
Let $\Omega$ be a lattice of $\mathbb{C}$ and $\xi \in \mathbb{C}^*$.
Then, $\wp _{\overline{\Omega}} (u)=\overline{\wp _{\Omega}(\overline{u})}$, 
$\zeta _{\overline{\Omega}} (u)=\overline{\zeta _{\Omega}(\overline{u})}$, 
$\sigma _{\overline{\Omega}} (u)=\overline{\sigma _{\Omega}(\overline{u})}$ and
$\widetilde{\sigma}_{\overline{\Omega},\overline{\xi}} (u)=\overline{\widetilde{\sigma}_{\Omega ,\xi}(\overline{u})}$.
Consequently, $\wp _\Omega$, $\zeta _\Omega$, $\sigma_\Omega$ are  $\R$-meromorphic functions if and only if 
$\Omega =\overline{\Omega}$.
\end{lemma}
\begin{proof}
We note that
\[
\overline{\wp _{\Omega}(u)}
=
\overline{\frac{1}{u^2}+
\sum _{\omega \in \Omega \setminus \{0\}}
\left( \frac{1}{(u-\omega ) ^2}-\frac{1}{\omega ^2} \right) }
=
\frac{1}{\overline{u}^2}+
\sum _{\omega \in \Omega \setminus \{0\}}
\left( \frac{1}{(\overline{u}-\overline{\omega }) ^2}-\frac{1}{\overline{\omega} ^2}\right).
\]
and therefore $\overline{\wp _{\Omega}(u)}=\wp _{\Omega}(\overline{u})$. The proof for $\zeta_{\Omega}$ and $\sigma_\Omega$ is similar using \cite[Ch.IV, eq.\,$(1.1)$]{Chandrasekharan} and \cite[Ch.IV, eq.\,$(2.5)$]{Chandrasekharan}.

\end{proof}

\begin{lemma}\label{real wp}  Let $\Lambda$ be a lattice of $\mathbb{C}$.

$1)$ Let $g:\mathbb{C}\dashrightarrow \mathbb{C}$ be an $\R$-meromorphic function algebraic over $\mathbb{C}(\wp _\Lambda )$ such that $\Lambda _g$ is a discrete subgroup of $\C$.
Then, there exists an invariant lattice $\Lambda '\leq\Lambda$ such that $g\in\mathbb{C}(\wp _{\Lambda '})^{alg}$.

$2)$ If $\wp _{\Lambda}$ and $\wp _{\overline{\Lambda}}$ are algebraically dependent then $\Lambda \cap \overline{\Lambda}$ is an invariant lattice. 

\end{lemma}
\begin{proof}
1) Since $\Lambda$ is a lattice, $\Lambda _g$ is also a lattice by Fact\,\ref{BDOrank}(1), which is clearly invariant. By the same fact, there exists an $N\in \mathbb{N}^*$ such that $\Lambda ':=N\Lambda_g$ is contained in $\Lambda$, so the result follows by Lemma\,\ref{C-wp}(1).

2) Clearly $\Lambda \cap \overline{\Lambda}$ is an invariant discrete subgroup, and by Lemma\,\ref{C-wp}(1) it contains a sublattice, as required.
\end{proof}

\begin{remark}\label{Desrank2}Note that a non-trivial invariant discrete subgroup $\Lambda$ of $\C$ is of rank $1$ if it is 
either of the form ${\langle a\rangle}_{\mathbb{Z}}$ or ${\langle ia\rangle}_{\mathbb{Z}}$, for some $a\in \mathbb{R}^*$; and it is of rank $2$ if it is has a finite 
index subgroup of the form ${\langle a,bi\rangle}_{\mathbb{Z}}$, for some $a,b\in \mathbb{R}^*$.
Indeed, since $\Lambda$ is invariant we must have $\overline{\lambda}\in \Lambda$, for any $\lambda \in \Lambda$.
The only special case is when $\Lambda={\langle \lambda ,\overline{\lambda}\rangle}_{\mathbb{Z}}$, with $\lambda =a+ib$ with both $a,b\neq 0$.
Then, ${\langle 2a,2ib\rangle}_{\mathbb{Z}}$ is the required finite index subgroup of $\Lambda$.
\end{remark}

\section{One-dimensional locally \texorpdfstring{$\mathbb{C}$}{C}-Nash groups}\label{one-dimensional}

One-dimensional (simply connected) locally $\C$-Nash groups are abelian. Hence, a classification of two-dimensional simply connected abelian  locally $\C$-Nash groups entails the one-dimensional classification. However,  we give in this section the one-dimensional classification to simplify the proof of the  two-dimensional one.

\begin{theorem}\label{1dim C-classification} 
Every simply connected one-dimensional locally $\mathbb{C}$-Nash group is isomorphic to a group of one and only one of the following types:
$$(1)\ (\mathbb{C},+,\id); \quad (2)\ (\mathbb{C},+,\exp), \text{and} \quad (3)\ (\mathbb{C},+,\wp _\omega), \text{ for some }\omega \in \mathbb{C}\setminus \mathbb{R}.$$
 
\noindent
The groups $(\mathbb{C},+,\wp _{\omega _1})$ and $(\mathbb{C},+,\wp _{\omega _2})$ are isomorphic if and only the associated elliptic curves are isogenous, \emph{i.e.}, $\omega _2=\frac{a\omega _1+b}{c\omega _1+d}$
for some $a,b,c,d\in \mathbb{Z}$ with $ad-bc\neq 0$.
If such is the case, the isomorphisms are of the form $\alpha (u)=\frac{nu}{c\omega _1+d}$, for some $n\in \mathbb{N}^*$.
\end{theorem}
\begin{proof}Since one-dimensional complex analytic groups are abelian, 
any simply connected one-dimensional locally $\mathbb{C}$-Nash group is isomorphic to some $(\mathbb{C},+,f)$ for some meromorphic function $f$ admitting an AAT. By a classical result of Weierstrass, 
 such $f\in\mathbb{C}(g\circ \alpha)^{alg}$, for some $\alpha\in \GL_1(\mathbb{C})$, and some $g$ which is either $\id,\exp$ or $\wp_\Lambda$, where   $\Lambda$ is  a lattice of $\mathbb{C}$. Note that the ranks of the groups of periods of $\id,\exp$ and  $\wp_\Lambda$ are $0,1$ and $2$, respectively. In addition, $\rank \Lambda_f=\rank \Lambda_g$. Therefore, none 
of the groups of types $(1)$, $(2)$ or $(3)$ can be isomorphic to other of a different type.  
Furthermore, we claim the following.

\noindent
\emph{Claim: If $\rank \Lambda _f=0$, then the identity map is an isomorphism from $(\mathbb{C},+,f)$ to $(\mathbb{C},+,\id)$. If $\rank \Lambda _f=1$, then the map $\gamma (u)=2\pi i\omega _0 ^{-1}u$ is an isomorphism 
from $(\mathbb{C},+,f)$ to $(\mathbb{C},+,\exp)$, where $\omega _0 \in \mathbb{C}^*$ is such that ${\langle \omega _0\rangle}_{\mathbb{Z}}=\Lambda _f$. If $\rank \Lambda _f=2$, then the identity map is an isomorphism from $(\mathbb{C},+,f)$ to 
$(\mathbb{C},+,\wp _{\Lambda _f})$.}

\noindent
\emph{Proof of the Claim.} We just give the proof of the  case $\rank f= 2$, and so $f\in\C(\wp _{\Lambda}\circ \alpha)^{alg}$, for some lattice $\Lambda$ of $\C$  and some $\alpha\in \GL_1(\mathbb{C})$. Making use of Lemma \,\ref{C-wp} we obtain what follows.
Since $f\in\mathbb{C}(\wp _{\alpha ^{-1}(\Lambda)})^{alg}$ and 
$\alpha ^{-1}(\Lambda)$ is the group of periods of $\wp _{\alpha ^{-1}(\Lambda)}$, by Fact\,\ref{BDOrank}
there exists $n\in \mathbb{N}^*$ such that $n\alpha ^{-1}(\Lambda)\leq\Lambda _f$.
Let $\Lambda ':=n\alpha ^{-1}(\Lambda)\leq \alpha ^{-1}(\Lambda)$. Therefore, $\wp _{\alpha ^{-1}(\Lambda)}\in\mathbb{C}(\wp _{\Lambda '})^{alg}$, and  since $\wp _{\Lambda '}\in\mathbb{C}(\wp _{\Lambda _f})^{alg}$ we get  that
 $f\in\mathbb{C}(\wp _{\Lambda_f})^{alg}$, and hence,\,the result follows.\hfill$\Box$
 
To finish the proof of the theorem, first note that the special form of the lattice in case $(3)$ can be obtained since for any $\mathbb{R}$-linearly independent $\omega _1,\omega _2 \in \mathbb{C}^*$, the map $u\mapsto \omega _1^{-1}u$ is an isomorphism from 
$(\mathbb{C},+,\wp _{{\langle \omega _1,\omega _2\rangle}_{\mathbb{Z}}})$ to $(\mathbb{C},+,\wp _{\omega _1 ^{-1}\omega _2})$. Finally, if $(\mathbb{C},+,\wp _{\omega _1})$ and $(\mathbb{C},+,\wp _{\omega _2})$ are isomorphic, then 
there exists $\alpha \in \GL_1(\mathbb{C})$ such that $\wp _{\omega _2} \circ \alpha\in\mathbb{C}(\wp _{\omega _1})^{alg}$. Let  $\Lambda _i:={\langle 1,\omega _i\rangle}_{\mathbb{Z}}$, $i\in \{1,2\}$, and  $\alpha (u)=\tau u$, with $\tau\in \mathbb{C}^*$.
Let $\Lambda _2':=\alpha ^{-1}(\Lambda _2)={\langle \tau^{-1},\tau^{-1}\omega _2\rangle}_{\mathbb{Z}}$, which is also the group of 
periods of $\wp _{\Lambda _2} \circ \alpha$.
	Since $\wp _{\Lambda _1}\in\mathbb{C}(\wp _{\Lambda _2} \circ \alpha)^{alg}$, by Fact\,\ref{BDOrank} there 
exists $n\in \mathbb{N}^*$ such that $n\Lambda _2'\leq  \Lambda _1$.
So $n\tau^{-1}\in {\langle 1,\omega _1\rangle}_{\mathbb{Z}}$ and, hence, $n\tau^{-1}=c\omega _1+d$, for some $c,d\in \mathbb{Z}$.
Also $n\tau^{-1}\omega _2\in {\langle 1,\omega _1\rangle}_{\mathbb{Z}}$ and, hence, $\omega _2=(a\omega _1+b)(c\omega _1+d)^{-1}$, for some 
$a,b,c,d\in \mathbb{Z}$.
Furthermore, since $\omega _2\notin \mathbb{R}$, we get that $a\omega _1+b$ and $c\omega _1+d$ are linearly independent over 
$\mathbb{R}$ and, so $ad-bc\neq 0$. To prove the other implication, let $\alpha(u)=\frac{nu}{cw_1+d}$ with $n\in \mathbb{N}^*$. Then, $n\alpha^{-1}(\Lambda_2)={\langle cw_1+d,aw_1+b\rangle}_\mathbb{Z}$ is a sublattice of $\Lambda_1$.
Hence,  $\wp _{\alpha^{-1}(\Lambda_2)}\in\mathbb{C}(\wp _{\Lambda _1})^{alg}$.
Thus, since  $\wp _{\Lambda _2}\circ \alpha\in\mathbb{C}(\wp _{\alpha ^{-1}(\Lambda _2)})^{alg}$, it also algebraic over  $\mathbb{C}(\wp _{\Lambda _1})$, as required.
\end{proof}

\begin{proposition}\label{C-automorphisms} The automorphism groups of the one-dimensional simply connected locally $\C$-Nash groups are as follows: $\Aut(\mathbb{C},+,\id)\cong \mathbb{C}^*$, $\Aut(\mathbb{C},+,\exp) \cong \mathbb{Q}^*$ and,  for each $\omega\in \mathbb{C}\setminus \mathbb{R}$, $\Aut(\mathbb{C},+,\wp _{\omega})$ is isomorphic to  $K_{\omega}$.
\end{proposition}
\begin{proof}
In each case, $f=\id,\exp,\wp _{\omega}$,  it is enough to check for which $a\in \mathbb{C}^*$ the map $\alpha (u)=au$ has the property that $f\circ \alpha\in\mathbb{C}(f)^{alg}$. That can  be done by routine  computations, using Theorem\,\ref{1dim C-classification}  for the last case.
\end{proof}

\section{Painlev\'e families of maps admitting an AAT}\label{SPainleve}

In this section we will recall the description of  two-variable meromorphic maps admitting an AAT, given by  Painlev\'e \cite{Painleve}, and analyse the algebraic dependence of the maps involved in such description. Let $\mathbb{K}$ be $\mathbb{R}$ or $\mathbb{C}$,  following Painlevé, we say that 
$f:\mathbb{K}^n\dashrightarrow \mathbb{K}^n$ 
is \emph{functionally independent} if its image 
has an interior point in $\mathbb{K}^n$. Thus,  in particular, the coordinate functions of such $f$ are algebraically independent  over $\K$. We denote by $\mathbb{C}(\Lambda )$ the field of meromorphic functions $f:\mathbb{C}^n\dashrightarrow \mathbb{C}$ such that $\Lambda\leq \Lambda_f$. The transcendence degree of $\mathbb{C}(\Lambda )$ can be from $0$ to $n$, and it is $n$ if and only if it contains a function whose period group is discrete (see,
\cite[Ch. 5 \textsection 11 Thms.\,5 and 6]{Siegel}).

\begin{fact}[{Painlev\'e, \cite[Main Thm.]{Painleve}}]\label{TP}
If $f:\mathbb{C}^2\dashrightarrow \mathbb{C}$ is a functionally independent meromorphic function admitting 
 an AAT then there exist $i\in \{1,\ldots ,6\}$ such that the coordinate functions of $f$ are both algebraic over a  field  of the family $\mathcal{P}_i$, where:
\begin{align*}
\mathcal{P}_1 &:=\{ \, \mathbb{C}(g_1\circ \alpha ) \suchthat \alpha \in \GL_2(\mathbb{C})\, \}, \text{ where } g_1(u,v):=(u,v).\\
\mathcal{P}_2 &:=\{ \, \mathbb{C}(g_2\circ \alpha ) \suchthat \alpha \in \GL_2(\mathbb{C})\, \}, \text{ where } g_2(u,v):=(e^u,v).\\
\mathcal{P}_3 &:=\{ \, \mathbb{C}(g_3\circ \alpha ) \suchthat \alpha \in \GL_2(\mathbb{C})\, \}, \text{ where } g_3(u,v):=(e^u,e^v).\\
\mathcal{P}_4 &:=\{ \, \mathbb{C}(g_{4,\Omega,\xi}\circ \alpha ) \suchthat \alpha \in \GL_2(\mathbb{C}),\, \xi \! \in \{0,1\},
\, \Omega \text{ is a lattice of } \mathbb{C}\, \},\\
& \text{ where $g_{4,\Omega,\xi}(u,v)=(\wp _{\Omega }(u),v-\xi \zeta _{\Omega } (u))$.}\\
\mathcal{P}_5 & :=\{ \, \mathbb{C}(g_{5,\Omega,\xi}\circ \alpha ) \suchthat \alpha \in \GL_2(\mathbb{C}),\, \xi \in \mathbb{C}, 
\,\Omega \text{ is a lattice of } \mathbb{C}\, \},\\
& \text{ where $g_{5,\Omega,\xi }(u,v)=\left(\wp _{\Omega}(u),\frac{\sigma _\Omega (u-\xi )}{\sigma _\Omega (u)}e^v\right)$.}\\
\mathcal{P}_6 &:=\{  \, \mathbb{C}(\Lambda ) \suchthat \Lambda \text{ is a lattice of } \mathbb{C}^2,\,
\tr_{\mathbb{C}}\mathbb{C}(\Lambda )=2 \}.
\end{align*}
\end{fact}

We point out that $g_1,g_2,g_3$  y $g_{4,\Omega,0}$ clearly admit an AAT and $g_{4,\Omega,1}$ and $g_{5,\Omega,\xi}$ also admit an AAT by \cite[Art.\,16 and 19]{Painleve}, moreover, each one  is  functionally independent. On the other hand, by \cite[Ch.\,5. \textsection\,13]{Siegel}, any map whose coordinate functions form a transcendence basis of a field in  $\mathcal{P}_6$ satisfies an AAT  and it is  functionally independent.  Now,  functionally independence implies  local diffeomorphism at some point. Therefore,  all the maps mentioned induce locally $\mathbb{C}$-Nash group structures on $(\mathbb{C}^2,+)$, and so they do  their compositions with some $\alpha$ in $\GL_2(\mathbb{C})$.

From Painlev\'e's result we will obtain  -- in the next section --  representatives of  isomorphism classes of two-dimensional simply connected abelian locally $\C$-Nash groups. We begin by computing the $\rank$s of the Painlev\'e's families,  this will show that locally $\C$-Nash groups coming from different families cannot be isomorphic.
\begin{proposition}\label{rank} The $\rank$s of the Painlev\'e's families are: $\rank \mathcal{P}_1=0,\, \rank \mathcal{P}_2=1,\, \rank \mathcal{P}_3=2,\, \rank \mathcal{P}_4=2,\, \rank \mathcal{P}_5=3,\, \rank \mathcal{P}_6=4.$	
\end{proposition}
\begin{proof}For the first five families it is enough to compute the ranks of $g_1,g_2,g_3$, $g_{4,\Omega,\xi}$ and $g_{5,\Omega,\xi}$ where $\Omega:={\langle \omega _1,\omega _2\rangle}_{\mathbb{Z}}$ 
is a lattice of $\mathbb{C}$ and $\xi \neq 0$. It is easy to check that $\displaystyle \Lambda _{g_1}=\{(0,0)\}$,
 $\displaystyle \Lambda _{g_2}={\langle(2\pi i,0)\rangle}_{\mathbb{Z}}$ and
 $\displaystyle \Lambda _{g_3}={\langle(2\pi i,0), (0,2\pi i)\rangle}_{\mathbb{Z}}$. 
 
 We show that $\displaystyle \Lambda _{g_{4,\Omega,\xi}}={\langle(\omega _1 ,2\xi \zeta _\Omega (\omega _1/2)),
(\omega _2 ,2\xi \zeta _\Omega (\omega _2/2))\rangle}_{\mathbb{Z}}$. Let $\Lambda$ denote the period group of $g_{4,\Omega,\xi}$, and let $\Lambda_1$ and $\Lambda_2$ denote the periods groups of its coordinate functions. Clearly, $\Lambda \leq \Lambda _{1}\cap \Lambda _{2}$. Fix $\lambda :=(\lambda _1,\lambda _2)\in \Lambda$.

Since $\lambda \in \Lambda _{1}$, we have $\lambda _1\in \Omega$.
Fix $m,n\in \mathbb{Z}$ such that $\lambda _1 =m\omega _1+n\omega _2$.
By \cite[Ch.IV, Thm.\,1]{Chandrasekharan} it follows that
\begin{equation}\label{Chfact}
\zeta _\Omega (u+m\omega _1+n\omega _2) - \zeta _\Omega (u)=2m\zeta _\Omega (\omega _1/2)+2n\zeta _\Omega (\omega _2/2).
\end{equation}
Since $\lambda \in \Lambda _{2}$, we deduce $v+\lambda _2-\xi \zeta _\Omega (u+m\omega _1+n\omega _2) = v-\xi \zeta _\Omega (u)
$
and, hence, by equation (\ref{Chfact}), 
$
\lambda _2=2\xi m\zeta _\Omega (\omega _1/2)+2\xi n\zeta _\Omega (\omega _2/2).
$
This means that the elements of $\Lambda$ are of the form 
\[
\big(m\omega _1+n\omega _2, 2\xi m\zeta _\Omega (\omega _1/2)+2\xi n\zeta _\Omega (\omega _2/2)\big),
\]
for some  $m,n\in \mathbb{Z}$, so we are done with this case.

We now show that $\Lambda _{g_{5,\Omega,\xi}}={\langle(\omega _1 ,2\xi \zeta _\Omega (\omega _1/2)),
(\omega _2 ,2\xi \zeta _\Omega (\omega _2/2),(0,2\pi i))\rangle}_{\mathbb{Z}}$.
Again, let $\Lambda$ denote the period group of $g_{5,\Omega,\xi}$, and let $\Lambda_1$ and $\Lambda_2$ denote the periods groups of its coordinate functions. Fix $\lambda :=(\lambda _1,\lambda _2)\in \Lambda$.
Reasoning as before, there exists $m,n\in \mathbb{Z}$ such that $\lambda_1 =m\omega _1+n\omega _2$.
Moreover, \cite[Ch.IV, Thm.\,3]{Chandrasekharan},
\begin{equation}\label{Chfact2}
\frac{\sigma _\Omega (u+m\omega _1+n\omega _2)}{\sigma _\Omega (u)}=
Ce^{u(2m\zeta _\Omega (\omega _1/2)+2n\zeta _\Omega (\omega _2/2))},
\end{equation}
for some constant $C\in \mathbb{C}$.
Since $\lambda \in \Lambda _{2}$, we deduce
\[
\frac{\sigma _\Omega (u+\lambda _1-\xi)}{\sigma _\Omega (u+\lambda _1)}e^{v+\lambda _2} 
=
\frac{\sigma _\Omega (u-\xi)}{\sigma _\Omega (u)}e^v.
\]
Consequently,
\[
e^{\lambda _2}=\frac{\sigma _\Omega (u-\xi)}{\sigma _\Omega (u)}\frac{\sigma _\Omega (u+\lambda _1)}{\sigma _\Omega (u+\lambda _1-\xi)}.
\]
So, by equation (\ref{Chfact2}), we get $e^{\lambda _2}=e^{2\xi (m\zeta _\Omega (\omega _1/2)+n\zeta _\Omega (\omega _2/2))}$, and therefore 
$\lambda _2=2\xi m\zeta _\Omega (\omega _1/2)+2\xi n\zeta _\Omega (\omega _2/2)+2p \pi i$, for some $p \in \mathbb{Z}$.
This means that the elements of $\Lambda _g$ are of the form 
\[
\big(m\omega _1+n\omega _2, 2\xi m\zeta _\Omega (\omega _1/2)+2\xi n\zeta _\Omega (\omega _2/2)+2p \pi i\big),
\]
with $m,n,p\in \mathbb{Z}$, as required.

Finally, we consider the case of the $\mathcal{P}_6$. Let  $\{f_1,f_2\}$ be a transcendence basis of $\mathbb{C}(\Lambda)$, and let us see that $\rank \Lambda_f=4$. 
By definition, $\Lambda \leq \Lambda_f$ and, therefore, it is enough to check that $\Lambda_f$ is discrete. 
Since $\tr_\mathbb{C}\mathbb{C}(\Lambda)=2$, there exists a meromorphic function $g\in \mathbb{C}(\Lambda)$ such that $\Lambda_g$ is discrete. 
 We conclude  by Fact\,\ref{BDOrank}(1), since $g\in\mathbb{C}(f)^{alg}$.
 \end{proof}

 Next we have a series of technical results dedicated to characterise   algebraic dependence of  maps  within a fixed Painlev\'e family.
The cases of the three first families are trivial and we will not treat the case of the sixth family (being beyond the objectives of this paper). The strategy for the family $\mathcal{P}_4$ is as follows:  Firstly, we reduce to the case of  lattices of the type $\langle1,\tau\rangle$ (Lemma\,\ref{Z lattice}); secondly, we study the case when the relevant lattices are one a subgroup of the other (Lemma\,\ref{Z sublattices}); and finally, we consider the general case (Proposition\,\ref{Z general}). The strategy for the family $\mathcal{P}_5$ is similar for the first  two steps (Lemmas\,\ref{S parameter} and \ref{S sublattices}). However,  the corresponding  general case  is more involved due to the existence -- besides a lattice --  of a parameter $\xi$. The latter is analised  in Lemmas\,\ref{S Theta}  and \ref{representation}  via the introduction (Definition\,\ref{defXi}) of the  set  $\varXi$  (see also Corollary\,\ref{S characterization}). We finish the characterisation of this family in  Proposition\,\ref{S general}.  We will make use of  Fact\,\ref{Chandrasekharan} without further reference.

\

We begin with the first step for the family $\mathcal{P}_4$.

\begin{lemma}\label{Z lattice} Let $\Omega :={\langle \omega _1,\omega _2\rangle}_{\mathbb{Z}}$ be a lattice of $\mathbb{C}$, 
$\xi \in \mathbb{C}^*$,  $\tau := \omega _1^{-1}\omega _2$ and define $\alpha (u,v) := (\omega _1 u,\xi \omega _1^{-1} v)\in \GL_2(\mathbb{C})$. 
Then, $\big(\wp _{\Omega}(u), v-\xi \zeta _{\Omega} (u)\big)\circ \alpha\in\mathbb{C}\big(\wp _\tau (u), v-\zeta _\tau (u)\big)^{alg}$.
\end{lemma}

For the second step for the family $\mathcal{P}_4$ we will use the notation $\mathfrak{c}(\Omega _2,\Omega _1)$ introduced in  Definition\,\ref{defc}.

\begin{lemma}\label{Z sublattices} Let $\Omega _1$ and $\Omega _2$ be lattices of $\C$ such that 
$\Omega _1\leq\Omega _2$. Let $\alpha (u,v) := (u,\mathfrak{c}(\Omega _2,\Omega _1)u+[\Omega _2:\Omega _1]v)$. Then, $\big(\wp _{\Omega _2}(u), v-\zeta _{\Omega_2} (u)\big)\circ \alpha\in\mathbb{C}\big(\wp _{\Omega _1}(u), v-\zeta _{\Omega _1} (u)\big)^{alg}$.
\end{lemma}
\begin{proof} Let $n$ be the index of $\Omega _1$ in $\Omega _2$,  $\mathfrak{c}:=\mathfrak{c}(\Omega _2,\Omega _1)$ and $\Omega _2=\bigcup _{i=1}^n(\Omega _1+a_i)$.
Let $C$ be as in Lemma\,\ref{cosets}.
Then, for $D:=\sum _{i=1}^n n^{-1}\mathfrak{c}a_i-C$, we have that
\[  v-\zeta _{\Omega _2}(u)
=D+\sum _{i=1}^n \left(n^{-1}v-n^{-1}\mathfrak{c}(u+a_i)-\zeta _{\Omega _1}(u+a_i)\right), 
\]
since $D-\sum _{i=1}^n n^{-1}\mathfrak{c}(u+a_i)=-\mathfrak{c}u-C$.

Let $\beta (u,v)=(u,n^{-1}v-n^{-1}\mathfrak{c}u)$, so that 
$v-\zeta _{\Omega _2}(u)=D+\sum _{i=1}^n\big(\big(v-\zeta _{\Omega _1}(u)\big)\circ \beta \big)(u+a_i,v)$.
Since $\big(\wp _{\Omega _1}(u), v-\zeta _{\Omega _1}(u)\big)$ admits an AAT,  so does
$\big(\wp _{\Omega _1}(u), v-\zeta _{\Omega _1}(u)\big)\circ \beta$.
By Fact\,\ref{Lemma 3}, we deduce that $$\big(\big(v-\zeta _{\Omega _1}(u)\big)\circ \beta \big) (u+a_i,v)\in\C\big(\big(\wp _{\Omega _1}(u), v-\zeta _{\Omega _1}(u)\big)\circ \beta \big)^{alg},$$ for each $i\in \{1,\ldots, n\}$.
So $v-\zeta _{\Omega _2}(u)\in\mathbb{C}\big(\big(\wp _{\Omega _1}(u), v-\zeta _{\Omega _1} (u)\big)\circ \beta \big)^{alg}$.
Since $\wp _{\Omega _2}(u)\in\C(\wp _{\Omega _1}(u))^{alg}$ and 
 $\wp _{\Omega _1}(u)=\big(\wp _{\Omega _1}(u)\big)\circ \beta$ -- as a function of two variables --,  we also get that 
$$\big(\wp _{\Omega _2}(u), v-\zeta _{\Omega _2}(u)\big)\in\mathbb{C}\big(\big(\wp _{\Omega _1}(u), v-\zeta _{\Omega _1} (u)\big)\circ \beta \big)^{alg}.$$
We are done by taking $\alpha = \beta ^{-1}$.
\end{proof}

We now  give the  characterisation  of algebraic dependence  of the maps of the family $\mathcal{P}_4$.

\begin{proposition}\label{Z general} 
Let $\Omega _1:={\langle 1,\omega _1\rangle}_{\mathbb{Z}}$ and $\Omega _2:={\langle 1,\omega _2\rangle}_{\mathbb{Z}}$, with 
$\omega _1,\omega _2 \in \mathbb{C}\setminus \mathbb{R}$.
Then, there exists $\alpha \in \GL_2(\mathbb{C})$ such that $\big(\wp _{\Omega _2}(u), v-\zeta _{\Omega _2}(u)\big)\circ \alpha$ is algebraic 
over $\mathbb{C}\big(\wp _{\Omega _1}(u), v-\zeta _{\Omega _1}(u)\big)$ if and only if there exist $a,b,c,d\in \mathbb{Z}$ such that $ad-bc\neq 0$ and $\omega _2=\frac{a\omega _1+b}{c\omega _1+d}$.
Furthermore, if such is the case, then for some  $n\in \mathbb{N}^*$, 
\[
\alpha (u,v)=\left(\rho u,\rho ^{-1}\left(\mathfrak{c}(\Omega:n\Omega)-\frac{\mathfrak{c}(\Omega_1,n\Omega)[\Omega:n\Omega]}{[\Omega_1:n\Omega]}\right)u+\frac{[\Omega :n\Omega ]}{[\Omega _1:n\Omega]}\rho ^{-1}v\right), 
\]
where $\rho :=\frac{n}{c\omega _1+d}$ and $\Omega=\rho ^{-1}\Omega_2$.
\end{proposition}
\begin{proof}
We begin with the right to left implication. Fix $n\in \mathbb{N}^*$ and let 
$\Omega :={\langle \frac{c\omega _1+d}{n},\frac{a\omega_1+b}{n}\rangle}_{\mathbb{Z}}$. 
Let $\alpha _1(u,v):=(\frac{c\omega _1+d}{n}u,\frac{n}{c\omega _1+d}v)$.
Then, applying  Lemma\,\ref{Z lattice}  to $\Omega$ and $\xi =1$, we get 
$$\big(\wp _{\Omega}(u), v-\zeta _{\Omega}(u)\big)\circ \alpha _1\in\mathbb{C}\big(\wp _{\Omega _2}(u), v-\zeta _{\Omega _2}(u)\big)^{alg}.$$
Next, we apply twice Lemma\,\ref{Z sublattices}, firstly to $n\Omega\leq\Omega$ and $$\alpha_2(u,v):=\big(u,\mathfrak{c}(\Omega,n\Omega)u+[\Omega:n\Omega]v\big),$$  and then to $n\Omega\leq\Omega_1$ and $$\alpha _3(u,v):=\big(u,\mathfrak{c}(\Omega _1,n\Omega)u+[\Omega _1:n\Omega]v\big),$$ and we get 
$\big(\wp _{\Omega}(u), v-\zeta _{\Omega}(u)\big)\in\mathbb{C}\big(\big(\wp _{\Omega _1}(u), v-\zeta _{\Omega _1}(u)\big)\circ \alpha _3\circ \alpha_2^{-1}\big)^{alg}$. Thus, 
$\big(\wp _{\Omega _2}(u), v-\zeta _{\Omega _2}(u)\big)\in\mathbb{C}\big(\big(\wp _{\Omega _1}(u), v-\zeta _{\Omega _1}(u)\big)\circ \alpha _3\circ \alpha_2^{-1} \circ \alpha_1\big)^{alg}$. 
We conclude by taking $\alpha=(\alpha _3\circ \alpha_2^{-1} \circ \alpha_1)^{-1}$, which also proves the furthermore clause.

For the converse, let $\alpha (u,v)= (a'u+b'v,c'u+d'v)\in \GL_2(\mathbb{C})$.
If we show that $a'\neq 0$ and $\wp _{\Omega _2}(a'u)\in\mathbb{C}(\wp _{\Omega _1}(u))^{alg}$, we get that   the function $u\mapsto a' u$ is an isomorphism between 
$(\mathbb{C},+,\wp _{\Omega _1})$ and $(\mathbb{C},+,\wp _{\Omega _2})$ and hence, the lemma follows from 
Theorem\,\ref{1dim C-classification}.
Thus,  it is enough to prove the following more general result:

\medskip

\noindent
 \emph{Claim. Let $\alpha (u,v)= (au+bv,cu+dv)\in \GL_2(\mathbb{C})$ and $\xi \in \{0,1\}$.
If $\big(\wp _{\Omega _2}(u), v-\zeta _{\Omega _2}(u)\big)\circ \alpha\in\mathbb{C}\big(\wp _{\Omega _1}(u), v-\xi\zeta _{\Omega _1}(u)\big)^{alg}$ then $b=0$ and $\wp _{\Omega _2}(au)$ is 
algebraic over $\mathbb{C}(\wp _{\Omega _1}(u))$.}

\smallskip

\noindent
\emph{Proof of the claim.} Assume for the sake of contradiction that $b\neq 0$.
 Note that $a\neq 0$, otherwise  $\wp _{\Omega _2}(bv)\in\mathbb{C}(\wp _{\Omega _1}(u), v- \xi\zeta _{\Omega _1} (u))^{alg}$. On the other hand, 
 $\wp _{\Omega _2}$ admits an AAT and so  $\wp _{\Omega _2}(au+bv)\in\mathbb{C}(\wp _{\Omega _2}(au),\wp _{\Omega _2}(bv))^{alg}$.
Hence, $\wp _{\Omega _2}(bv)\in\mathbb{C}(\wp _{\Omega _2}(au+bv), \wp _{\Omega _2}(au))^{alg}$. 
Thus, 
\[
\wp _{\Omega _2}(bv)\in\mathbb{C}\left(\wp _{\Omega _2}(au+bv), 
cu+dv-\zeta _{\Omega _2} (au+bv),\wp _{\Omega _2}(au)\right)^{alg}. 
\]
By hypothesis, we also have 
\[
\left(\wp _{\Omega _2}(au+bv), cu+dv-\zeta _{\Omega _2} (au+bv)
\right)\in \mathbb{C}(\wp _{\Omega _1}(u), v- \xi\zeta _{\Omega _1} (u))^{alg},
\]
hence 
\[
\wp _{\Omega _2}(bv)\in\mathbb{C}\left(\wp _{\Omega _1}(u), v- \xi\zeta _{\Omega _1} (u),\wp _{\Omega _2}(au)\right)^{alg}.
\]
Since $\wp _{\Omega _2}(bv)$ only depends on $v$, we get that $\wp _{\Omega _2}(bv)\in\mathbb{C}(v)^{alg}$.
This contradicts Fact\,\ref{Weierstrass algebraicity}.
Therefore,  $b=0$ and hence $a\neq 0$. We conclude that $\wp _{\Omega _2}(au)\in\mathbb{C}(\wp _{\Omega _1}(u), v-\xi\zeta_{\Omega_1} (u))^{alg}$. 
If $\wp _{\Omega _2}(au)\not\in\mathbb{C}(\wp _{\Omega _1}(u))^{alg}$, 
then $v-\xi\zeta_{\Omega_1} (u)\in\mathbb{C}(\wp _{\Omega _1}(u),\wp _{\Omega _2}(au))^{alg}$, which is impossible.
\end{proof}

We now  study  algebraic dependence for maps of  family $\mathcal{P}_5$. The problem is  that we cannot normalise simultaneously both parameters $\Omega$ and $\xi$.
We recall the notation $\widetilde{\sigma} _{\Omega ,\xi}(u):=\frac{\sigma _{\Omega}(u-\xi)}{\sigma_{\Omega}(u)}$. Again, the relevant lattices can be taken of the form $\langle1,\tau\rangle$, as follows:
\begin{lemma}\label{S parameter} 
Let $\Omega :={\langle \omega _1,\omega _2\rangle}_{\mathbb{Z}}$ be a lattice of $\mathbb{C}$ and $\xi \in \mathbb{C}$.
Let $\tau := \omega _1^{-1}\omega _2$.
Let $\alpha (u,v) := (\omega _1 u,v)$.
Then, $$\big(\wp _{\Omega}(u), e^v\widetilde{\sigma} _{\Omega,\xi}(u)\big)\circ \alpha\in\mathbb{C}\big(\wp _\tau (u),e^v\widetilde{\sigma} _{\tau,\omega _1^{-1}\xi}(u)\big)^{alg}.$$
\end{lemma}

The second step for the family $\mathcal{P}_5$ is as follows.

\begin{lemma}\label{S sublattices} Let $\Omega _1$ and $\Omega _2$ be lattices of $\C$ such that 
$\Omega _1\leq \Omega _2$, and let $\xi \in \mathbb{C}$. Let $\alpha (u,v) := (u,\xi \mathfrak{c}(\Omega _2,\Omega _1)u+[\Omega _2:\Omega _1]v)$.
Then, $$\big(\wp _{\Omega _2}(u), e^v\widetilde{\sigma} _{\Omega_2,\xi} (u)\big)\circ \alpha\in\mathbb{C}\big(\wp _{\Omega _1}(u), e^v\widetilde{\sigma}_{\Omega_1,\xi} (u)\big)^{alg}.$$
\end{lemma}
\begin{proof} Let $n:=[\Omega _2:\Omega _1]$, $\mathfrak{c}:=\mathfrak{c}(\Omega _2,\Omega _1)$ and $\Omega _2=\bigcup _{i=1}^n(\Omega _1+a_i)$.
Let $C$ be as in Lemma\,\ref{cosets} and let 
$D:=e^{(\mathfrak{c}/2)\xi ^2-C\xi}\prod_{i=1}^n e^{n^{-1}\xi \mathfrak{c}a_i}$.
Note that,
\[
e^{-\xi \mathfrak{c}u+(\mathfrak{c}/2)\xi ^2-C\xi}=D\prod _{i=1}^n (e^{-n^{-1}\xi \mathfrak{c}(u+a_i)}),
\]
thus,
\[ \displaystyle e^v\widetilde{\sigma}_{\Omega _2,\xi}(u)=
 D\prod _{i=1}^n \left(e^{n^{-1}v-n^{-1}\xi \mathfrak{c}(u+a_i))} 
\widetilde{\sigma}_{\Omega _1,\xi}(u+a_i)\right).
\]
Take $\beta (u,v):=(u,n^{-1}v-n^{-1}\xi \mathfrak{c}u)$, so that 
\[
e^v\widetilde{\sigma}_{\Omega_2,\xi} (u)=D\prod _{i=1}^n\big(\big(e^v\widetilde{\sigma}_{\Omega_1,\xi} (u)\big)\circ \beta \big)(u+a_i,v).
\]
Since $\big(\wp _{\Omega _1}(u), e^v\widetilde{\sigma}_{\Omega _1,\xi}(u)\big)$ admits an AAT, 
$\big(\wp _{\Omega _1}(u), e^v\widetilde{\sigma}_{\Omega _1,\xi}(u)\big)\circ \beta$ also admits an AAT.
By Fact\,\ref{Lemma 3}, we get 
$$\big(\big(e^v\widetilde{\sigma}_{\Omega _1,\xi}(u)\big)\circ \beta \big) (u+a_i,v)\in\C\big(\big(\wp _{\Omega _1}(u), e^v\widetilde{\sigma}_{\Omega _1,\xi}(u)\big)\circ \beta \big)^{alg},$$ 
for each $i\in \{1,\ldots n\}$.
So $e^v\widetilde{\sigma}_{\Omega _2,\xi}(u)\in\mathbb{C}\big(\big(\wp _{\Omega _1}(u), e^v\widetilde{\sigma}_{\Omega _1,\xi}(u)\big)\circ \beta \big)^{alg}$. Then, reasoning as in Lemma\,\ref{Z sublattices}, we get 
 that $\big(\wp _{\Omega _2}(u), e^v\widetilde{\sigma}_{\Omega _2,\xi}(u)\big)\circ \beta ^{-1}\in\mathbb{C}\big(\wp _{\Omega _1}(u), e^v\widetilde{\sigma}_{\Omega _1,\xi}(u)\big)^{alg} $.
\end{proof}

We now  fix  a lattice in $\C$ and consider  the parameters (in $\C$) which give us algebraic dependent maps in the family $\mathcal{P}_5$

\begin{definition}\label{defXi}Let $\Omega$ be a lattice of $\C$ and  $\xi \in \mathbb{C}$. We define

\smallskip
 
 \noindent
%
%

\noindent
{\small$\varXi (\Omega ,\xi):=\{ \rho\in \mathbb{C}:\exists\alpha \in \GL_2(\mathbb{C}), \left(\wp _{\Omega} (u),e^v\widetilde{\sigma} _{\Omega ,\rho} (u)\right)\circ \alpha\in \mathbb{C}\left(\wp _{\Omega}(u),e^v\widetilde{\sigma}_{\Omega ,\xi}(u)\right)^{alg}\}$.}

\end{definition}
Given $\omega \in \mathbb{C}\setminus \mathbb{R}$, let $\varXi (\omega ,\xi):=\varXi ({\langle 1,\omega \rangle}_\mathbb{Z} ,\xi)$.
\begin{lemma}\label{S Theta}
Let $\Omega$ be a lattice of $\C$ and $\xi ,\xi ' \in \mathbb{C}$. 
We have:  
\begin{enumerate}
 \item[$(1)$] \emph{$\xi \sim _{\Omega} \xi '$ if and only if $\xi '\in \varXi (\Omega ,\xi)$} is an equivalence 
 relation.
 \item[$(2)$] If $\Omega '$ is a sublattice of $\Omega$, then $\varXi (\Omega ',\xi)=\varXi (\Omega ,\xi)$.
 \item[$(3)$] If $\wp _{\Omega} (\xi '\xi ^{-1}u)\in\C(\wp _{\Omega}(u))^{alg}$, 
 then $\xi '\in \varXi (\Omega ,\xi)$. In particular, $K_\omega ^*\xi\subseteq \varXi (\omega ,\xi)$, 
 for each $\omega \in \mathbb{C}\setminus \mathbb{R}$.
 \item[$(4)$] For each $q\in \mathbb{Q}^*$, $e^{qv}\widetilde{\sigma}_{\Omega ,q\xi}(u)\in\mathbb{C}\left(\wp _{\Omega}(u), e^v\widetilde{\sigma}_{\Omega,\xi}(u)\right)^{alg}$.
 \item[$(5)$] If $\xi' \in \Omega$,  then  
 $\big(\wp_{\Omega}(u),e^{cu+v}\widetilde{\sigma}_{\Omega ,\xi+\xi'}(u)\big)\in\mathbb{C}\big(\wp_{\Omega}(u),e^v\widetilde{\sigma}_{\Omega ,\xi}(u)\big)^{alg}$, for some $c\in \mathbb{C}$.
 In particular, $\Omega +\xi \subseteq \varXi(\Omega ,\xi)$.
\end{enumerate}
\end{lemma}
\begin{proof}
(2) Note that, by Lemma\,\ref{S sublattices}, for each $\rho \in \mathbb{C}$, there is $\alpha _{\rho} \in \GL_2(\mathbb{C})$ 
such that $(\wp _{\Omega}(u),e^v\widetilde{\sigma}_{\Omega,{\rho}}(u))\circ \alpha _{\rho}\in\mathbb{C}\big(\wp _{\Omega '}(u),e^v\widetilde{\sigma}_{\Omega ',{\rho}}(u)\big)^{alg}$.
In particular, $(\wp _{\Omega '}(u),e^v\widetilde{\sigma}_{\Omega ',{\rho}}(u))\circ \alpha _{\rho}^{-1}\in\mathbb{C}\big(\wp _{\Omega}(u),e^v\widetilde{\sigma}_{\Omega,{\rho}}(u)\big)^{alg}$.
For each $\xi '\in \varXi (\Omega ',\xi )$, there is $\alpha \in \GL_2(\mathbb{C})$ such that
$(\wp _{\Omega '}(u),e^v\widetilde{\sigma}_{\Omega ',\xi '}(u))\circ \alpha$ is algebraic over 
$\mathbb{C}\big(\wp _{\Omega '}(u),e^v\widetilde{\sigma}_{\Omega ',\xi}(u)\big)$.
Thus, $$(\wp _{\Omega}(u),e^v\widetilde{\sigma}_{\Omega,\xi '}(u))\circ 
\alpha _{\xi '} \circ \alpha \circ \alpha _{\xi}^{-1}\in\mathbb{C}\big(\wp _{\Omega}(u),e^v\widetilde{\sigma}_{\Omega,\xi}(u)\big)^{alg},$$
and so $\xi '\in \varXi (\Omega ,\xi)$.
Similarly, for any given $\xi '\in \varXi (\Omega,\xi)$ there exists $\beta\in \GL_2(\mathbb{C})$ such that
$(\wp _{\Omega}(u),e^v\widetilde{\sigma}_{\Omega,\xi '}(u))\circ \beta\in\mathbb{C}\big(\wp _{\Omega}(u),e^v\widetilde{\sigma}_{\Omega,\xi}(u)\big)^{alg}$.
Thus, $(\wp _{\Omega '}(u),e^v\widetilde{\sigma}_{\Omega ',\xi '}(u))
\circ \alpha _{\xi '}^{-1} \circ \beta\circ \alpha _{\xi}\in\mathbb{C}\big(\wp _{\Omega '}(u),e^v\widetilde{\sigma}_{\Omega ',\xi}(u)\big)^{alg}$, so $\xi '\in \varXi (\Omega ',\xi)$.

(3) Let $\Omega ':=\xi '^{-1} \xi \Omega$.
Since $\wp _{\Omega '}(u)=\xi '^2\xi ^{-2}\wp _{\Omega}(\xi '\xi ^{-1}u)$, we get
$\wp _{\Omega '}(u)\in\mathbb{C}(\wp _{\Omega}(u))^{alg}$.
By Lemma\,\ref{C-wp}, there exists a lattice $\Omega ''$ of $\C$ such that both $\Omega ''\leq\Omega$ and $\Omega '' \leq\Omega '$.
By Lemma\,\ref{S sublattices}, there exist $\alpha,\beta \in \GL_2(\mathbb{C})$ with 
$(\wp _{\Omega}(u), e^v\widetilde{\sigma} _{\Omega,\xi}(u))\circ \alpha\in\mathbb{C}\big(\wp _{\Omega ''}(u), e^v\widetilde{\sigma} _{\Omega '',\xi}(u)\big)^{alg}$ and 
$$(\wp _{\Omega '}(u), e^v\widetilde{\sigma} _{\Omega ',\xi}(u))\circ \beta\in\mathbb{C}\big(\wp _{\Omega ''}(u), e^v\widetilde{\sigma} _{\Omega '',\xi}(u)\big)^{alg}.$$
Let $\gamma (u,v)=(\xi '\xi ^{-1}u,v)$. Then, 
$$(\wp _{\Omega}(u), e^v\widetilde{\sigma} _{\Omega,\xi '}(u))\circ \gamma\in\mathbb{C}\big(\wp _{\Omega '}(u), e^v\widetilde{\sigma} _{\Omega ',\xi}(u)\big)^{alg}.$$ 
Thus, $(\wp _{\Omega}(u), e^v\widetilde{\sigma} _{\Omega,\xi '}(u))\circ \gamma \circ \beta\circ \alpha ^ {-1}\in\mathbb{C}\big(\wp _{\Omega}(u), e^v\widetilde{\sigma} _{\Omega,\xi}(u)\big)^{alg}$, and hence  $\xi '\in \varXi (\Omega ,\xi)$.
The additional condition follows by Proposition\,\ref{C-automorphisms}.

We will prove $(4)$ in three steps.

\noindent
 \emph{Claim $(1)$. For each lattice $\Omega$ of $\C$,
\[
\mathbb{C}\left(\wp _{\Omega}(u), e^{-v}\widetilde{\sigma}_{\Omega,-\xi}(u)\right) 
=\mathbb{C}\left(\wp _{\Omega}(u), e^v\widetilde{\sigma}_{\Omega,\xi }(u)\right).
\]
Furthermore, $\Omega \subseteq \varXi (\Omega ,0)$.}

\smallskip

\noindent
\emph{Proof of Claim $(1)$.}
We first consider the case where $\xi \notin \Omega$.
By \cite[Ch.IV, \textsection\,$3$, Cor.]{Chandrasekharan},
\[
\wp _{\Omega} (u) - \wp _{\Omega} (z) = -\frac{\sigma _{\Omega} (u+z)\sigma  _{\Omega} (u-z)}
{\sigma ^2 _{\Omega}(u) \sigma ^2  _{\Omega} (z)}.
\]
Evaluating the above expression at $z=\xi$ and solving for $\frac{\sigma _{\Omega} (u+\xi)}
{\sigma _{\Omega}(u)}$, we deduce that
\[
e^{-v}\frac{\sigma _{\Omega}(u+\xi)}{\sigma _{\Omega}(u)}=(\wp _{\Omega}(u)-\wp _{\Omega}(\xi))
\left(e^{-v}\frac{\sigma _{\Omega}(u)}{\sigma _{\Omega}(u-\xi)}\right)\sigma _{\Omega}^2(\xi).
\]
The claim follows easily from this, so we are done with this case.
Now, consider the case when $\xi \in \Omega$.
Since $\wp _{\Omega} (u-\xi )=\wp _{\Omega }(u)$, there exists $C\in \mathbb{C}$ such that 
$\zeta _{\Omega}(u-\xi )=\zeta _{\Omega}(u)+C$, so there also exists $D\in \mathbb{C}$ such that 
$\widetilde{\sigma}_{\Omega,\xi}(u)=e^{Cu+D}$. 
Since $\sigma_{\Omega} $ is an odd function, we get that 
$\widetilde{\sigma}_{\Omega,-\xi}(u)=\widetilde{\sigma}_{\Omega,\xi}(-u)=e^{-Cu+D}$
and, therefore, 
$e^{-v}\widetilde{\sigma}_{\Omega,-\xi}(u)=e^{-v}e^{-Cu+D}=e^{2D}(e^{v}e^{Cu+D})^{-1}=e^{2D}(e^{v}\widetilde{\sigma}_{\Omega,\xi}(u))^{-1}$,
which shows the first part.
For the remaining condition, let $\alpha (u,v)=(u,-Cu+v)$ and note that 
$(\wp _{\Omega}(u),e^v\widetilde{\sigma}_{\Omega,\xi }(u))\circ \alpha= (\wp _{\Omega}(u),e^{v+D})$.
Since $\alpha \in \GL_2(\mathbb{C})$, this shows that $\xi \in \varXi (\Omega, 0)$.
Consequently, $\Omega \subseteq \varXi (\Omega, 0)$. {\hfill $_\square$}

\medskip

\noindent
{\it Claim $(2)$: For each given lattice $\Omega$ of $\C$ and for each $n\in \mathbb{N}^*$, 
\[
e^{nv}\widetilde{\sigma}_{\Omega,n\xi}(u)\in 
 \mathbb{C}\left(\wp _{\Omega}(u), e^v\widetilde{\sigma}_{\Omega,\xi}(u)\right)^{alg}. 
\]
}
\emph{Proof of Claim $(2)$.}
By induction on $n$.
Suppose that $e^{(n-1)v}\widetilde{\sigma}_{\Omega,(n-1)\xi}(u)$ is 
algebraic over $\mathbb{C}\big(\wp _{\Omega}(u), e^v\widetilde{\sigma}_{\Omega,\xi}(u)\big)$, and so over $\mathbb{C}\big(\wp _{\Omega}(u), e^{-v}\widetilde{\sigma}_{\Omega,-\xi}(u)\big)$ 
by Claim $(1)$.
Changing $(u,v)$ by $(u-\xi,v)$, we get that 
$$e^{(n-1)v}\widetilde{\sigma}_{\Omega,(n-1)\xi}(u-\xi)\in\mathbb{C}\big(\wp _{\Omega}(u-\xi), e^{-v}(\widetilde{\sigma}_{\Omega,\xi}(u))^{-1}\big)^{alg}.$$
By Fact\,\ref{Lemma 3}, $\wp _{\Omega}(u-\xi)\in\mathbb{C}\big(\wp _{\Omega}(u)\big)^{alg}$, so $e^{(n-1)v}\widetilde{\sigma}_{\Omega,(n-1)\xi}(u-\xi)$ 
is algebraic over $\mathbb{C}\big(\wp _{\Omega}(u), e^v\widetilde{\sigma}_{\Omega,\xi}(u)\big)$.
Multiplying by $e^v\widetilde{\sigma}_{\Omega,\xi}(u)$ we conclude that $e^{nv}\widetilde{\sigma}_{\Omega,n\xi}(u)\in\mathbb{C}\big(\wp _{\Omega}(u), e^v\widetilde{\sigma}_{\Omega,\xi}(u)\big)^{alg}$. {\hfill $_{\square}$}

\medskip

Let us finish the proof of $(4)$.
Let $q=a/b\in \mathbb{Q}^*$.
By Claim (1) and (2) $e^{av}\widetilde{\sigma}_{b\Omega,a\xi}(u)$  and $e^{bv}\widetilde{\sigma}_{b\Omega,b\xi}(u)$
belong to 
$\mathbb{C}(\wp _{b\Omega}(u), e^v\widetilde{\sigma}_{b\Omega,\xi}(u))^{alg}$.
Therefore, $(\wp _{b\Omega}(u), e^{av}\widetilde{\sigma}_{b\Omega,a\xi}(u))\in\mathbb{C}(\wp _{b\Omega}(u), e^{bv}\widetilde{\sigma}_{b\Omega,b\xi}(u))^{alg}$.
We conclude by  a change of variables ($bu$ for $u$ and $b^{-1}v$ for $v$).

Finally, we prove $(5)$. 
First note that
\[
\widetilde{\sigma}_{\Omega,\xi+\xi'}(u)=\frac{\sigma _{\Omega} (u-\xi-\xi')}{\sigma_{\Omega}(u-\xi')}\frac{\sigma _{\Omega}(u-\xi')}{\sigma _{\Omega}(u)}=
\widetilde{\sigma}_{\Omega ,\xi}(u-\xi')\widetilde{\sigma}_{\Omega ,\xi'}(u).
\]
Since $\xi'\in \Omega$, it follows from the proof of Claim\,(1) of Lemma\,\ref{S Theta}.$(4)$ that 
$\widetilde{\sigma}_{\Omega,\xi'}(u)=e^{Cu+D}$, for some $C,D\in \mathbb{C}$, and, therefore,
\[
\big(\wp_{\Omega} (u),e^{-D-Cu+v}\widetilde{\sigma}_{\Omega ,\xi+\xi'}(u)\big)
=
\big(\wp_{\Omega} (u-\xi'),e^v\widetilde{\sigma}_{\Omega ,\xi}(u-\xi')\big).
\]
Now, by Fact\,\ref{Lemma 3}, we deduce that 
$$
\big(\wp_{\Omega} (u-\xi'),e^v\widetilde{\sigma}_{\Omega ,\xi}(u-\xi')\big)\in
\mathbb{C}\big(\wp_{\Omega} (u),e^v\widetilde{\sigma}_{\Omega ,\xi}(u)\big)^{alg}.
$$
Hence, taking $c:=-C$ and noting that $e^{-D}$ is a constant in $\mathbb{C}$, we conclude that
$
\big(\wp_{\Omega}(u),e^{cu+v}\widetilde{\sigma}_{\Omega ,\xi+\xi'}(u)\big)\in
\mathbb{C}\big(\wp_{\Omega}(u),e^v\widetilde{\sigma}_{\Omega ,\xi}(u)\big)^{alg},
$
as requested.
\end{proof}

The next lemma  corresponds to the  claim in the proof of  Proposition\,\ref{Z general}, for groups associated to the  family $\mathcal{P}_5$ with the same lattice.

\begin{lemma}\label{representation}
Let $\Omega$ be a lattice of $\C$, $\xi _1,\xi _2\in \mathbb{C}$ and $\alpha (u,v):=(au+bv,cu+dv)\in \GL_2(\mathbb{C})$.
If $(\wp _{\Omega}(u), e^v\widetilde{\sigma}_{\Omega,\xi _2}(u)) \circ \alpha\in\mathbb{C}(\wp _{\Omega}(u), e^v\widetilde{\sigma}_{\Omega,\xi _1}(u))^{alg}$, then $b=0$,
$\wp _{\Omega} (au)\in\C(\wp _{\Omega}(u))^{alg}$,  $d\in \mathbb{Q}^*$ and $a^{-1}\xi _2-q\xi _1\in \Omega$,  
for some $q\in \mathbb{Q}^*$. 
If in addition $a=1$ then $\xi _2-d\xi _1\in \Omega$.
\end{lemma}
\begin{proof}We first proof the following:

\medskip

\noindent
{\it Claim. Let $\Omega _1$ and $\Omega _2$ be lattices of $\mathbb{C}$, $\xi _1,\xi _2\in \mathbb{C}$ and
$\alpha (u,v) =(au+bv,cu+dv)\in \GL_2(\mathbb{C})$. If $\big(\wp _{\Omega _2}(u), e^v\widetilde{\sigma} _{\Omega_2,\xi _2} (u)\big)\circ \alpha\in\mathbb{C}\big(\wp _{\Omega _1}(u), e^v\widetilde{\sigma}_{\Omega_1,\xi _1} (u)\big)^{alg}$ then $b=0$, $\wp _{\Omega _2}(au)\in\mathbb{C}(\wp _{\Omega _1}(u))^{alg}$  and $d\in \Q^*$.}
   
\smallskip

\noindent
\emph{Proof of the claim.} We begin by arguing as in the proof of the claim in  Lemma\,\ref{Z general}, and  get $b=0$  and   $\wp _{\Omega_2}(au)\in\mathbb{C}(\wp _{\Omega_1}(u))^{alg}$. Now, note that $e^{cu+dv}\widetilde{\sigma}_{\Omega_2,\xi_2}(au)\in\mathbb{C}(\wp _{\Omega_1}(u),
e^v\widetilde{\sigma}_{\Omega_1,\xi_1} (u))^{alg}$.
Thus, if we evaluate $u$ in a point that does not belong to $a^{-1}\Omega_2\cup \Omega_1$, we conclude that 
$e^{dv}\in\mathbb{C}(e^v)^{alg}$, so $d\in \mathbb{Q}^*$ as required. {\hfill $_{\square}$}

\medskip

By the above claim, we have that $\wp _{a^{-1}\Omega}(u)\in \mathbb{C}(\wp _{\Omega}(au))$, so 
it is algebraic over $\mathbb{C}(\wp _{\Omega}(u))$. 
Thus, by Lemma\,\ref{C-wp}, there exists a lattice $\Omega '\leq\mathbb{C}$ such that $\Omega '\leq\Omega$ and 
$\Omega '\leq a^{-1}\Omega$.

By hypothesis, as usual, we get, $$(\wp_{a^{-1}\Omega}(u),e^{cu+dv}\widetilde{\sigma}_{a^{-1}\Omega,a^{-1}\xi_2}(u))\in\mathbb{C}(\wp _{\Omega}(u), e^v\widetilde{\sigma}_{\Omega,\xi _1}(u))^{alg}.$$ 
Our aim now is to derive from the latter an algebraic relation that is uniform with respect to $\Omega'$.

On one hand, defining $\beta_1(u,v):=(u,a^{-1}\xi_2 \mathfrak{c}(a^{-1}\Omega,\Omega')u+[a^{-1}\Omega:\Omega']v)$, 
we get, by Lemma\,\ref{S sublattices}, that $\big(\wp _{a^{-1}\Omega}(u), e^{cu+dv}\widetilde{\sigma}_{a^{-1}\Omega,a^{-1}\xi_2}(u)\big)\circ \beta_1$ 
is algebraic over $\mathbb{C}(\wp _{\Omega'}(u), e^{cu+dv}\widetilde{\sigma}_{\Omega',a^{-1}\xi_2}(u))$. 
On the other hand, if we define $\beta_2(u,v):=(u,\xi_1 \mathfrak{c}(\Omega,\Omega')u+[\Omega:\Omega']v)$ then, by Lemma\,\ref{S sublattices}, 
$$\big(\wp _{\Omega}(u), e^v\widetilde{\sigma}_{\Omega,\xi _1}(u)\big)\circ \beta_2\in \mathbb{C}(\wp_{\Omega'}(u), e^v\widetilde{\sigma}_{\Omega',\xi _1}(u))^{alg}.$$ 
All in all, we deduce that 
$$
(\wp _{\Omega'}(u), e^{cu+dv}\widetilde{\sigma}_{\Omega',a^{-1}\xi_2}(u))\circ \beta_1^{-1} \circ \beta_2 
\in\mathbb{C}(\wp _{\Omega'}(u), e^v\widetilde{\sigma}_{\Omega',\xi _1}(u))^{alg}.$$ 
Making some computations, we deduce that
\[
e^{c'u+d'v}\widetilde{\sigma}_{\Omega',a^{-1}\xi_2}(u)\in\mathbb{C}(\wp _{\Omega'}(u), e^v\widetilde{\sigma}_{\Omega',\xi _1}(u))^{alg},
\]
where 
\[
c':=c+d\frac{\xi_1 \mathfrak{c}(\Omega,\Omega')-a^{-1}\xi_2 \mathfrak{c}(a^{-1}\Omega,\Omega')}{[a^{-1}\Omega:\Omega']}  \quad 
\text{and} \quad d':=d\frac{[\Omega:\Omega']}{[a^{-1}\Omega:\Omega']}\in \mathbb{Q}^*. 
\]

Next, by Lemma\,\ref{S Theta}.$(4)$, $e^{-d'v}\widetilde{\sigma}_{\Omega ',-d'\xi _1}(u)\in\mathbb{C}(\wp _{\Omega '}(u), e^v\widetilde{\sigma}_{\Omega',\xi _1}(u))^{alg}$  since $d'\in \mathbb{Q}^*$. 
Thus, taking the product by the above function, we deduce that 
$e^{c'u}\widetilde{\sigma}_{\Omega',a^{-1}\xi_2}(au)\widetilde{\sigma}_{\Omega ',-d'\xi _1}(u)$ is algebraic over 
$\mathbb{C}(\wp _{\Omega'}(u), e^v\widetilde{\sigma}_{\Omega',\xi _1}(u))$ and, in particular, over $\mathbb{C}(\wp _{\Omega '}(u))$.

By \cite[Ch.\,IV, Thm.\,$4$]{Chandrasekharan}, the function
\[
\widetilde{\sigma}_{\Omega ',a^{-1}\xi _2}(u)\cdot \widetilde{\sigma}_{\Omega ',-d'\xi _1}(u)\cdot 
\widetilde{\sigma}_{\Omega ',-a^{-1}\xi _2+d'\xi _1}(u) 
\] 
is elliptic of period $\Omega'$, so it is algebraic over $\mathbb{C}(\wp_{\Omega'}(u))$. 
In particular,
$f(u):=e^{-c'u}\widetilde{\sigma}_{\Omega ',-a^{-1}\xi _2+d'\xi _1}(u)$ is algebraic over $\mathbb{C}(\wp _{\Omega '}(u))$, 
which means that $f(u)$ is an elliptic function. 
As, it has a unique simple pole in its fundamental domain (see \cite[Ch.IV, \S\,2]{Chandrasekharan}), so
$f(u)$ must be constant (\cite[Ch.,II, Cor. to Thm\,$2$]{Chandrasekharan}).
Thus, there exists $C\in \mathbb{C}$ such that $\widetilde{\sigma}_{\Omega',-a^{-1}\xi _2+d'\xi _1}(u)=e^{C+c'u}$.
Consequently, 
\[
\ln\big(\sigma_{\Omega'}(u-(-a^{-1}\xi _2+d'\xi _1))\big)-\ln(\sigma _{\Omega '}(u))=C+c'u.
\]
Differentiating twice, we get 
\[
\wp_{\Omega'}(u-(-a^{-1}\xi _2+d'\xi _1))-\wp _{\Omega '}(u)=0,
\]
so $-a^{-1}\xi _2+d'\xi _1\in \Omega '\leq \Omega$, so it is enough to define $q:=d'$, as required. 
Finally, note that in the particular case $a=1$ we get that $d'=d$.
\end{proof} 

The two latter  lemmas give us a characterisation of the sets $\varXi $ as follows.
\begin{corollary}\label{S characterization}
Let $\omega \in \mathbb{C}\setminus \mathbb{R}$. 
Then, $\varXi (\omega ,\xi )={\langle 1,\omega \rangle}_\mathbb{Q}+ K_\omega ^*\xi$, for each $\xi\in \mathbb{C}$.
\end{corollary}
\begin{proof}
Fix $\xi \in \mathbb{C}$ and let $\Omega:={\langle 1,\omega\rangle}_{\mathbb{Z}}$. 
Firstly, let us show that $\varXi (\Omega, \xi)\subseteq \, {\langle 1,\omega\rangle}_{\mathbb{Q}}+K_\omega ^* \xi$. 
Indeed, take $\xi'\in \varXi (\Omega, \xi)$.
By Lemma\,\ref{representation}, there exist $a\in K_\omega ^*$ and $q\in \mathbb{Q} ^*$ such that $a^{-1}\xi'-q\xi\in \Omega \subseteq {\langle 1,\omega\rangle}_{\mathbb{Q}}$.
Note that  $K_\omega^*$ is contained in ${\langle 1,\omega\rangle}^*_{\mathbb{Q}}$ and, therefore, $\xi'\in {\langle 1,\omega \rangle}_\mathbb{Q}+ K_\omega ^*\xi$, as required.

For the other inclusion, let us see first the following:

\medskip

\noindent
{\it Claim. $\lambda+\xi \in \varXi (\omega ,\xi)$, for all $\lambda \in \, {\langle 1,\omega \rangle}_{\mathbb{Q}}$.}

\smallskip

\noindent
\emph{Proof of the claim.}
There exists $n\in \mathbb{N}^*$ such that $n\lambda\in  \, {\langle 1,\omega \rangle}_{\mathbb{Z}}$.
In particular, $\lambda\in n^{-1}\Omega$.
By Lemma\,\ref{S Theta}(5), $\lambda+\xi \in \varXi (n^{-1}\Omega ,\xi)$.
Since $\Omega \leq n^{-1}\Omega$, by Lemma\,\ref{S Theta}(2), we get $\lambda+\xi \in \varXi (\Omega ,\xi)$.{\hfill $_{\square}$}

\medskip

\noindent Let $\lambda\in \, {\langle 1,\omega \rangle}_{\mathbb{Q}}$ and $q\in K_\omega ^*$, we show now that 
$\lambda+q\xi \in \varXi (\omega ,\xi)$. 
Since $q\in K_\omega ^*$, there exists by Fact\,\ref{Chandrasekharan} a lattice $\Omega '$ that is both a sublattice 
of $\Omega$ and $q^{-1}\Omega$. 
Then, by Lemma\,\ref{S sublattices}, there exists $\alpha_1\in \GL_2(\mathbb{C})$ such that
\[
\big(\wp _{\Omega}(u),e^v\widetilde{\sigma}_{\Omega ,\xi}(u)\big)\circ \alpha _1\in\mathbb{C}\big(\wp _{\Omega '}(u),e^v\widetilde{\sigma}_{\Omega ',\xi}(u)\big)^{alg}.
\]
Similarly, there exists $\alpha_2\in \GL_2(\mathbb{C})$ such that
\[
\big(\wp _{q^{-1}\Omega}(u),e^v\widetilde{\sigma}_{q^{-1}\Omega ,\xi}(u)\big)\circ \alpha _2\in\mathbb{C}\big(\wp _{\Omega '}(u),e^v\widetilde{\sigma}_{\Omega ',\xi}(u)\big)^{alg}. 
\]
Let $\beta _2:=\alpha _2\circ \alpha _1^{-1}$.Thus,
$$\big(\wp _{q^{-1}\Omega}(u),e^v\widetilde{\sigma}_{q^{-1}\Omega ,\xi}(u)\big)\circ \beta _2\in\mathbb{C}\big(\wp _{\Omega}(u),e^v\widetilde{\sigma}_{\Omega,\xi}(u)\big)^{alg}.$$
Moreover, if we define $\alpha _3(u,v):=(qu,v)$ and $\beta _3(u,v):=\alpha _3\circ \beta _2$, then
$
\big(\wp _{\Omega}(u),e^v\widetilde{\sigma}_{\Omega ,q\xi}(u)\big)\circ \beta _3\in\mathbb{C}\big(\wp _{\Omega}(u),e^v\widetilde{\sigma}_{\Omega,\xi}(u)\big)^{alg}.
$
By the claim above, there exists $\alpha _4\in \GL_2(\mathbb{C})$ such that 
\[
\big(\wp _{\Omega}(u),e^v\widetilde{\sigma}_{\Omega ,\lambda+q\xi}(u)\big)\circ \alpha _4\in\mathbb{C}\big(\wp _{\Omega}(u),e^v\widetilde{\sigma}_{\Omega,q\xi}(u)\big)^{alg}.
\]
Thus, if we denote $\beta _4(u,v):=\alpha _4\circ \beta _3$, then $\big(\wp _{\Omega}(u),e^v\widetilde{\sigma}_{\Omega ,q\xi+\lambda}(u)\big)\circ \beta _4$ 
belongs to  $\mathbb{C}\big(\wp _{\Omega}(u),e^v\widetilde{\sigma}_{\Omega,\xi}(u)\big)^{alg}$, as required.
\end{proof}

Finally, we  can characterise algebraic dependence for maps of the  family $\mathcal{P}_5$.

\begin{proposition}\label{S general} 
Let $\omega _1,\omega _2 \in \mathbb{C}\setminus \mathbb{R}$, $\xi _1,\xi _2\in \mathbb{C}$. Then,  there exists $\alpha \in \GL_2(\mathbb{C})$ such that 
$\big(\wp _{\omega _2}(u), e^v\widetilde{\sigma}_{\omega _2,\xi_2}(u)\big)\circ \alpha\in\mathbb{C}\big(\wp _{\omega _1}(u), e^v\widetilde{\sigma}_{\omega _1,\xi_1}(u)\big)^{alg}$  if and only if there are $a,b,c,d\in \mathbb{Z}$ such that $ad-bc\neq 0$, $\omega _2=\frac{a\omega _1+b}{c\omega _1+d}$ and $(c\omega _1+d)\xi _2\in \varXi (\omega _1,\xi _1)$. 
\end{proposition}
\begin{proof}
Let  $\Omega _i:={\langle 1,\omega _i\rangle}_{\mathbb{Z}}, i=1,2$.
We begin with the right to left implication.
Let $\Omega :={\langle c\omega _1+d,a\omega _1+b\rangle}_{\mathbb{Z}}\leq \Omega_1$ and $\alpha _1(u,v):=((c\omega _1+d)u,v)$.
By Lemma\,\ref{S parameter}, $\big(\wp _{\Omega}(u), e^v\widetilde{\sigma}_{\Omega, (c\omega _1+d)\xi _2}(u)\big)\circ \alpha _1\in\mathbb{C}\big(\wp _{\Omega _2}(u), e^v\widetilde{\sigma}_{\Omega _2,\xi_2}(u)\big)^{alg}$.
Let $\alpha _2(u,v):=\big(u,\xi _1\mathfrak{c}(\Omega _1,\Omega)u+[\Omega _1:\Omega]v\big)$.
By Lemma\,\ref{S sublattices},
$$\big(\wp _{\Omega _1}(u), e^v\widetilde{\sigma}_{\Omega _1,\xi_1}(u)\big)\circ \alpha _2\in\mathbb{C}\big(\wp _{\Omega}(u), e^v\widetilde{\sigma}_{\Omega,\xi_1}(u)\big)^{alg}.$$
By Lemma\,\ref{S Theta}.$(2)$, $(c\omega _1+d)\xi _2\in \varXi (\Omega,\xi _1)$.
So there exists $\alpha _3 \in \GL_2(\mathbb{C})$, such that
$\big(\wp _{\Omega}(u), e^v\widetilde{\sigma}_{\Omega,(c\omega _1+d)\xi _2}(u)\big)\circ \alpha _3$
is algebraic over $\mathbb{C}\big(\wp _{\Omega}(u), e^v\widetilde{\sigma}_{\Omega,\xi_1}(u)\big)$.
So $\big(\wp _{\Omega _1}(u), e^v\widetilde{\sigma}_{\Omega _1,\xi_1}(u)\big)
\circ \alpha _2\circ \alpha _3^{-1}\circ \alpha _1\in\mathbb{C}\big(\wp _{\Omega _2}(u), e^v\widetilde{\sigma}_{\Omega _2,\xi_2}(u)\big)^{alg}$.
Define $\alpha :=\alpha_1^{-1}\circ \alpha_3 \circ \alpha_2^{-1}$ and we are done.

Now, we prove the converse. 
By hypothesis and the claim in the proof of Lemma\,\ref{representation}, we have $\alpha(u,v)=(a'u,c'u+d'v)$, for some 
$a',c',d'\in \mathbb{C}$ such that $\wp_{\Omega _2}(a'u)\in\mathbb{C}(\wp_{\Omega_1}(u))^{alg}$. 
Thus, by Theorem\,\ref{1dim C-classification} there exist 
$a,b,c,d\in \mathbb{Z}$ with $ad-bc\neq 0$ such that $\omega_2=\frac{a\omega _1+b}{c\omega _1+d}$ and 
$a'=\frac{n}{c\omega _1+d}$, for some $n\in \mathbb{N}^*$.

On the other hand, by Fact\,\ref{Chandrasekharan} and Lemma\,\ref{C-wp}, there exists a lattice $\Omega$ of 
$\C$ such that both $\Omega \leq  \Omega_1$ and $\Omega\leq  a'^{-1}\Omega_2$.
By hypothesis, \begin{center}
$\big(\wp _{\Omega _2}(a'u), e^{c'u+d'v}\widetilde{\sigma}_{\Omega _2,\xi_2}(a'u)\big)\in\mathbb{C}\big(\wp _{\Omega _1}(u), e^v\widetilde{\sigma}_{\Omega _1,\xi_1}(u)\big)^{alg}$, 
\end{center}
so, also $\big(\wp _{a'^{-1}\Omega _2}(u), 
e^{c'u+d'v}\widetilde{\sigma}_{a'^{-1}\Omega _2,a'^{-1}\xi _2}(u)\big)\in\mathbb{C}\big(\wp _{\Omega _1}(u), e^v\widetilde{\sigma}_{\Omega _1,\xi_1}(u)\big)^{alg}$.
By Lemma\,\ref{S sublattices}, there exists $\alpha _1\in \GL_2(\mathbb{C})$ such that
$$\big(\wp _{\Omega _1}(u), e^v\widetilde{\sigma}_{\Omega _1,\xi_1}(u)\big)\circ \alpha _1\in\mathbb{C}\big(\wp _{\Omega}(u), e^v\widetilde{\sigma}_{\Omega,\xi_1}(u)\big)^{alg}.$$

Similarly, pick $\alpha _2\in \GL_2(\mathbb{C})$ such that
$$\big(\wp _{a'^{-1}\Omega _2}(u), e^{c'u+d'v}\widetilde{\sigma}_{a'^{-1}\Omega _2,a'^{-1}\xi _2}(u)\big)\circ \alpha _2\in\mathbb{C}\big(\wp _{\Omega}(u), e^{c'u+d'v}\widetilde{\sigma}_{\Omega,a'^{-1}\xi _2}(u)\big)^{alg}.$$
Then,  $\big(\wp _{\Omega}(u), e^{c'u+d'v}\widetilde{\sigma}_{\Omega,a'^{-1}\xi _2}(u)\big)
\circ \alpha _2^{-1}\circ \alpha _1\in\mathbb{C}\big(\wp _{\Omega}(u), e^v\widetilde{\sigma}_{\Omega,\xi_1}(u)\big)^{alg}$,
 and so $a'^{-1}\xi _2\in \varXi (\Omega ,\xi _1)$.
Thus, by Lemma\,\ref{S Theta}.$(1)$ and\,\ref{S Theta}.$(2)$, we have that 
\[
\varXi(\omega_1,\xi_1)=\varXi(\Omega,\xi_1)=\varXi(\Omega,a'^{-1}\xi_2)=\varXi(\omega_1,a'^{-1}\xi_2).  
\]
Finally, it follows from Lemma\,\ref{S Theta}.$(3)$ that 
\[
(cw+d)\xi_2=na'^{-1}\xi_2\in \varXi(\omega_1,a'^{-1}\xi_2)=\varXi(\omega_1,\xi_1),  
\]
as required.
\end{proof}

\begin{corollary}\label{S general: specialcase}
Let $\omega \in \mathbb{C}\setminus \mathbb{R}$, $\Omega ={\langle 1,\omega\rangle}_{\mathbb{Z}}$ and $\xi \in \mathbb{C}$.
For each $q\in \mathbb{Z}^*$, $\big(\wp _{\omega}(u), e^v\widetilde{\sigma}_{\omega,\xi}(u)\big)\circ \alpha\in\mathbb{C}\big(\wp _{\omega }(u), e^v\widetilde{\sigma}_{\omega,\xi}(u)\big)^{alg}$, where 
$\alpha (u,v)=q^{-1}\big(u,-\xi \mathfrak{c}(\Omega,q\Omega)u+v\big)$.
\end{corollary}
\begin{proof}
Fix $q\in \mathbb{Z}^*$. Since by Lemma\,\ref{S Theta}(4) we have that $q\xi\in \varXi(\Omega,\xi)$, we can apply the right to left implication of Proposition\,\ref{S general} for $b=c=0$ and $a=d=q$. Thus, there is $\alpha \in \GL_2(\mathbb{C})$ such that 
$\big(\wp _{\omega _2}(u), e^v\widetilde{\sigma}_{\omega _2,\xi_2}(u)\big)\circ \alpha\in\mathbb{C}\big(\wp _{\omega _1}(u), e^v\widetilde{\sigma}_{\omega _1,\xi_1}(u)\big)^{alg}$. It only remains to explicitly compute  the map $\alpha$ -- provided by the proof of the proposition -- in our case. Following the notation there, we obtain
$\alpha _1(u,v)=(qu,v)$ and $\alpha _2(u,v)=\big(u,\xi\mathfrak{c}(\Omega,q\Omega)u+[\Omega:q\Omega]v\big)$. Also, by Lemma\,\ref{S Theta}.$(4)$, we can choose $\alpha_3(u,v)=(u,qv)$. 
Finally, since $[\Omega:q\Omega]=q^2$, we get that 
\[
\alpha (u,v) = \alpha _1^{-1}\circ \alpha _3\circ \alpha _2^{-1}(u,v)=q^{-1}\big(u,-\xi \mathfrak{c}(\Omega,q\Omega)u+v \big),  
\]
as required.
\end{proof}

We finish this section with notions that extend those of index and residue. We will make use of them to describe the automorphisms groups of our locally $\mathbb{K}$-Nash groups.
Recall that the constant $\mathfrak{c}$ was introduced in Definition\,\ref{defc}.

\begin{defnprop}\label{defci}Let $\Omega _1$ and $\Omega _2$ be two lattices of $\C$ that have $\Omega$ as a common sublattice.
We define 
\[
 [\Omega _2:\Omega _1]:= \frac{[\Omega _2:\Omega]}{[\Omega _1:\Omega]} \qquad \text{and} \qquad
 \qc(\Omega _2,\Omega _1):= \mathfrak{c}(\Omega_2,\Omega )- \frac{[\Omega _2:\Omega]}{[\Omega _1:\Omega]}\mathfrak{c}(\Omega_1,\Omega ).
\] 
The above definitions do not depend on the chosen sublattice $\Omega$.
\end{defnprop}
\begin{proof}
Let $\alpha_i (u,v):= (u, \mathfrak{c}(\Omega_i,\Omega)u+[\Omega _i:\Omega]v)$, $i=1,2$, 
 and $\alpha :=\alpha _2\circ \alpha _1^{-1}$ then, by Lemma\,\ref{S sublattices},
$\big(\wp _{\Omega _2}(u),v-\zeta _{\Omega _2}(u)\big)\circ \alpha\in\mathbb{C}\big( \wp _{\Omega_1}(u), v-\zeta _{\Omega_1}(u)\big)^{alg}$.
Note that
\[
 \alpha (u,v)= \left( u, \left(\mathfrak{c}(\Omega_2,\Omega)- \frac{[\Omega _2:\Omega]}{[\Omega _1:\Omega]}\mathfrak{c}(\Omega_1,\Omega)\right) u + \frac{[\Omega _2:\Omega]}{[\Omega _1:\Omega]} v\right).
\]
Suppose now that $\Omega'$ is another common sublattice of both $\Omega_1$ and $\Omega_2$. 
Then, $\big(\wp _{\Omega _2}(u),v-\zeta _{\Omega _2}(u)\big)\circ \alpha'\in\mathbb{C}\big( \wp _{\Omega_1}(u), v-\zeta _{\Omega_1}(u)\big)^{alg}$ for
\[
 \alpha'(u,v)= \left(u, \left(\mathfrak{c}(\Omega_2,\Omega')- \frac{[\Omega _2:\Omega']}{[\Omega _1:\Omega']}\mathfrak{c}(\Omega_1,\Omega')\right) u + \frac{[\Omega _2:\Omega']}{[\Omega _1:\Omega']} v\right).
\]
In particular, $\alpha^{-1} \circ \alpha'$ is an automorphism of $\big(\mathbb{C}^2,+,\big(\wp _{\Omega_1}(u),v-\zeta _{\Omega_1}(u)\big)$. 
To finish the proof, it is enough to show the following:

\medskip

\noindent
\emph{Claim. Let $\Omega$ be a lattice of $\C$ and let $\beta(u,v):=(u,cu+dv)\in \GL_2(\mathbb{C})$ be an automorphism of $\big(\mathbb{C}^2,+,\big(\wp _{\Omega}(u),v-\zeta _{\Omega}(u)\big)$. 
Then, $\beta$ is the identity.}

\smallskip

\noindent
\emph{Proof of the claim.}  As usual, we can assume $\Omega=\langle 1, \omega\rangle$, for some $\omega\in\C\setminus\R$. By hypothesis  $cu+dv-\zeta_{\omega}(u)\in\mathbb{C}(\wp_{\omega}(u),v-\zeta _{\omega}(u))^{alg}$. 
 Now, if we take $v=\zeta _{\omega}(u)$, we obtain that $f(u):=cu+(d-1)\zeta _{\omega}(u)\in\mathbb{C}(\wp_{\omega}(u))^{alg}$. By Fact\,\ref{Weierstrass algebraicity} such $f$ must be 0, in particular $c=0$, as required.
\end{proof}

\begin{lemma}\label{qc}Let $\Omega$ and $\Omega'$ be lattices of $\C$ with a common sublattice, and let $a\in \mathbb{C}^*$. Then,
\begin{enumerate}
 \item[$(1)$] $\qc(\Omega,\Omega')=a^2\qc(a\Omega,a\Omega')$.
 \item[$(2)$] $\qc(\Omega,\frac{m}{n}\Omega)=\mathfrak{c}(\Omega,m\Omega)-\mathfrak{c}(\Omega,n\Omega)$ for all $m,n\in \mathbb{Z}^*$.
 \item[$(3)$] If $\Omega$ and $\Omega'$ are invariant then  $\qc(\Omega,\Omega')\in \mathbb{R}$.
\end{enumerate} 
\end{lemma}
\begin{proof}$(1)$ Let $\Omega''$ be a common sublattice of $\Omega$ and $\Omega'$.
Take $b_1:=0,b_2,\ldots,b_{n}$ representatives of the cosets of the quotient of $\Omega$ by $\Omega''$. 
Then, from the definition of $\mathfrak{c}$, it follows that $\mathfrak{c}(\Omega,\Omega'')=\sum_{i=2}^{n} \wp_{\Omega'}(b_i)$. 
Since $0,ab_2,\ldots,ab_{n}$ are representatives of the cosets of the quotient of $a\Omega$ by $a\Omega''$, we obtain
\[
\mathfrak{c}(a\Omega,a\Omega'')=\sum_{i=2}^{n} \wp_{a\Omega''}(ab_i)= a^{-2}\sum_{i=2}^{n}\wp_{\Omega''}(b_i)=a^{-2}\mathfrak{c}(\Omega,\Omega'').
\]
Similarly, $\mathfrak{c}(a\Omega',a\Omega'')=a^{-2}\mathfrak{c}(\Omega',\Omega'')$.
Therefore,
\[
\qc(a\Omega,a\Omega')=\mathfrak{c}(a\Omega,a\Omega'')-\frac{[a\Omega:a\Omega'']}{[a\Omega':a\Omega'']}\mathfrak{c}(a\Omega ',a\Omega'')=a^{-2}\qc(\Omega,\Omega'). 
\]
$(2)$ Note that $m\Omega$ is a common sublattice of $\Omega$ and $\frac{m}{n}\Omega$.
By definition and property $(2)$, we deduce
\begin{equation*}
\begin{split}
\qc(\Omega,\tfrac{m}{n}\Omega)= & \mathfrak{c}(\Omega,m\Omega)-\frac{[\Omega:m\Omega]}{[\frac{m}{n}\Omega:m\Omega]}\mathfrak{c}(\tfrac{m}{n}\Omega,m\Omega)=\\
= & \mathfrak{c}(\Omega,m\Omega)-\tfrac{m^2}{n^2}\big(\tfrac{m}{n}\big)^{-2}\mathfrak{c}(\Omega,n\Omega), 
\end{split} 
\end{equation*}
as required. 

\noindent (3) We first assume that $\Omega'\leq \Omega$. We show that there are $a_1,\ldots,a_n\in \Omega$ such that, 
for each $i\in \{1,\ldots,n\}$ there exists $j\in \{1,\ldots ,n\}$ satisfying $\Omega'+a_i=\Omega'+\overline{a_j}$.
Indeed, let $\Omega=\bigcup _{i=1}^n \Omega'+b_i$.
If $\Omega'+\overline{b_1}=\Omega'+b_1$ then we take $a_1:=b_1$.
Otherwise, there exists $j\in \{1,\ldots ,n\}$ such that $\Omega'+\overline{b_1}=\Omega'+b_j$, so we define $a_1:=b_i$ and $a_j:=b_1$.
The claim is proved repeating this process.

Now, $0=\qc(\Omega,\Omega')+\sum _{i=1}^n\wp _{\Omega'}(u+a_i)-\wp _{\Omega}(u)$.
By our choice of $a_1,\ldots ,a_n$ and applying Lemma\,\ref{wp_conjugate} we get
\[
0=\overline{\qc(\Omega,\Omega')}+\sum _{i=1}^n\wp _{\Omega'}(u+\overline{a_i})-\wp _{\Omega}(u)=\overline{\qc(\Omega,\Omega')}+\sum _{i=1}^n\wp _{\Omega'}(u+a_i)-\wp _{\Omega}(u),
\]
and therefore $\qc(\Omega,\Omega')\in \mathbb{R}$, as required.

Finally, if $\Omega'$ and $\Omega$ have a common sublattice $\Omega''$, then  we can assume that $\Omega''$ is invariant by Lemma\,\ref{C-wp}$(1)$ and Fact\,\ref{BDOrank}(1). Therefore $\qc(\Omega,\Omega')\in \R$ by definition and the previous case.\end{proof}

\section{Two-dimensional complex case}\label{Sproofs}

Now we have all the ingredients to prove the classification of the  two-dimensional simply connected abelian locally $\C$-Nash groups. We will use  the notation introduce in pages \pageref{Z-rank} and \pageref{Weierstrass algebraicity}.

\begin{theorem}\label{2dim C-classification}
$(I)$ Every two-dimensional simply connected abelian locally $\mathbb{C}$-Nash group is isomorphic
to a group of one and only one of the following types:

\smallskip

\noindent\emph{(1c)} A direct product of one-dimensional locally $\mathbb{C}$-Nash groups.
 
\noindent\emph{(2c)} $(\mathbb{C}^2,+,(\wp _\omega(u), v-\zeta _\omega (u))$, for some 
 $\omega\in \mathbb{C}\setminus \mathbb{R}$.
 
\noindent\emph{(3c)} $(\mathbb{C}^2,+,(\wp _\omega(u), e^v\widetilde{\sigma} _{\omega,\xi} (u)))$, 
 for some $\omega\in \mathbb{C}\setminus \mathbb{R}$ and $\xi \notin {\langle 1,\omega \rangle}_{\mathbb{Q}}$.
 
\noindent\emph{(4c)} Universal covering groups of complex abelian surfaces which are not direct products of elliptic curves.

\smallskip

$(II)$ The isomorphism classes are as follows:

\smallskip
\noindent\emph{(i)} Two direct products are isomorphic if and only if their factor groups are isomorphic.

\noindent\emph{(ii)} Two groups of type \emph{(2c)} defined by data $\omega _1$ and $\omega _2$, respectively, 
are isomorphic if and only if there exist $a,b,c,d\in \mathbb{Z}$ with $ad-bc\neq 0$ such that 
$\omega _2=\frac{a\omega _1+b}{c\omega _1+d}$.

\noindent\emph{(iii)} Two groups of type \emph{(3c)} defined by data $(\omega _1,\xi _1)$ and $(\omega _2,\xi _2)$, 
respectively, are isomorphic if and only if there exist $a,b,c,d\in \mathbb{Z}$ with $ad-bc\neq 0$ such that 
$\omega _2=\frac{a\omega _1+b}{c\omega _1+d}$ and $(c\omega _1+d)\xi _2 \in {\langle 1,\omega _1 \rangle}_{\mathbb{Q}}+\xi _1 K_{\omega _1}^*$. 

\noindent\emph{(iv)} Two groups that both are universal covering groups of abelian surfaces are 
isomorphic if and only if the corresponding abelian surfaces are isogenous as abelian varieties.

\end{theorem}

\begin{proof}
(I) Firstly note that, as mentioned after Fact\,\ref{TP},  the groups that appear  in the statement are indeed  locally $\C$-Nash groups. On the other hand, any two-dimensional simply connected abelian locally
$\mathbb{C}$-Nash group is isomorphic to some $(\mathbb{C}^2,+,f)$, where $f$ is a meromorphic map satisfying AAT and condition (*) of page \pageref{Z-rank}. In particular, $f$ is functionally independent so we can apply Fact \ref{TP},  and hence assume that either $f$ is exactly one of the functions defining the first five families
of the Painlev\'e's description, or that $f$ is algebraic over a field of the six family $\mathcal{P}_6$.
Moreover, by Lemmas \ref{Z lattice} and \ref{S parameter}, we can assume that all the relevant lattices of $\C$ are of the form
${\langle  1, \omega\rangle}_{\mathbb{Z}}$ for some $\omega\in \mathbb{C} \setminus \mathbb{R}$.

Clearly, if $f$ belongs to $\mathcal{P}_1$, $\mathcal{P}_2$ or $\mathcal{P}_3$ then the corresponding group is a direct product of one-dimensional groups.
We now study direct products in the families $\mathcal{P}_4$ and $\mathcal{P}_5$.

\smallskip
\noindent \emph{Claim $(1)$.} \emph{$(\mathbb{C}^2,+,(\wp _\omega(u), v-\xi\zeta _\omega (u)))$ is a direct product  of one-dimensional locally $\C$-Nash groups if and only if $\xi=0$.}

\proof For the nontrivial implication, suppose the group is not a direct product. Then, by rank considerations, there are only two possible cases  (see Fact \ref{BDOrank}(2) and Proposition \ref{rank}). One case is when there exists an isomorphism
$$\alpha:(\mathbb{C}^2,+,(\wp _\omega(u), v-\zeta _\omega (u))) \rightarrow (\mathbb{C}^2,+,(e^u,e^v)).$$
Let $\alpha(u,v):=(au+bv,cu+dv)$. Assume  first that $b\neq 0$ (the case $d\not=0$ is similar). Then, we have that
$
\left( e^{au+bv}, e^{cu+dv} \right)\in\mathbb{C}(\wp _{\Omega}(u),\zeta _{\Omega} (u), v)^{alg}.
$
Since $e^{bv}$ is algebraic over $\mathbb{C}(e^{au}, e^{au+bv})$, we get 
$e^{bv}\in\mathbb{C}(e^{au}, \wp _{\Omega}(u),\zeta _{\Omega} (u), v)^{alg}$, a contradiction.

In the other case, we have  an isomorphism
\[
\alpha: (\mathbb{C}^2,+,(\wp_{\omega}(u), v-\zeta _{\omega}(u))) \rightarrow (\mathbb{C}^2,+,(\wp_{\omega_0}(u), v)),
\] for some $\omega_0\in \C\setminus \R$.
Let $\alpha^{-1}(u,v):=(au+bv,cu+dv)$. Note that $\big(\wp _{\omega}(u), v-\zeta _{\omega}(u)\big)\circ \alpha^{-1}\in\mathbb{C}\big(\wp _{\omega _0}(u), v\big)^{alg}$.
By the claim in the proof of Lemma \ref{Z general}, we deduce that $b=0$ and $\wp _{\omega}(au)\in\mathbb{C}(\wp _{\omega_0}(u))^{alg}$.
In particular, $cu+dv-\zeta _{\omega}(au)\in\mathbb{C}\big(\wp _{\omega}(au),v\big)^{alg}$.
Evaluating $v=0$, we get  a  contradiction with Fact \ref{Weierstrass algebraicity}.
\qed

\smallskip
\noindent \emph{Claim $(2)$.}  \emph{$(\mathbb{C}^2,+,(\wp _\omega(u), e^v\widetilde{\sigma} _{\omega,\xi} (u)))$ is a direct product of  $1$-dimensional locally $\C$-Nash groups if and only if $\xi\in {\langle 1,\omega \rangle}_{\mathbb{Q}}$.}
\proof We first show left to right implication.  By rank considerations, for some $\omega_0\in \C\setminus \R$ there exists an isomorphism
\[
\alpha:(\mathbb{C}^2,+,(\wp _{\omega}(u),
e^v\widetilde{\sigma} _{\omega,\xi}(u)) \rightarrow
(\mathbb{C}^2,+,(\wp _{\omega_0}(u),e^v)).
\]
Then $\big(\wp _{\omega_0}(u),e^v\big)\circ \alpha\in\mathbb{C}\big(\wp _{\omega}(u), e^v\widetilde{\sigma} _{\omega,\xi}(u)\big)^{alg}$, so
by Lemma \ref{S general}, $0\in \varXi (\omega,\xi)$.
By Lemmas \ref{S Theta}.$(1)$ and \ref{S characterization} we get that
$\xi\in \varXi (\omega,0)={\langle 1,\omega\rangle}_{\mathbb{Q}}$, as required.
On the other hand, if $\xi\in {\langle 1,\omega\rangle}_{\mathbb{Q}}=\varXi (\omega,0)$ then $0\in \varXi (\omega,\xi)$ and therefore, by Lemma \ref{S general},  $(\mathbb{C}^2,+,(\wp _{\omega}(u),
e^v\widetilde{\sigma} _{\omega,\xi}(u))$ and 
$(\mathbb{C}^2,+,(\wp _{\omega}(u),e^v))$ are isomorphic, as required.\qed

All in all, to prove that the types from (1c) to (4c) cover all the possible isomorphic types it only remains to show that if the coordinate functions of $f$ are in $\mathcal{P}_6$ then $(\mathbb{C}^2,+,f)$  is isomorphic to the universal covering of an abelian surface. Suppose that the coordinates functions of $f:=(f_1,f_2)$ are algebraic over a field of the family $\mathcal{P}_6$,
\emph{i.e.}, over a field of the form $\mathbb{C}(\Lambda)$ for some lattice $\Lambda$ of $\C^2$ and satisfying $\tr_{\mathbb{C}} \mathbb{C}(\Lambda)=2$. Because of the latter we know  that $\mathbb{C}^2/\Lambda$ is an abelian variety. Thus, we only have to show that the analytic covering map $(\mathbb{C}^2,+,f)\rightarrow \mathbb{C}^2/\Lambda$ lies properly in the locally $\mathbb{C}$-Nash category. For that reason, we must describe the algebraic structure of $\mathbb{C}^2/\Lambda$.

By \cite[Ch.V,\,\textsection\,$12$,\,Thm.\,$1$]{Siegel}, there are theta functions $g_0(z), \ldots ,g_m(z)$ whose quotients generate
$\mathbb{C}(\Lambda )$ and such that if we write out all the homogeneous algebraic relations of the form
$P(g_0(z),\ldots ,g_m(z))=0$ then the equations $P(x_0,\ldots ,x_m)=0$ define a nonsingular irreducible algebraic variety $\mathcal{N}$
in a $m$-dimensional projective space, which may be mapped biregularly onto the period torus $\mathcal{T}:=\mathbb{C}^2/\Lambda$
by the correspondence $x_i=g_i(z)$ for $i\in \{0,\ldots ,m\}$.
Thus, define
\[
\Phi:\mathcal{T}\rightarrow \mathcal{N}: z\mapsto (g_0(z):\ldots :g_{m}(z)).
\]
We can assume that $g_0(0)\neq 0$.
Clearly, the addition group operation in $\mathcal{T}$ induces a group operation in $\mathcal{N}$ via
\[
a \oplus  b := \Phi (\Phi ^{-1}( a)+ \Phi ^{-1}(b)), \quad  a, b\in \mathcal{N}.
\]
This group operation is regular and $\mathcal{N}$ is an algebraic group (see the paragraph above \cite[Ch.V, \textsection $13$, Thm.\,$2$]{Siegel}).
Now, consider the continuous homomorphism
\[
\Psi:(\mathbb{C}^2,+,f)\rightarrow  \mathcal{N}:(u,v)\mapsto \Phi(\bar{u},\bar{v})
\]
and let us show that it is a locally $\mathbb{C}$-Nash map.
We can assume that a chart of the identity of the locally $\mathbb{C}$-Nash structure on $\mathcal{N}$ is given
by the restriction of the projection $\pi:\mathcal{N}\rightarrow \mathbb{C}^2:(x_0:\ldots:x_m)\mapsto (x_1x^{-1}_0,x_2x^{-1}_0)$ to some open neighbourhood $W$ of the identity of $\mathcal{N}$.
By \cite[Prop.\,3.3]{BDOLCN}, the map $\Psi$ is a locally $\mathbb{C}$-Nash map if $\pi \circ\Psi \circ f^{-1}$ is a $\mathbb{C}$-Nash map.
In other words, if $(g_1g^{-1}_0,g_2g^{-1}_0)\in\C(f)^{alg}$.
Recall that the theta functions satisfy
\[
\mathbb{C}(\Lambda)=\mathbb{C}(g_ig^{-1}_j \suchthat 1\leq i,j\leq m)=\mathbb{C}(g_ig^{-1}_0 \suchthat 1\leq i\leq m).
\]
Since $\tr_{\mathbb{C}} \mathbb{C}(\Lambda)=2$ and we are assuming that $g_1g^{-1}_0$ and $g_2g^{-1}_0$ are algebraically independent, it follows that
$\mathbb{C}(\Lambda)$ is an algebraic extension of $\mathbb{C}(g_1g^{-1}_0,g_2g^{-1}_0)$.
In particular, $f\in\mathbb{C}(g_1g^{-1}_0,g_2g^{-1}_0)^{alg}$.
This shows that $(g_1g^{-1}_0,g_2g^{-1}_0)\in\C(f)^{alg}$, as required.

To finish the proof of (I) we must show that two groups  of different types (from (1c) to (4c)) cannot be isomorphic. This  follows by Fact \ref{BDOrank}(2) and Proposition \ref{rank} considering ranks and the cases for which rank do not make a difference, are done in the proofs of Claim (1) and Claim (2) above in this proof.

(II) To prove (i), let $f_1,f_2,h_1,h_2$ be $\id,\exp$ or $\wp_\Omega$ for a lattice $\Omega$ of $(\C,+)$. Let $$\alpha:(\mathbb{C}^2,+,(f_1(u),f_2(v)))\rightarrow
(\mathbb{C}^2,+,(h_1(u),h_2(v)))$$ be an isomorphism, say $\alpha(u,v)=(au+bv,cu+dv)$. Therefore,  $f_1(au+bv)$ and $f_2(cu+dv)$ belong to   $\C(h_1(u),h_2(v))^{alg}$. Assuming first that $a\neq 0$ and evaluating $v=0$ we get the result. If $a=0$, then the result is proved in a similar way.

The other cases follow from Proposition\,\ref{Z general},   Proposition\,\ref{S general} and  Corollary\,\ref{S characterization}, and \cite[Thm.\,3.10]{BDOLCN}, respectively.
\end{proof}

We end this section by computing the automorphism groups of the two-dimensional locally $\C$-Nash groups given in  Theorem\,\ref{2dim C-classification}. 
We need to introduce more notation (recall also Definition\,\ref{defci}).
Given two complex numbers $a$ and $b$, we denote by $\diag(a,b)$ the diagonal $2\times 2$ matrix whose entries in 
the diagonal are $a$ and $b$ and that, given two subsets $A$ and $B$, we denote $\Diag(A,B):=\{\diag(a,b)\,|\,a\in A,\,b\in B\}$.
The isomorphisms and automorphisms of the universal coverings of abelian varieties --which are not isomorphic to a direct product
of elliptic curves-- are beyond the objectives of this paper.

\begin{proposition}\label{C-automorphisms 2}
Let $\omega,\omega_1,\omega_2\in \mathbb{C}\setminus \mathbb{R}$, $\xi \in \mathbb{C} \setminus {\langle 1,\omega\rangle}_{\mathbb{Q}}$
and  $\Omega :={\langle 1,\omega \rangle}_\mathbb{Z}$. 
Then,

\smallskip

\noindent\emph{(1)} $\Aut(\mathbb{C}^2,+,\id\times \id)=\GL_2(\mathbb{C})$.\\
\emph{(2)} $\Aut(\mathbb{C}^2,+,\exp\times \id)= \Diag(\mathbb{Q}^*, \mathbb{C}^*)$.\\
\emph{(3)} $\Aut(\mathbb{C}^2,+,\wp _\omega \times \id)= \Diag( K_\omega ^*, \mathbb{C}^*)$.\\
\emph{(4)} $\Aut(\mathbb{C}^2,+,\exp\times \exp)=\GL_2(\mathbb{Q})$.\\
\emph{(5)} $\Aut(\mathbb{C}^2,+,\wp _\omega \times \exp)=\Diag(K_\omega ^*, \mathbb{Q}^*)$.\\
\emph{(6.1)} $\Aut(\mathbb{C}^2,+,\wp _{\omega_1} \times \wp_{\omega_2})=\diag(1,\tau ^{-1})\GL_2(K_{\omega_1})\diag(1,\tau )$,  

if  $\wp_{\omega_1}(\tau u)\in\mathbb{C}(\wp_{\omega_2}(u))^{alg}$, for some $\tau \in \mathbb{C}^*$,  and\\
\emph{(6.2)} $\Aut(\mathbb{C}^2,+,\wp _{\omega_1} \times \wp_{\omega_2})=\Diag(K^*_{\omega_1}, K^*_{\omega_2})$,  otherwise.\\ 
\emph{(7)} $\Aut(\mathbb{C}^2,+,(\wp _\omega(u),v-\zeta _\omega(u)))=
\left\{ q\begin{pmatrix} 1 & 0\\ 
\qc(\Omega,q\Omega) & [\Omega: q\Omega]q^{-2} \end{pmatrix}
\suchthat q\in K_\omega^* \right\}$.\\
\emph{(8)} $\Aut(\mathbb{C}^2,+,(\wp _\omega(u),e^v\widetilde{\sigma} _{\omega,\xi}(u)))=
\left\{ q\begin{pmatrix} 1 & 0\\ \xi \qc(\Omega,q\Omega) & 1 \end{pmatrix}
\suchthat q\in \mathbb{Q}^* \right\}$.\\
\end{proposition}
\begin{proof}
Let $(\C^2, +,f)$ be one of the groups under consideration. Then, recall that  $\alpha\in\Aut(\C^2, +,f)$ if and only if $\alpha\in \GL_2(\mathbb{C})$ and $f\circ \alpha\in\mathbb{C}(f)^{alg}$. Making use of the one-dimensional corresponding result of Proposition\,\ref{C-automorphisms}, the computation (1) to (5) is easy. Note that the two cases (6.1) and (6.2) correspond to whether $(\mathbb{C},+,\wp_{\omega_1})$ and $(\mathbb{C},+,\wp_{\omega_2})$ are isomorphic or not, respectively

$(6.1)$ We first  check that $\Aut(\mathbb{C}^2,+,\wp_{\omega}\times\wp_{\omega})= \GL_2(K_{\omega})$. If $\alpha:=(au+bv,cu+dv)\in\Aut(\mathbb{C}^2,+,\wp_{\omega}\times\wp_{\omega})$ then $(\wp_{\omega}(au+bv),\wp_{\omega}(cu+dv))$ 
is algebraic over $\mathbb{C}(\wp_{\omega}(u),\wp_{\omega}(v))$.
If we evaluate at $v=v_0$ -- for some $v_0\notin {\langle 1,\omega _1\rangle}_{\mathbb{Z}}$ --  we get 
$(\wp_{\omega}(au+bv_0),\wp_{\omega}(cu+dv_0))\in\mathbb{C}(\wp_{\omega}(u))^{alg}$.
Hence, both $\wp_{\omega}(au)$ and $\wp_{\omega}(cu)$ 
are algebraic over $\mathbb{C}(\wp_{\omega}(u))$, so that $a,c\in K_{\omega}$. 
Similarly,  we obtain $b,d\in K_{\omega}$. 
Thus, $\alpha\in \GL_2(K_{\omega})$. 
For the converse, note that if $a,b,c,d\in K_{\omega}$, we obtain, since $\wp_{\omega}(u)$ admits an AAT, that $\wp_{\omega}(au+bv)$ is algebraic over 
$\mathbb{C}(\wp_{\omega}(au),\wp_{\omega}(bv))$, which in turn is algebraic over $\mathbb{C}(\wp_{\omega}(u),\wp_{\omega}(v))$. 
Similarly, $\wp_{\omega}(cu+dv)\in\mathbb{C}(\wp_{\omega}(u),\wp_{\omega_1}(v))^{alg}$, as required.
For the general case, assume there is  a $\tau \in \mathbb{C}^*$ such that $\wp_{\omega_1}(\tau u)\in\wp_{\omega_2}(u)^{alg}$, and note that 
$(u,\tau v)$ is an isomorphism from $(\mathbb{C}^2,+,\wp_{\omega_1}\times\wp_{\omega_2})$ to $(\mathbb{C}^2,+,\wp_{\omega_1}\times\wp_{\omega_1})$, so that we can apply the previous case. 

$(6.2)$ By hypothesis, there is not $\tau \in \mathbb{C}^*$ such that $\wp_{\omega_1}(\tau u)\in\wp_{\omega_2}(u)^{alg}$. 
If $\alpha:=(au+bv,cu+dv)\in \Aut(\mathbb{C}^2,+,\wp_{\omega_1}\times\wp_{\omega_2})$ then $$(\wp_{\omega_1}(au+bv),\wp_{\omega_2}(cu+dv))\in\mathbb{C}(\wp_{\omega_1}(u),\wp_{\omega_2}(v))^{alg}.$$ 
If we evaluate at $v=v_0$, for some $v_0\notin {\langle 1,\omega _2\rangle}_{\mathbb{Z}}$, then we obtain that $\wp_{\omega_1}(au)\in\mathbb{C}(\wp_{\omega_1}(u))^{alg}$ and $\wp_{\omega_2}(cu)\in\mathbb{C}(\wp_{\omega_1}(u))^{alg}$. 
In particular, $a\in K_{\omega_1}^*$ and $c=0$ (otherwise we would have that $\wp_{\omega_1}(c^{-1}u)\in\mathbb{C}(\wp_{\omega_2}(u))^{alg}$, a contradiction). 
If we evaluate at $u=u_0$, for some $u_0\notin {\langle 1,\omega _1\rangle}_{\mathbb{Z}}$, then we obtain $b=0$ and $d\in K_{\omega_2}^*$,as required. 
The converse is obvious. 

$(7)$ We first show the  inclusion  from right to left. Suppose that $\omega$ is quadratic over $\mathbb{Q}$, so there exist $A,B\in \mathbb{Z}$ and $C\in \mathbb{N}^*$ such that $C\omega^2+B\omega+A=0$. 
Let $q\in K^*_\omega$. 
We can write $q=\frac{n}{cw+d}$, for certain $n\in \mathbb{N}^*$ and $c,d\in \mathbb{Z}$. 
If we define $a:=Cd-cB$ and $b:=-Ac$ then
\[
\omega=\frac{a\omega +b}{C(c\omega +d)}. 
\]
Moreover, $ad-bc\neq 0$, because otherwise the equality above would provide that $\omega\in \mathbb{R}$, a contradiction. 
Finally, since $q=\frac{Cn}{C(cw_1+d)}$, by Lemma\,\ref{Z general} the map $\alpha(u,v)$ defined by
\[
\left(qu,q^{-1}\left(\mathfrak{c}(\Omega':Cn\Omega')-\frac{\mathfrak{c}(\Omega,Cn\Omega')[\Omega',Cn\Omega']}{[\Omega,Cn\Omega']}\right)u+\frac{[\Omega' :Cn\Omega' ]}{[\Omega:Cn\Omega']}q^{-1}v\right)
\]
where $\Omega':=q^{-1}\Omega$, is in $\Aut(\mathbb{C}^2,+,(\wp _{\omega}(u),v-\zeta _{\omega}(u)))$. On the other hand, by Proposition\,\ref{defci} and by Lemma\,\ref{qc}(1) we get $\alpha (u,v)  = q\big(u,\qc(\Omega,q\Omega)u+[\Omega:q\Omega]q^{-2}v\big)$, 
as required. Suppose now that $\omega$ is not quadratic over $\mathbb{Q}$ and let $q\in K^*_\omega$, so $q=\frac{n}{d}$ for some $n\in \mathbb{N}$ and $d\in \mathbb{Z}^*$. 
Clearly, $\omega=\frac{d\omega}{d}$, so that, by Lemma\,\ref{Z general}, we obtain
\[
\alpha (u,v)  = \left(qu,q^{-1}\left(\mathfrak{c}(\Omega':n\Omega')-\frac{\mathfrak{c}(\Omega,n\Omega')[\Omega',n\Omega']}{[\Omega,n\Omega']}\right)u+\frac{[\Omega' :n\Omega' ]}{[\Omega:n\Omega']}q^{-1}v\right)
\]
where $\Omega':=q^{-1}\Omega$, is in $\Aut(\mathbb{C}^2,+,(\wp _{\omega}(u),v-\zeta _{\omega}(u)))$. Again, it is enough to apply Proposition\,\ref{defci} and Lemma\,\ref{qc}(1).

For the other inclusion, let  $$\alpha (u,v)=(au+bv,cu+dv)\in \Aut(\mathbb{C}^2,+,(\wp _{\omega}(u),v-\zeta _{\omega}(u))).$$ 
By the claim in the proof of Lemma\,\ref{Z general} (with $\xi=1$), we have that $b=0$ and $\wp _{\omega} (au)\in\mathbb{C}(\wp _{\omega}(u))^{alg}$.
By Proposition\,\ref{C-automorphisms}, $a\in K_\omega^*$. 
So it is enough to prove that, for each $a\in K^*_\omega$, there exists at most one automorphism of the form $\alpha (u,v)=(au,cu+dv)\in \GL_2(\mathbb{C})$. 
That is, an automorphism of the form $(u,cu+dv)$ must be the identity, which follows from the claim of the proof of Proposition\,\ref{defci}. 

$(8)$ Firstly, we show $q\begin{pmatrix} 1 & 0\\ \xi \qc(\Omega,q\Omega) & 1\end{pmatrix}\in 
\Aut(\mathbb{C}^2,+,(\wp _{\omega}(u),e^v\widetilde{\sigma} _{\omega,\xi}(u)))$ for each $q\in \mathbb{Q}^*$.
Take $m,n\in \mathbb{Z}^*$ coprimes such that $q=mn^{-1}$.
By Corollary\,\ref{S general: specialcase}, we have that both  
\begin{align*}
\alpha_1(u,v) &:=m^{-1}\big(u,-\xi \mathfrak{c}(\Omega,m\Omega)u+v\big)\\
\alpha_2(u,v) &:=n^{-1}\big(u,-\xi \mathfrak{c}(\Omega,n\Omega)u+v\big)
\end{align*}
belong to $\Aut(\mathbb{C}^2,+,(\wp _{\omega}(u),e^v\widetilde{\sigma} _{\omega,\xi}(u)))$.
Therefore, by Lemma\,\ref{qc}(2),
\[
(\alpha_1^{-1}\circ \alpha _2)(u,v)
=q(u,\xi\qc(\Omega,q\Omega)u+v)
\]
also belongs to $\Aut(\mathbb{C}^2,+,(\wp _{\omega}(u),e^v\widetilde{\sigma} _{\omega,\xi}(u)))$, as required.

Next, let  $\alpha (u,v)=(au+bv,cu+dv)\in \Aut(\mathbb{C}^2,+,(\wp _{\omega}(u),e^v\widetilde{\sigma} _{\omega,\xi}(u)))$. 
By Lemma\,\ref{representation}, $a\in K^*_\omega$, $b=0$ and $d\in \mathbb{Q}^*$. 
If $\omega$ is not quadratic then $a\in \mathbb{Q}^*$ because $K^*_\omega=\mathbb{Q}^*$.
Moreover:

\medskip

\noindent
{\it Claim. Let $\omega$ be quadratic over $\mathbb{Q}$. If $\alpha (u,v):=(au+bv,cu+dv)\in \GL_2(\mathbb{C})$ is an automorphism of $\big(\mathbb{C}^2,+,(\wp_{\omega}(u),e^v\widetilde{\sigma}_{\omega ,\xi}(u))\big)$
then $a\in \mathbb{Q}^*$.}

\smallskip

\noindent
\emph{Proof of the claim.}
Recall that, by hypothesis, $\xi \notin {\langle 1,\omega \rangle}_{\mathbb{Q}}=K_\omega$.
Suppose by contradiction that $a\in K_\omega \setminus \mathbb{Q}^*$.
By Lemma\,\ref{representation}, both $b=0$ and $\wp_{\Omega}(au)\in\mathbb{C}(\wp _{\Omega}(u))^{alg}$.
In particular, $a\neq 0$ and $\wp_{a^{-1}\Omega}(u)\in\mathbb{C}(\wp _{\Omega}(u))^{alg}$.
By Lemma\,\ref{C-wp}, there exists a sublattice $\Omega'$ of both $a^{-1}\Omega$ and $\Omega$.
By Lemma\,\ref{S sublattices}, there exists $\alpha _1 (u,v):=(u,c_1u+d_1v)$ in $\GL_2(\mathbb{C})$ such that 
$
\big(\wp_{\Omega}(u),e^v\widetilde{\sigma}_{\Omega ,\xi}(u)\big)\circ \alpha _1\in\mathbb{C}\big(\wp_{\Omega'}(u),e^v\widetilde{\sigma}_{\Omega',\xi}(u)\big)^{alg}.$

Similarly, there exists $\alpha _2 (u,v):=(u,c_2u+d_2v)$ in $\GL_2(\mathbb{C})$ such that 
$
\big(\wp_{a^{-1}\Omega}(u),e^v\widetilde{\sigma}_{a^{-1}\Omega ,a^{-1}\xi}(u)\big)\circ \alpha _2\in\mathbb{C}\big(\wp_{\Omega'}(u),e^v\widetilde{\sigma}_{\Omega',a^{-1}\xi}(u)\big)^{alg}.
$
Let $\beta(u,v) :=(u,cu+dv)$.
Since $\big(\wp_{\Omega}(u),e^v\widetilde{\sigma}_{\Omega ,\xi}(u)\big)\circ \alpha\in\mathbb{C}\big(\wp_{\Omega}(u),e^v\widetilde{\sigma}_{\Omega ,\xi}(u)\big)^{alg}$, we have that
$
 \big(\wp_{a^{-1}\Omega}(u),e^v\widetilde{\sigma}_{a^{-1}\Omega ,a^{-1}\xi}(u)\big)\circ \beta$ is algebraic over $\mathbb{C}\big(\wp_{\Omega}(u),e^v\widetilde{\sigma}_{\Omega ,\xi}(u)\big).
$
Thus, 
$
\mathbb{C}\big(\wp_{\Omega'}(u),e^v\widetilde{\sigma}_{\Omega',a^{-1}\xi}(u)\big)\circ \alpha _2^{-1} \circ \beta\circ \alpha _1\in\mathbb{C}\big(\wp_{\Omega'}(u),e^v\widetilde{\sigma}_{\Omega',\xi}(u)\big)^{alg}.
$
Note that the isomorphism $\alpha _2^{-1} \circ \beta\circ \alpha _1$ is of the form $(u,c_4u+d_4v)\in \GL_2(\mathbb{C})$.
Thus, by Lemma\,\ref{representation}, there exists $d'\in \mathbb{Q}^*$ such that
\[
 a^{-1}\xi -d'\xi=(a^{-1}-d')\xi \in \Omega'\leq  \Omega \leq  K_\omega= {\langle 1,\omega \rangle}_{\mathbb{Q}}\ .
\]
Recall that $a\in K^*_\omega$, so that $a^{-1}-d'\in K_\omega$. 
Moreover, by hypothesis, $a\notin \mathbb{Q}^*$ and, therefore, $a^{-1}-d'\neq 0$.
We deduce that $\xi\in K_\omega$, which is the desired contradiction. {\hfill $_\square$}

\medskip

Finally, as we did in the case $(7)$, it is enough to show that the identity map is the only automorphism of the form $(u,cu+dv)$, for $c\in \mathbb{C}$ and $d\in \mathbb{Q}^*$. 
Indeed, if $(\wp _\omega (u), e^{cu+dv}\widetilde{\sigma}_{\omega,\xi}(u))\in\mathbb{C}(\wp _{\omega}(u),e^v\widetilde{\sigma} _{\omega,\xi}(u))^{alg}$ then, by Lemma\,\ref{representation}, $\xi-d\xi \in \Omega$.
Since $d\in \mathbb{Q}^*$ and $\xi \notin {\langle 1,\omega\rangle}_{\mathbb{Q}}$, we have $d=1$.

As $e^{cu}$ is algebraic over $\mathbb{C}(\wp _{\omega}(u),e^v\widetilde{\sigma} _{\omega,\xi}(u))$ and, hence, over
$\mathbb{C}(\wp _{\omega}(u))$, which implies $c=0$.
\end{proof}

\section{The two-dimensional real case}\label{2dlng}

We are now ready to give the classification of the abelian simply-connected two-dimensional locally Nash groups. A classification of the one-dimensional was given by Madden and Stanton in \cite[\S\,4,Thm.]{Madden_Stanton}
(see also \cite{Madden_Stanton_Errata}).

\begin{fact}\label{1dim R-classification}
Every simply connected one-dimensional locally Nash group is isomorphic to a group of one and only one of the following types:
\begin{enumerate}
\item[$(1)$] $(\mathbb{R},+,\id)$.
\item[$(2)$] $(\mathbb{R},+,\exp)$.
\item[$(3)$] $(\mathbb{R},+,\sin)$.
\item[$(4)$] $(\mathbb{R},+,\wp _\Lambda )$, where $\Lambda ={\langle 1,ia\rangle}_{\mathbb{Z}}$ for some $[a]\in \mathbb{R}^*/\mathbb{Q}^*$.
\end{enumerate}
Moreover, $\Aut(\mathbb{R},+,\exp)$, $\Aut(\mathbb{R},+,\sin)$ and $\Aut(\mathbb{R},+,\wp _{ia})$ are isomorphic to $\mathbb{Q}^*$,
and $\Aut(\mathbb{R},+,\id)$ to $\mathbb{R}^*$.
\end{fact}

For the two-dimensional classification we will make use of the notation introduced in \S\,Preliminaries, and Fact\,\ref{Chandrasekharan}.
\begin{theorem}\label{2dim R-classification}
$(I)$ Every two-dimensional simply connected abelian locally Nash group is isomorphic to a group of one and only one of the following types:

\smallskip
\noindent \emph{(1r)} A direct product of one-dimensional locally Nash groups with charts $\id$, $\exp$, $\sin$ or $\wp _{ai}$ for some $a\in \mathbb{R}^*$.

\noindent \emph{(2r)} $(\mathbb{R}^2,+,(\wp _{ai}(u),v-\zeta _{ai}(u)))$, for some $a\in \mathbb{R}^*$.

\noindent \emph{(3r)} $(\mathbb{R}^2,+,(\wp _{ai}(u),e^v\widetilde{\sigma} _{ai,\xi}(u)))$, for some $a\in \mathbb{R}^*$ and $\xi \in \mathbb{R}\setminus \mathbb{Q}$.

\noindent \emph{(4r)} $\big(\mathbb{R}^2,+,\big(\wp _{ai}(u),\frac{1}{2i}(e^{iv}\widetilde{\sigma}_{ai,\xi i}(u)-e^{-iv}\widetilde{\sigma}_{ai,-\xi i}(u))\big)\big)$, for some  $a\in \mathbb{R}^*$ and $\xi \in \mathbb{R}\setminus a\mathbb{Q}$.

\noindent \emph{(5r)} The universal covering group $(\mathbb{R}^2,+,f)$ of the connected component of the real points of a simple abelian surface defined over $\mathbb{R}$.

\smallskip

$(II)$ The isomorphism classes within each type are defined as follows:

\noindent\emph{(i)} Two groups of type \emph{(1r)} are isomorphic if and only if their factor groups are isomorphic,
where $(\mathbb{R},+,\wp _{ai})$ and $(\mathbb{R},+,\wp _{bi})$ are isomorphic if and only if $a/b\in \mathbb{Q}^*$.

\noindent\emph{(ii)} Two groups of type \emph{(2r)}, defined by $a$ and $b$,\,respectively, are isomorphic if and only if $a/b\in \mathbb{Q}^*$.

\noindent\emph{(iii)} Two groups of type \emph{(3r)}, defined by $(a,\xi _1)$ and $(b,\xi _2)$,\,respectively, are isomorphic if and only if
$a/b\in \mathbb{Q}^*    $ and $\xi _2\in \mathbb{Q}+\xi _1\mathbb{Q}^*$.

\noindent\emph{(iv)} Two groups of type \emph{(4r)}, defined by $(a,\xi _1)$ and  $(b,\xi _2)$,\,respectively, are isomorphic if and only if
$a/b\in \mathbb{Q}^*$ and $\xi _2\in a\mathbb{Q}+\xi _1\mathbb{Q}^*$.

\noindent\emph{(v)} Two groups of type \emph{(5r)} are isomorphic if and only if there is an isogeny between the corresponding abelian surfaces.

\end{theorem}

Related to the isomorphism classes of groups of type (5r), we remark that, given a meromorphic map
$f:\mathbb{C}^2\dashrightarrow \mathbb{C}^2$ admitting an AAT, it may happen that $(\mathbb{R}^2,+,f)$ is the universal covering of
a real abelian surface that is not a direct product of two real elliptic curves whereas $(\mathbb{C}^2,+,f)$ is locally $\mathbb{C}$-Nash
isomorphic to the universal covering of a direct product of two elliptic curves (see \cite[Example\,$48$]{Thesis:Huisman}).

\begin{proof}[Proof of Theorem \emph{\ref{2dim R-classification}}]

$(I)$ Firstly, note that by Lemma\,\ref{wp_conjugate}, all the maps that appear in cases (1r)-(4r) are $\R$-meromorphic. Also, by Theorem\,\ref{2dim C-classification},  the corresponding complex versions of the groups that appear in the statement are locally $\C$-Nash groups (see case (3c) below, for case (4r)). Therefore the groups that appear in cases (1r)-(5r) are indeed locally Nash groups.  On the other hand, as usual, every simply connected two-dimensional locally Nash group is isomorphic to some $(\mathbb{R}^2,+,f)$, for some $\R$-meromorphic map $f$. Moreover, $(\mathbb{C}^2,+,f)$ is a locally $\C$-Nash group.  We already know that there is an isomorphism $\alpha$ between $(\mathbb{C}^2,+,f)$ and a group $(\mathbb{C}^2,+,g)$ of one of the four types (1c)-(4c) in the statement of Theorem \ref{2dim C-classification}. However, such $g$ is not necessarily invariant. We now explain the strategy we are going to follow in all  cases. Since $f$ is invariant, it follows that $\hat{\alpha}(u,v):=\overline{\alpha(\overline{u},\overline{v})}$ is a isomorphism from $(\mathbb{C}^2,+,f)$ to $(\mathbb{C}^2,+,\hat{g})$. In particular, the map $h:=\hat{\alpha}\circ \alpha^{-1}$ is a isomorphism from $(\mathbb{C}^2,+,g)$ to $(\mathbb{C}^2,+,\hat{g})$ which satisfies $\hat{h}\circ h=id$. From the existence of such a  map $h$, and using the computation of the automorphism groups in  Proposition \ref{C-automorphisms 2}, we will find a locally $\C$-Nash group $(\mathbb{C}^2,+,g_0)$,  with $g_0$ invariant, $\R$-isomorphic to  $(\mathbb{C}^2,+,f)$ and such that $(\mathbb{R}^2,+,g_0)$ is of one of the five required types (1r)-(5r). Consequently, we divide the proof depending on cases (1c)-(4c) of Theorem\,\ref{2dim C-classification}.

 
(1c) $(\mathbb{C}^2,+,f)$ is $\C$-isomorphic to a direct product of $1$-dimensional locally $\C$-Nash groups:  We exclude the case of a universal covering of a non-simple complex abelian surface -- \emph{i.e.}, $(\C^2,+,(\wp_{\omega_1}(u),\wp_{\omega_2}(v)))$ --, which will be studied below, together with case (4c). We focus on the two cases that present some difficulties, the others are similar. 
Suppose firstly that there is an isomorphism  $\alpha (u,v)\in \GL_2(\mathbb{C})$ between $(\C^2,+,f)$ and $(\C^2,+,(e^u,e^v))$.  We have that $h:=\hat{\alpha}\circ \alpha^{-1}$ is an automorphism of $(\C^2,+,(e^u,e^v))$,\emph{ i.e.}, it belongs to $\GL_2(\Q)$, and so $\hat{h}=h$. Since $\hat{h}\circ h=\text{id}$, the eigenvalues of $h$ are $1$ or $-1$. Suppose that the eigenvalues are $1$ and $-1$ (for the other eigenvalues the argument is similar). Then, there is $\beta\in \GL_2(\Q)$ such that $h=\beta^{-1} \circ (u,-v) \circ \beta$. In particular, the isomorphism $\gamma:= \beta \circ \alpha$ from $(\C^2,+,f)$ to $(\C^2,+,(e^u,e^v))$
satisfies $\hat{\gamma}\circ \gamma^{-1}=(u,-v)$, and therefore $\gamma(u,v)=(au,biv)$, for some $a,b\in \R$. Hence $\gamma'(u,v)=(au,bv)\in\GL_2(\R)$ is an isomorphism from $(\C^2,+,f)$ to $(\C^2,+,(e^u,e^{iv}))$. Since $\id$ is an isomorphism from $(\C^2,+,(e^u,e^{iv}))$ to $(\C^2,+,(e^u,\sin(v))$, we are done.

Next, assume that there is an isomorphism  $\alpha (u,v)\in \GL_2(\mathbb{C})$ between $(\C^2,+,f)$ and $(\C^2,+,(\wp_{\Omega}(u),v))$, for some lattice $\Omega$ of $\C$.  We have that $h:=\hat{\alpha}\circ \alpha^{-1}$ is an isomorphism from $(\C^2,+,(\wp_{\Omega}(u),v))$ to $(\C^2,+,(\wp_{\overline{\Omega}}(u),v))$. Clearly, it is of the form $h(u,v)=(au,bv)$ for some $a,b\in \C^*$. Since $\wp_{\overline{\Omega}}(au)\in\wp_{\Omega}(u)^{alg}$, we have that $a^{-1}\overline{\Omega}\cap \Omega$ is a lattice. On the other hand, since $\hat{h}\circ h=\text{id}$ we get that $|a|=1$. Thus there is $a_1\in \C^*$ such that $\overline{a_1}a^{-1}_1=a^{-1}$. We deduce that $\Omega_1:=\overline{a_1\Omega}\cap a_1\Omega $ is an invariant sublattice of $a_1\Omega$. Since $|b|=1$, there is also  $b_1\in \C^*$ such that $\overline{b_1}b^{-1}_1=b^{-1}$. Now, clearly  $\beta(u,v):=(a_1u,b_1v)$ is an isomorphism from $(\C^2,+,(\wp_\Omega(u),v))$ to $(\C^2,+,(\wp_{a_1\Omega}(u),v))$. Moreover, by Lemma \ref{C-wp}, the identity is an isomorphism from $(\C^2,+,(\wp_{a_1\Omega}(u),v))$ to $(\C^2,+,(\wp_{\Omega_{1}}(u),v))$. Therefore, $\gamma:=\beta \circ \alpha$ is an isomorphism from $(\C^2,+,f)$ to $(\C^2,+,(\wp_{\Omega_{1}}(u),v))$. We need to show that $\gamma\in \GL_2(\R)$. Indeed,  an easy computation shows that $\hat{\gamma}\circ \gamma^{-1}=\id$, as required.

(2c)  $\alpha (u,v)\in \GL_2(\mathbb{C})$ is an isomorphism from $ (\C^2,+,f)$ to $(\C^2,+,g_{4,\Omega})$, where $g_{4,\Omega}:=(\wp_{\Omega}(u),v-\zeta_{\Omega}(u))$, for some lattice $\Omega$ of $\C$:

By Lemma \ref{wp_conjugate} we have that $h:=\hat{\alpha}\circ \alpha^{-1}$ is an isomorphism from $(\C^2,+,g_{4,\Omega})$ to $(\C^2,+,g_{4,\overline{\Omega}})$. By the claim of the proof of Lemma\,\ref{Z general}, $h$ is of the form $(au,cu+dv)$, for some $a,c,d\in \C$ with $ad\neq 0$, and $\wp_{\overline{\Omega}}(au)$ is algebraic over $\wp_{\Omega}(u)$. Since $\hat{h}\circ h=\text{id}$, similarly as with the above reasoning, we get that there is $a_1\in \C$ such that $\overline{a_1}a^{-1}_1=a^{-1}$ and $\Omega_1:=\overline{a_1\Omega}\cap a_1\Omega $ is an invariant sublattice of $a_1\Omega$.

Clearly $\beta_1(u,v):=(a_1u,a^{-1}_1v)$ is an isomorphism from $(\C^2,+,g_{4,\Omega})$ to $(\C^2,+,g_{4,a_1\Omega})$. On the other hand, by Lemma\,\ref{Z sublattices} we have that $\beta_2(u,v):=(u,\mathfrak{c}(a_1\Omega,\Omega_1)u+[a_1\Omega:\Omega_1]v)$ is an isomorphism from $(\C^2,+,g_{4,a_1\Omega})$ to $(\C^2,+,g_{4,\Omega_1})$. Since $\Omega_1$ is invariant, it has a sublattice of the form $\Omega_0={\langle r_1,r_2i\rangle}$ for some $r_1,r_2\in \R^*$. Therefore $\beta_3(u,v)=(u,\mathfrak{c}(\Omega_1,\Omega_0)u+[\Omega_1:\Omega_0]v)$ is an isomorphism from $(\C^2,+,g_{4,\Omega_1})$ to $(\C^2,+,g_{4,\Omega_0})$.
Moreover, we have that $\beta_4(u,v):=(r_1^{-1}u,r_1v)$ is an isomorphism from $(\C^2,+,g_{4,\Omega_0})$ to $(\C^2,+,g_{4,ri})$ where $r:=\frac{r_2}{r_1}$.

All in all, we have that  $\gamma:=\beta_4 \circ \beta_3\circ \beta_2 \circ\beta_1\circ \alpha$ is an isomorphism from $(\C^2,+,f)$ to $(\C^2,+,(\wp_{ri}(u),v-\zeta_{ri}(u)))$, with $r\in\R^*$. It remains to show that $\gamma  \in \GL_2(\mathbb{R})$. Indeed, $\hat{\gamma}\circ \gamma^{-1}(u,v)\in \text{Aut}(\C^2,+,g_{4,ri})$ and an easy computation shows that its first coordinate function is $u$. Then, it follows from Proposition\,\ref{C-automorphisms 2} that $\hat{\gamma}\circ \gamma^{-1}=\id$, as required.

(3c) $\alpha (u,v)\in \GL_2(\mathbb{C})$ is an isomorphism   from $(\C^2,+,f)$ to $(\C^2,+,g_{5,\omega,\xi})$, where $\omega\in \mathbb{C}\setminus \mathbb{R}$, $\xi \notin {\langle 1,\omega \rangle}_{\mathbb{Q}}$, and  $g_{5,\omega,\xi}(u,v):=(\wp_{\omega}(u),e^v\widetilde{\sigma} _{\omega ,\xi}(u))$). As we will see, this complex case produces two real cases. 

By Lemma \ref{wp_conjugate}, $\hat{g}_{5,\xi ,\omega}(u,v)=\overline{g_{5,\omega,\xi}(\overline{u},\overline{v})}=g_{5,\overline{\omega},\overline{\xi}}(u,v)$. If we denote $\Omega:=\langle 1,\omega \rangle$, then by the claim in the proof of Lemma \ref{representation} and arguing as in the above case, there is $a_1\in \C$ such that  $\Omega_1:=\overline{a_1\Omega}\cap a_1\Omega$ is an invariant sublattice of $a_1\Omega$.

Once again  $\beta_1(u,v):=(a_1u,v)$ is an isomorphism from $(\C^2,+,g_{5,\Omega,\xi})$ to $(\C^2,+,g_{5,a_1\Omega,a_1\xi})$. On the other hand, by Lemma \ref{S sublattices} we have that $$\beta_2(u,v):=(u,a_1\xi\mathfrak{c}(a_1\Omega,\Omega_1)u+[a_1\Omega:\Omega_1]v)$$ is an isomorphism from $(\C^2,+,g_{5,a_1\Omega,a_1\xi})$ to $(\C^2,+,g_{5,\Omega_1,a_1\xi})$.
Since $\Omega_1$ is invariant, it has a sublattice of the form $\Omega_0={\langle r_1,r_2i\rangle}$ for some $r_1,r_2\in \R^*$. Therefore $\beta_3(u,v)=(u,a_1\xi\mathfrak{c}(\Omega_1,\Omega_0)u+[\Omega_1:\Omega_0]v)$ is an isomorphism from $(\C^2,+,g_{5,\Omega_1,a_1\xi})$ to $(\C^2,+,g_{5,\Omega_0,a_1\xi})$. Moreover, we have that $\beta_4(u,v):=(r_1^{-1}u,v)$ is an isomorphism from $(\C^2,+,g_{5,\Omega_0,a_1\xi})$ to $(\C^2,+,g_{5,ri,r^{-1}_1a_1\xi})$ where $r:=\frac{r_2}{r_1}$.  All in all, we have that $\gamma:=\beta_4 \circ \beta_3 \circ \beta_2 \circ\beta_1\circ \alpha$ is an isomorphism from $(\C^2,+,f)$ to $(\C^2,+,g_{5,ri,r_1^{-1}a_1\xi})$.

Note that $\xi_0:=r_1^{-1}a_1\xi \notin {\langle 1,ri \rangle}_{\mathbb{Q}}$, otherwise $\xi\in  {\langle 1,\omega \rangle}_{\mathbb{Q}}$. The map  $\hat{\gamma}\circ \gamma^{-1}$ is an isomorphism from $(\C^2,+,g_{5,ri,\xi_0})$ to $(\C^2,+,g_{5,ri,\overline{\xi_0}})$ and therefore $\overline{\xi_0}\in \Xi(ri,\xi_0)={\langle 1,ri \rangle}_\Q+\xi_0 K^*_{ri}$.

Let $\xi_0':=\xi_0+\overline{\xi_0}$. We are going to study two settings depending on $\xi_0'\in {\langle 1,ri\rangle}_\Q$ or not. Assume  first that $\xi_0'\notin  {\langle 1,ri\rangle}_\Q$. Then, the real number $\xi_0'$ belongs to ${\langle 1,ri\rangle}_\Q+\xi_0 K^*_{ri}=\Xi(ri,\xi_0)$, so there is an isomorphism
$$\beta_5:(\mathbb{C}^2,+,g_{5,ri,\xi_0})\rightarrow (\mathbb{C}^2,+,g_{5,ri,\xi'_0}).$$
Consider the isomorphism $\gamma_1:=\beta_5\circ \gamma$ from $(\C^2,+,f)$ to $(\mathbb{C}^2,+,g_{5,ri,\xi_0'})$. Let $h_1:=\hat{\gamma}_1 \circ \gamma_1^{-1} \in \text{Aut} (\mathbb{C}^2,+,g_{5,ri,\xi'_0})$, which satisfies $\hat{h}_1 \circ h_1=\text{id}$. By Proposition\,\ref{C-automorphisms 2} we have that $h_1=\text{id}$ or $h_1=-\text{id}$.     If $h_1=\text{id}$ then $\gamma_1\in \text{GL}_2(\R)$ and we have proved that $(\R^2,+,f)$ is of type (3r), as required (note that $\xi'_0\notin \Q$ because $\xi_0'\notin {\langle 1,ri\rangle}_\Q$). If $h_1=-\text{id}$ then $\gamma_1(u,v)=\gamma'_1(iu,iv)$ for some $\gamma'_1 \in \text{GL}_2(\R)$. Then $\gamma'_1$ is an isomorphism from $(\C^2,+,f)$ to $(\mathbb{C}^2,+,g_{5,ri,\xi_0'}(iu,iv))$. Note that $g_{5,ri,\xi_0'}(iu,iv)$ equals
$$(\wp_{\langle 1,ri \rangle}(iu),e^{iv}\widetilde{\sigma}_{{\langle 1,ri \rangle},\xi'_0}(iu))
=(-\wp_{\langle i,r \rangle}(u),e^{iv}\widetilde{\sigma}_{{\langle i,r \rangle},-\xi'_0i}(u)),$$
and therefore by claim (1) of Lemma \ref{S Theta} we have that $$g_{5,ri,\xi_0'}(iu,iv)\in\big(\wp_{{\langle i,r \rangle}}(u),\frac{1}{2i}(e^{iv}\widetilde{\sigma}_{{\langle i,r \rangle},-\xi_0' i}(u)-e^{-iv}\widetilde{\sigma}_{{\langle i,r \rangle},\xi'_0 i}(u))\big)\big)^{alg}.$$
Finally, for $\beta_6(u,v):=(r^{-1}u,v)$ we have that $\gamma_1'':=\beta_6\circ \gamma'_1$ is an isomorphism from $(\C^2,+,f)$ to
$$(\C^2,+,\big(\wp_{r^{-1}i}(u),\frac{1}{2i}(e^{iv}\widetilde{\sigma}_{r^{-1}i,-r^{-1}\xi_0'i}(u)-e^{-iv}\widetilde{\sigma}_{r^{-1}i,r^{-1}\xi'_0 i}(u))\big)\big),$$
and hence we have proved that  the corresponding locally (real) Nash group  is of type (4r), as required (again note that $r^{-1}\xi_0'\notin r^{-1}\Q$).

Assume now that  $\xi'_0\in {\langle 1,ri\rangle}_\Q$. Then, the pure imaginary number $\xi_0'':=\xi_0-\overline{\xi_0}\notin {\langle 1,ri\rangle}_\Q$, otherwise $\xi_0 \in {\langle 1,ri\rangle}_\Q$. In particular,
$\xi_0''\in {\langle 1,ri\rangle}_\Q+\xi_0 K^*_{ri}=\Xi(ri,\xi_0)$
and therefore there is an isomorphism
$$\beta_7:(\mathbb{C}^2,+,g_{5,ri,\xi_0})\rightarrow (\mathbb{C}^2,+,g_{5,ri,\xi_0''}).$$

Consider the isomorphism $\gamma_2:=\beta_7\circ \gamma$ from $(\C^2,+,f)$ to $(\mathbb{C}^2,+,g_{5,ri,\xi_0''})$, and note that  $\beta_8(u,v)=(-u,v)$ is an isomorphism from $(\mathbb{C}^2,+,g_{5,ri,-\xi_0''})$ to $(\mathbb{C}^2,+,g_{5,ri,\xi_0''})$. Therefore, we can consider the map $h_2:=\beta_8\circ \hat{\gamma}_2\circ \gamma_2^{-1} \in \text{Aut} (\mathbb{C}^2,+,g_{5,ri,\xi''_0})$, which satisfies $\hat{h}_2\circ h_2=\text{id}$. By Proposition \ref{C-automorphisms 2} we have that $h_2=\text{id}$ or $h_2=-\text{id}$. If $h_2=\text{id}$ then $\gamma_2(u,v)=\gamma'_2(iu,v)$ for some $\gamma'_2\in \text{GL}_2(\R)$.
In particular, $\gamma'_2$ is an isomorphism from $(\C^2,+,f)$ to $(\mathbb{C}^2,+,g_{5,ri,\xi_0''}(iu,v))$. Now, we have that $g_{5,ri,\xi_0''}(iu,v)$ equals
$$(\wp_{\langle 1,ri \rangle}(iu),e^{v}\widetilde{\sigma}_{{\langle 1,ri \rangle},\xi''_0}(iu))
=(-\wp_{\langle i,r \rangle}(u),e^{v}\widetilde{\sigma}_{{\langle i,r \rangle},-\xi''_0i}(u)).$$
Therefore, if we denote $\beta_9(u,v)=(r^{-1}u,v)$, the map $\gamma''_2:=\beta_9\circ \gamma'_2$ is an isomorphism from $(\C^2,+,f)$ to
$(\C^2,+,g_{5,r^{-1}i,-r^{-1}\xi''_0i})$. Since $-r^{-1}\xi''_0i$ is real and do not belong to $\Q$ -- otherwise $\xi''_0\in {\langle 1,ri\rangle}_\Q$ -- we have showed that the corresponding real group  is of type (3r), as required.
If $h_2=-\text{id}$ then $\gamma_2(u,v)=\gamma'_2(u,iv)$, for some $\gamma'_2\in \text{GL}_2(\R)$, and similarly as with the argument above, it turns out that $(\R^2,+,f)$ is of type (4r).

\smallskip
(4c) and remain case of (1c):  $(\C^2,+,f)$ is the universal covering of an abelian surface defined over $\C$, non necessarily simple. Then, 
$f\in\mathbb{C}(\Lambda)^{alg}$, for some lattice $\Lambda \leq \C^2$  with $\tr_{\mathbb{C}} \mathbb{C}(\Lambda) = 2$.
In particular, by \cite[Lem.4.6(3)]{BDOLCN} and Fact \ref{BDOrank}, the period group $\Lambda_f$ of $f$ is also a lattice of
$\C^2$ and, since the coordinates functions $f_1,f_2\in \mathbb{C}(\Lambda_f)$ are algebraically independent, $\tr_\mathbb{C} \mathbb{C}(\Lambda_f) = 2$
(see \cite[Ch.\,5 \textsection 11 Thms.\,5 and 6]{Siegel}).

On the other hand, by the proof of Theorem \ref{2dim C-classification}(I.4c) the quotient map $\Psi:(\mathbb{C}^2 ,+,f) \rightarrow \mathbb{C}^2 /\Lambda_f$
is a locally $\mathbb{C}$-Nash universal covering map, where the abelian variety $\mathbb{C}^2/\Lambda_f$ is endowed with the canonical structure of complex projective variety.
Moreover, since $\Lambda_f$ is invariant by \cite[Lem.4.6\,(3)]{BDOLCN}, the conjugation map induces an antiholomorphic involution $\tau$ of $\mathbb{C}^2/\Lambda_f$,
so we can assume that $\mathbb{C}^2/\Lambda_f$ is defined over $\mathbb{R}$ (see \cite[pag.\,56]{Serre2} or the discussion after Corollary \ref{isogenous corollary}). Note also that $\Psi(\mathbb{R}^2)$ is contained in the fixed points of $\tau$, {\it i.e.}, the real points of the real abelian variety $\mathbb{C}^2/\Lambda_f$.
All in all, $\Psi|_{\mathbb{R}}$ maps $(\mathbb{R}^2 ,+,f)$ onto the connected component of the real points of $\mathbb{C}^2/\Lambda_f$ with discrete kernel
and, therefore, it is a locally Nash covering. Note that if $\mathbb{C}^2/\Lambda_f$ is simple over $\R$ then $(\mathbb{R}^2 ,+,f)$ is of type (5r) in the statement of the theorem, and if it is not simple over $\R$ then, it is of type (1r).

To finish the proof of $(I)$ it suffices to check that two groups of different type are not isomorphic.
We note that if $(\mathbb{R}^2,+,f)$ and $(\mathbb{R}^2,+,g)$ are isomorphic as locally Nash groups then $(\mathbb{C}^2,+,f)$ and $(\mathbb{C}^2,+,g)$ are also isomorphic as locally $\mathbb{C}$-Nash groups. Therefore -- by Theorem \ref{2dim C-classification} -- it only remains to prove that groups of type (3r) and (4r) cannot be isomorphic.
Suppose there is an isomorphism  $\alpha(u,v)=(au+bv,cu+dv)\in  \GL_2(\R)$  from $(\R^2,+,g_{5,\omega_1,\xi_1})$ to $(\R^2,+,g_{6,\omega_2,i\xi_2})$, where
$\omega_1\in i\mathbb{R}^*, \xi_1 \in \mathbb{R}\setminus \Q$ and $\omega_2\in i\mathbb{R}^*,\xi_2 \in \mathbb{R}\setminus i\omega_2\Q$ and
\[
g_{6,\omega_2,i\xi_2}(u,v):=\left(\wp_{\omega_2}(u),\frac{1}{2i}\big(e^{iv}\widetilde{\sigma}_{\omega_2,i\xi_2}(u)-e^{-iv}\widetilde{\sigma}_{\omega_2,-i\xi_2}(u)\big)\right).
\]
Since the map $\beta(u,v)=(u,iv)$ is an isomorphism from $(\C^2,+,g_{6,\omega_2,i\xi_2})$ to $(\C^2,+,g_{5,\omega_2,i\xi_2})$, we get that $(\beta\circ \alpha)(u,v)=(au+bv,ciu+div)$ is an isomorphism from $(\C^2,+,g_{5,\omega_1,\xi_1})$ to $(\C^2,+,g_{5,\omega_2,i\xi_2})$. By the claim in the proof of Proposition \ref{representation} we have that $di\in \Q^*$, a contradiction.

\smallskip
$(II)$ We now check the isomorphism classes inside each family. To prove (i) it is enough to note that, as we pointed out in the proof of Theorem\,\ref{2dim C-classification}(II.i), if two groups -- which are direct products of $1$-dimensional -- are isomorphic, then their factors are isomorphic.

(ii) Take $\omega _1,\omega _2\in i\mathbb{R}^*$.
If $\alpha (u,v):=(a_{11}u+a_{12}v,a_{21}u+a_{22}v)\in \GL_2(\mathbb{R})$ is an isomorphism from $(\mathbb{R}^2,+,g_{4,\omega_1,1})$ to $(\mathbb{R}^2,+,g_{4,\omega _2,1})$ then,
arguing as in the proof of the claim in Lemma \ref{Z general}, $\wp _{\omega _2}(a_{11}u)\in\mathbb{C}(\wp _{\omega _1}(u))^{alg}$.
Thus, by Fact \ref{1dim R-classification}, we infer that $\omega _1/\omega _2\in \mathbb{Q}$.
Conversely, if $\omega _1/\omega _2\in \mathbb{Q}$ then, by Lemma \ref{Z general} and by Lemma \ref{qc}, there exists $\alpha \in \GL_2(\mathbb{R})$ such
that $g_{4,\omega_2,1}\circ \alpha$ is algebraic over $\mathbb{C}(g_{4,\omega_1,1})$.

(iii) Take $\omega _1,\omega _2\in i\mathbb{R}^*$ and $\xi _1,\xi _2\in \mathbb{R}$.
If $\alpha (u,v):=(a_{11}u+a_{12}v,a_{21}u+a_{22}v)\in \GL_2(\mathbb{R})$ is an isomorphism from $(\mathbb{R}^2,+,g_{5,\omega _1,\xi _1})$ to $(\mathbb{R}^2,+,g_{5,\omega _2,\xi _2})$ then,
by the claim in the proof of Lemma \ref{representation}, $\wp _{\omega _2}(a_{11}u)\in\mathbb{C}(\wp _{\omega _1}(u))^{alg}$.
By Fact \ref{1dim R-classification}, we deduce $\omega _1/\omega _2\in \mathbb{Q}$.
By Theorem \ref{2dim C-classification}(II.ii), we also have that $\xi _2\in {\langle 1,\omega _1\rangle}_{\mathbb{Q}}+\xi _1 K_{\omega _1}^*$.
Since $\xi _1,\xi _2\in \mathbb{R}$, we get that $\xi _2\in \mathbb{Q}+\xi _1 \mathbb{Q}^*$.
For the converse, if $\omega _1/\omega _2\in \mathbb{Q}$ and $\xi _2\in \mathbb{Q}+\xi _1 \mathbb{Q}^*$ then, by Proposition \ref{S characterization} and Lemma \ref{S general}, there exists
$\alpha \in \GL_2(\mathbb{C})$ such that $g_{5,\omega _2,\xi _2}\circ \alpha\in\mathbb{C}(g_{5, \omega _1,\xi _1})^{alg}$.
So we must check that $\alpha\in \GL_2(\mathbb{R})$.
Indeed, following the computations of the above results, if we choose $a,d\in \mathbb{Z}$ such that $\frac{\omega _1}{\omega _2}=\frac{d}{a}$ and we denote by $p,q\in \mathbb{Q}$, $q\neq 0$,
the rational numbers satisfying $d\xi_2=p+q\xi_1$, and we take the invariant sublattice $\Omega={\langle d,a\omega _1\rangle}_\mathbb{Z}$ of $\Omega_1:={\langle 1,\omega _1\rangle}_{\mathbb{Z}}$ then
\[
\alpha(u,v)=\left(\frac{1}{d}u, -\frac{d\xi_2\mathfrak{c}(\Omega_1,\Omega)+C}{[\Omega_1,\Omega]}u+\frac{q}{[\Omega_1,\Omega]}v\right)\in \GL_2(\mathbb{R}),
\]
where $C\in \mathbb{R}$ satisfies $\widetilde{\sigma}_{p_2^{-1}\Omega_1,p}(u)=e^{Cu+D}$ for a certain $D\in \mathbb{R}$ and $p_1,p_2\in \mathbb{Z}$ such that $p=p_1/p_2$.
The existence of such $C$ and $D$ was shown in the proof of Lemma \ref{S Theta}$(5)$ and we note that both numbers must be real, because $\widetilde{\sigma}_{p_2^{-1}\Omega_1,p}$ is invariant.

(iv) Let $\omega_1,\omega_2\in i\mathbb{R}^*$, $\xi_1,\xi_2\in \mathbb{R}$ and 
\[
g_{6,\omega_1,i\xi_1}(u,v):=\frac{1}{2i}\big(e^{iv}\widetilde{\sigma}_{\omega_1,i\xi_1}(u)-e^{-iv}\widetilde{\sigma}_{\omega_1,-i\xi_1}(u)\big).
\]
First note that for $\beta_1(u,v)=(u,iv)$ we have $g_{6,\omega_1,i\xi_1}\in\mathbb{C}(g_{5,\omega_1,i\xi}\circ \beta_1)^{alg}$.
Moreover,  it is easy to check that for $\gamma_1(u,v)=(\omega_1u,-iv)$ we have $$g_{5,\omega_1,i\xi}\circ \beta_1 \circ \gamma_1\in\mathbb{C}(g_{5,\omega_1^{-1},\omega_1^{-1}i\xi_1})^{alg}.$$
Therefore, $g_{6,\omega_1,i\xi_1}\circ \gamma_1$ is algebraic over $\mathbb{C}(g_{5,\omega_1^{-1},\omega_1^{-1}i\xi_1})$.
Similarly, for $\gamma_2(u,v)=(\omega_2u,-iv)$, so $g_{6,\omega_2,i\xi_2}\circ \gamma_2\in\mathbb{C}(g_{5,\omega_2^{-1},\omega_2^{-1}i\xi_2})^{alg}$.
Hence, given $\alpha(u,v)=(au+bv,cu+dv)\in \GL_2(\mathbb{R})$, we deduce that $g_{6,\omega_2,i\xi_2}\circ \alpha\in\mathbb{C}(g_{6,\omega_1,i\xi_1})^{alg}$ if and only if
$g_{5,\omega_2^{-1},\omega_2^{-1}i\xi_2}\circ \gamma^{-1}_2\circ \alpha\circ \gamma_1\in\mathbb{C}(g_{5,\omega_1^{-1},\omega_1^{-1}i\xi_1})^{alg}$.
Since $$(\gamma^{-1}_2\circ \alpha\circ \gamma_1)(u,v)=(\tfrac{\omega_1}{\omega_2}au-\tfrac{i}{\omega_2}bv,i\omega_1cu+dv)\in\GL_2(\mathbb{R}),$$
it follows from case (iii) that $g_{5,\omega_2^{-1},\omega_2^{-1}i\xi_2}\circ \gamma^{-1}_2\circ \alpha\circ \gamma_1\in\mathbb{C}(g_{5,\omega_1^{-1},\omega_1^{-1}i\xi_1})^{alg}$
if and only if $\omega_1/\omega_2\in \mathbb{Q}^*$ and $\xi_2\in i\omega_2\mathbb{Q}^*+\frac{\omega_2}{\omega_1}\xi_1\mathbb{Q}^*=i\omega_1\mathbb{Q}^*+\xi_1\mathbb{Q}^*$, as required.

(v) Let $\alpha:(\mathbb{R}^2,+,f)\rightarrow (\mathbb{R}^2,+,g)$, where $\alpha\in \GL_2(\mathbb{R})$, be  a locally Nash isomorphism between groups of type (5r).
Then, $\alpha:(\mathbb{C}^2,+,f)\rightarrow (\mathbb{C}^2,+,g)$ is a locally $\mathbb{C}$-Nash isomorphism and, therefore, by the proof of
\cite[Thm.\,3.10]{BDOLCN}, we get an isogeny between $\mathbb{C}^2/\Lambda_f$ and $\mathbb{C}^2/\Lambda_g$ defined over $\mathbb{R}$.
Conversely, an isogeny between the real abelian varieties which is defined over $\mathbb{R}$ gives
a Nash map between two connected neighbourhoods of the identity which is a local isomorphism, which in turn extends to a locally Nash isomorphism between
the universal coverings $(\mathbb{R}^2,+,f)$ and $(\mathbb{R}^2,+,g)$.
This ends the proof of $(II)$ and, hence, of the theorem.
\end{proof}

Finally, we compute the automorphism groups of the locally Nash groups of Theorem \ref{2dim R-classification}.

\begin{proposition}\label{R-automorphisms 2}
Let $(\mathbb{R}^2,+,f)$ be a locally Nash group.
$\Aut(\mathbb{R}^2,+,f)$ is one of the following:

\smallskip

\noindent $(1)$ $\GL_2(\mathbb{R})$, if $f=\id\times \id$.\vspace{1mm} \\
\noindent $(2)$ $\Diag(\mathbb{Q}^*, \mathbb{R}^*)$, if $f=\id\times g$ with $g=\exp$, $\sin$ or $\wp _{ai}$, for some $a\in \mathbb{R}^*$. \vspace{1mm} \\
\noindent $(3)$ $\GL_2(\mathbb{Q})$ if $f=g\times g$, with $g=\exp$ or $\sin$. \vspace{1mm} \\
\noindent $(4)$ $\Diag(\mathbb{Q}^*, \mathbb{Q}^*)$, if $f=\exp\times \sin$, $\wp_{ai}\times \exp$ or $\wp _{ai} \times \sin$, for some $a\in \mathbb{R}^*$.\vspace{1mm} \\
\noindent $(5.1)$ $\GL_2(\mathbb{Q})$, if $f=\wp _{ai} \times \wp_{bi}$, for some $a,b\in \mathbb{R}^*$ such that $a/b\in \mathbb{Q}^*$. \vspace{1mm} \\
\noindent $(5.2)$ $\Diag(\mathbb{Q}^*,\mathbb{Q}^*)$, if $f=\wp _{ai} \times \wp_{bi}$, for some $a,b\in \mathbb{R}^*$ such that $a/b\notin \mathbb{Q}^*$. \vspace{1mm} \\
\noindent $(6)$ $\left\{ q \begin{pmatrix} 1 & 0\\
\qc(\Omega,q \Omega) & [\Omega: q \Omega]q ^{-2} \end{pmatrix}
\suchthat  q \in \mathbb{Q}^* \right\}$, where $\Omega ={\langle 1,ai\rangle}_{\mathbb{Z}}$ for some $a\in \mathbb{R}^*$, if $f=(\wp _{ai}(u),v-\zeta _{ai}(u))$. \vspace{1mm} \\
\noindent $(7)$ $\left\{ q \begin{pmatrix} 1 & 0\\ \xi \qc(\Omega,q \Omega) & 1 \end{pmatrix}
\suchthat  q \in \mathbb{Q}^* \right\}$, where $\Omega ={\langle 1,ai\rangle}_{\mathbb{Z}}$ for some $a\in \mathbb{R}^*$ and $\xi \in \mathbb{R}\setminus \mathbb{Q}$,
if $f=(\wp _{ai}(u),e^v\widetilde{\sigma} _{ai,\xi}(u))$.\vspace{1mm} \\
\noindent $(8)$ $\left\{ q \begin{pmatrix} 1 & 0\\ \xi \qc(\Omega,q \Omega) & 1 \end{pmatrix}
\suchthat  q \in \mathbb{Q}^* \right\}$, where $\Omega ={\langle 1,ai\rangle}_{\mathbb{Z}}$ for some $a\in \mathbb{R}^*$ and $\xi \in \mathbb{R}\setminus a\mathbb{Q}$,
if $f=(\wp _{ai}(u),\frac{1}{2i}(e^{iv}\widetilde{\sigma}_{ai,\xi i}(u)-e^{-iv}\widetilde{\sigma}_{ai,-\xi i}(u)))$.
\end{proposition}
\begin{proof}Since $\Aut(\mathbb{C}^2,+,f)\cap \GL_2(\R)=\Aut(\mathbb{R}^2,+,f)$, the cases $(1)-(7)$ are an easy consequence of Proposition \ref{C-automorphisms 2} (and Lemma \ref{qc}, in the cases $(6)$ and $(7)$). We just give two examples. Firstly, we compute  $\Aut(\mathbb{R}^2,+,\exp\times \sin)$.
The map $\beta(u,v)=(u,iv)$ is an isomorphism from $(\mathbb{C}^2,+,\exp\times \sin)$ to $(\mathbb{C}^2,+,\exp\times \sin)$.
Therefore, $$\Aut(\mathbb{C}^2,+,\exp\times \sin)=\beta^{-1} \circ \Aut(\mathbb{C}^2,+,\exp\times \exp) \circ \beta$$equals $\{(au+ibv,-icu+dv) \suchthat a,b,c,d \in \mathbb{Q}, ad-bd\neq 0\}$
and hence $\Aut(\mathbb{R}^2,+,\exp\times \sin)=\{(au,dv) \,|\, a,d \in \mathbb{Q}^*\}$.

Next, we compute $\Aut(\mathbb{R}^2,+,\wp _{ai}\times \wp _{bi})$.
If $a/b\in \mathbb{Q}^*$ then, by Fact \ref{1dim R-classification}, the identity map is an isomorphism from $(\mathbb{R},+,\wp _{ai})$ to $(\mathbb{R},+,\wp _{bi})$ and, therefore, we can follow the proof of case
$(6.1)$ of Proposition \ref{C-automorphisms 2} for $\tau =1$.
Otherwise, $(\mathbb{R},+,\wp _{ai})$ and $(\mathbb{R},+,\wp _{bi})$ are not isomorphic and, therefore, we can proceed as in case $(6.2)$ of Proposition \ref{C-automorphisms 2}.

Finally, we consider case $(8)$.
Given $a\in \mathbb{R}^*$ and $\xi \in \mathbb{R}\setminus a\mathbb{Q}$, let
\[
g_{6,ai,i\xi}(u,v):=\left(\wp _{ai}(u),\frac{1}{2i}\big(e^{iv}\widetilde{\sigma}_{ai,i\xi}(u)-e^{-iv}\widetilde{\sigma}_{ai,-i\xi}(u)\big)\right).
\]
We recall that, in the proof of Theorem \ref{2dim R-classification}(II.iv), we showed that $\beta _1(u,v)=(u,iv)$ is a locally $\mathbb{C}$-Nash isomorphism from
$(\mathbb{C}^2,+,g_{6,ai,i\xi})$ to $(\mathbb{C}^2,+,g_{5,ai,i\xi})$.
Hence, we get that $\Aut(\mathbb{C}^2,+,g_{6,ai,i\xi})$ equals to
\[
\beta_1^{-1} \circ \Aut(\mathbb{C}^2,+,g_{5,ai,i\xi}) \circ \beta_1=\{q (u,\xi \qc(\Omega,q \Omega)u+v) \suchthat  q \in \mathbb{Q}^*\},
\]
which in turn must equal $\Aut(\mathbb{R}^2,+,g_{6,ai,i\xi})$, as required.
\end{proof}
\begin{remark}We would like to finish this section by pointing out that the results of \S \ref{Sproofs} and \S \ref{2dlng} provide a description of \emph{all} two dimensional connected locally $\K$-Nash groups. Indeed, in \cite[Prop.\,4.9]{BDOLCN} we proved that such a group is of the form $(\K^2,+,f)/\Gamma$ for some discrete subgroup $\Gamma$ of $\K^2$. Moreover, $(\K^2,+,f)/\Gamma_1$ and $(\K^2,+,g)/\Gamma_2$ are isomorphic if and only if there is an isomorphism $\alpha\in \GL_2(\K)$ between $(\K^2,+,f)$ and $(\K^2,+,g)$ such that $\alpha(\Gamma_1)=\Gamma_2$. On the other hand, in the complex case note that by the explicit calculations of the automorphism groups in Proposition \ref{C-automorphisms 2} and the computations carried out in the proofs of Theorem \ref{2dim C-classification}(II), we have a complete control of the isomorphisms between the different families of our classification. Similarly, in the real case we have such a control by the corresponding results Proposition \ref{R-automorphisms 2} and Theorem \ref{2dim R-classification}(II).
 
\end{remark}

\section{Consequences for (real) algebraic groups}\label{finalcomments}

By \cite[Thm.\,4.10]{BDOLCN} we know that the groups in Theorem \ref{2dim C-classification} are universal coverings of abelian irreducible algebraic groups.  In the following result we identify these algebraic groups.

\begin{prop}\label{algebraic groups}
The groups listed in Theorem\,\emph{\ref{2dim C-classification}} are  universal coverings of  abelian irreducible algebraic groups. For each group of type $(1c)$, the corresponding algebraic group is a direct product of two of the following ones:
$\C$, $\mathbb{C}^*$ and an elliptic curve.
For groups of type $(2c)$, $(3c)$ and $(4c)$, we get  (specific) extensions of  elliptic curves by $\C$,
 (specific) extensions of  elliptic curves by $\mathbb{C}^*$ and  abelian surfaces, respectively.
\end{prop}
\begin{proof}
We only need to study groups of  type (2c) and (3c). We use extensively the reference \cite{Hindry}, in which M. Hindry present explicit embeddings of two-dimensional algebraic groups in projective spaces
(and whose existence was implicitly proved by J.-P. Serre).

We begin with a locally $\mathbb{C}$-Nash group of the form
\[
\big(\mathbb{C}^2,+,(\wp _\omega (u),v-\zeta _\omega (u))\big),
\]
for some $\omega\in \mathbb{C}\setminus \mathbb{R}$.
Let $\Omega:={\langle 1,\omega \rangle}_{\mathbb{Z}}$ and consider the elliptic curve $E=:\mathbb{C}/\Omega $.
Recall that, since $\zeta'_\omega=-\wp_\omega$, for each $\lambda \in \Omega$, the difference $\zeta_\omega (u+\lambda)-\zeta_\omega(u)$ is a constant function
that we denote $\eta (\lambda)$.
Moreover, given $\lambda_1,\lambda_2\in \Omega$,
\[
\zeta_\omega(u+\lambda_1+\lambda_2)=\zeta_\omega(u+\lambda_1)+\eta(\lambda_2)=\zeta_\omega(u)+\eta(\lambda_1)+\eta(\lambda_2) 
\]
and, therefore, $\eta(\lambda_1+\lambda_2)=\eta(\lambda_1)+\eta(\lambda_2)$.

Consider the map $\varphi :\mathbb{C}^2 \rightarrow \mathbb{P}^5$, \begin{multline*}
\varphi (u,v):= (1:\wp _\omega (u):\wp '_\omega (u):v-\zeta _\omega (u):\wp _\omega (u)(v-\zeta _\omega (u))-\frac{1}{2}\wp '_{\omega}(u)\\
: \wp '_\omega (u)(v-\zeta _\omega (u))-2\wp ^2_\omega (u))
\end{multline*}
if $u\notin \Omega$, and $\varphi (u,v)=(0:0:1:0:0:v-\eta (u))$, otherwise.

By \cite[pag.\,28, Thm.]{Hindry}, $\varphi$ is a parametrisation of a quasi-projective subvariety $G$ of $\mathbb{P}^5$.
The addition in $\mathbb{C}^2$ induces an algebraic group structure on $G$ and  making $\varphi$ an homomorphism with kernel $\Omega_G :=\{ (\lambda , \eta (\lambda)) \suchthat \lambda \in \Omega \}$.
Moreover, $G$ is an extension of the elliptic curve $E$ by $(\mathbb{C}, +)$.
Therefore, it is enough to show that $\varphi$ is a locally $\mathbb{C}$-Nash map when we equip $\mathbb{C}^2$ with the locally
$\mathbb{C}$-Nash structure $(\mathbb{C}^2,+,(\wp _\omega (u),v-\zeta _\omega (u)))$.

We recall that $(\mathbb{C}^2,+,(\wp _\omega (u),v-\zeta _\omega (u)))$ is a notation that means that there exists $a\in \mathbb{C}$
such that $\phi_a(u,v):=(\wp _\omega (u+a),v-\zeta _\omega (u+a))$ is a chart of the identity of the locally
$\mathbb{C}$-Nash group structure on $(\mathbb{C}^2,+)$, see the Introduction.
On the other hand, there is an open neighbourhood $W$ of the identity of $G$ such that the projection
$$\pi:W\rightarrow \mathbb{C}^2:(x_0:\ldots :x_5)\mapsto(x_1/x_2,x_3/x_2)$$
is a chart of the locally $\mathbb{C}$-Nash structure of $G$. Thus, we have to prove that $ \pi\circ\varphi\circ\phi_a^{-1}$ is a $\C$-Nash map (see \cite[Prop.\,3.3]{BDOLCN}), and the latter is obtained by proving that $(\wp _\omega (u+a),v-\zeta _\omega (u+a))\in\mathbb{C}(\psi(u,v))^{alg}$, where
$\psi(u,v):=(\wp_\omega(u)(\wp'_\omega(u))^{-1},(v-\zeta _\omega (u))(\wp'_\omega(u))^{-1})$.
In fact,  it suffices to prove that
$\phi:=(\wp _\omega (u),v-\zeta _\omega (u))\in\mathbb{C}(\psi(u,v))^{alg}$.
The latter is done by noting that both
$\mathbb{C}(\phi(u,v))$ and $\mathbb{C}(\psi(u,v))$ are subfields of $\mathbb{C}(\wp _\omega (u),v-\zeta _\omega (u),\wp'_\omega(u))$ which has  transcendence degree 2
because $\wp'_\omega(u)$ is algebraic over $\mathbb{C}(\wp_\omega(u))$.

Next, we consider the case of a locally $\mathbb{C}$-Nash group of the form
\[
\big(\mathbb{C}^2,+,(\wp _\omega (u),e^v\widetilde{\sigma}_{\omega,\xi} (u))\big),
\]
for some $\omega\in \mathbb{C}\setminus \mathbb{R}$ and $\xi \notin {\langle 1,\omega\rangle}_{\mathbb{Q}}$.
Again, we consider the elliptic curve $E:=\mathbb{C}/\Omega $, where $\Omega:={\langle 1,\omega \rangle}_{\mathbb{Z}}$.

Consider the functions
\begin{align*}
\Phi(u,v) &:=\frac{\sigma _\omega (u-\xi)}{\sigma _\omega (u)\sigma _\omega (\xi)}e^{v+u\zeta _\omega (\xi)}=\tfrac{1}{\sigma _\omega (\xi)}\widetilde{\sigma}_{\omega,\xi} (u)e^{v+u\zeta _\omega (\xi)},\\
F(u,v) &:= \frac{\wp '_\omega (u)+\wp ' _\omega (\xi )}{\wp _\omega (u)- \wp _\omega (\xi)}.
\end{align*}
Then, by \cite[pags.\,32-34]{Hindry} (see also D. Caveny and R. Tubbs \cite[\textsection 2]{Caveny_Tubbs}), we have that
$\varphi:\mathbb{C}^2\rightarrow \mathbb{P}^8$ is a parametrisation of a quasi-projective subvariety $G$ of $\mathbb{P}^8$, where
\medskip
\begin{multline*}
\varphi (u,v)=(1:\wp _\omega (u):\wp '_\omega (u):\Phi(u,v):\Phi(-u,-v):\wp _\omega (u)\Phi(u,v)\\
:\wp _\omega (u)\Phi(-u,-v):\Phi(u,v)F(u):\Phi(-u,-v)F(-u))
\end{multline*}
\medskip
if $u\notin \Omega$ and
\[
\varphi (u,v)=(0:0:1:0:0: \tfrac{\sigma_\omega(u-\xi)}{\sigma_\omega(\xi)}e^{v+u\zeta_\omega(\xi)}:\tfrac{\sigma_\omega(-u-\xi)}{\sigma_\omega(\xi)}e^{-v-u\zeta_\omega(\xi)}:0:0)
\]
otherwise.

The addition in $\mathbb{C}^2$ induces an algebraic group structure on $G$ and the kernel of $\varphi$ is
$\Omega_G :=\{ (\lambda, \xi \eta (\lambda)-\lambda \zeta _{\omega} (\xi )+ 2\pi i m) \suchthat \lambda \in \Omega, m\in \mathbb{Z} \}$,
which is a discrete subgroup of $\mathbb{C}^2$ of rank $3$.
Moreover, $G$ is an extension of the elliptic curve $E$ by $\mathbb{C}^*$.

Since $(u,v+u\zeta _\omega(\xi)):(\mathbb{C}^2,+, (\wp _\omega(u), \Phi(u,v))\to (\mathbb{C}^2,+, (\wp _\omega(u), e^v\widetilde{\sigma} _{\omega ,\xi} (u)))$ is an isomorphism, it is enough to prove that $\varphi$ is a locally $\mathbb{C}$-Nash map when we equip $\mathbb{C}^2$ with the locally
$\mathbb{C}$-Nash structure $(\mathbb{C}^2,+, (\wp _\omega(u), \Phi(u,v))$.

As above, the  notation means that there exists $a\in \mathbb{C}$
such that $\phi_a(u,v):=(\wp _\omega (u+a),\Phi(u+a,v))$ is a chart of the identity of the locally
$\mathbb{C}$-Nash group structure on $(\mathbb{C}^2,+)$. On the other hand, there is an open neighbourhood $W$ of the identity of $G$ such that the projection
$$\pi:W\rightarrow \mathbb{C}^2:(x_0:\ldots :x_8)\mapsto(x_1/x_2,x_3/x_2)$$ is a chart of the locally $\mathbb{C}$-Nash structure of $G$. The situation now reduces to prove that
$\phi:=(\wp _\omega (u),\Phi(u,v))$ is algebraic over $\mathbb{C}(\psi(u,v))$, where
$\psi(u,v):=(\wp_\omega(u)(\wp'_\omega(u))^{-1},\Phi(u,v)(\wp'_\omega(u))^{-1})$. 
We again conclude noting  that both $\mathbb{C}(\phi(u,v))$ and $\mathbb{C}(\psi(u,v))$ are subfields of
$\mathbb{C}(\wp _\omega (u), \Phi(u,v),\wp'_\omega(u))$ which has transcendence degree $2$
because $\wp'_\omega(u)$ is algebraic over $\mathbb{C}(\wp_\omega(u))$.
\end{proof}
\begin{remark}\label{rmk:realalgebraicgroups}We point out that if in the proof of Proposition \ref{algebraic groups} the relevant lattices are invariant (and the $\xi$ parameter is real in case (3c)) then the obtained algebraic groups are defined over $\R$.
\end{remark}

Obviously, we could infer Theorem\,\ref{2dim C-classification} from the fact that every simply connected abelian locally $\C$-Nash group is the universal covering of an irreducible group \cite[Thm.\,4.10]{BDOLCN} and from the classification of
two-dimensional abelian irreducible algebraic groups (see J.-P. Serre \cite{Serre} and also P. Corvaja, D. Masser and U. Zannier \cite{Corvaja_Masser_Zannier}).
However, we would like to stress the other direction:
we have provided a new proof with analytic methods of the
classification of two-dimensional abelian irreducible complex algebraic groups.

\begin{corollary}\label{isogenous corollary}
Every two-dimensional abelian irreducible complex algebraic group is isogenous to  a group of one and only one of the following types:

\emph{(1ac)} A direct product of  two among $\C,\,\C^*$ and an elliptic curve.

\emph{(2ac)} An extension of an elliptic curve  by $\C$.

\emph{(3ac)} An extension of an elliptic curve  by $\C^*$.

\emph{(4ac)} A simple abelian surface.
\end{corollary}
\begin{proof}
Let $G$ be a two-dimensional abelian irreducible algebraic group.
Then, by \cite[Lem.\,3.9]{BDOLCN} and \cite[Prop.\,3.7]{BDOLCN}, its universal covering $\widetilde{G}$ is a simply connected two-dimensional
abelian locally $\mathbb{C}$-Nash group.
Then, by Proposition\,\ref{algebraic groups}, we have that $\widetilde{G}$ is the universal covering of one of the algebraic groups listed in
Proposition\,\ref{algebraic groups}.
In particular, we have a local isomorphism which is a locally $\mathbb{C}$-Nash map between $G$ and one these groups,
so the result follows from \cite[Thm.\,3.10]{BDOLCN}.
\end{proof}

Our final purpose is to provide a classification of algebraic groups defined over $\R$. In Theorem \ref{2dim R-classification} we already showed a classification of the two-dimensional simply-connected abelian  locally Nash groups, and each one is the universal covering of the connected component of the real points of an irreducible algebraic group defined over $\R$. Therefore, it only remains to recognise such  algebraic groups defined over $\R$. In fact, by Remark \ref{rmk:realalgebraicgroups} it is enough to study the groups of type (4r). And for that aim, we need to make  some comments about the \emph{descent datum} of two-dimensional irreducible algebraic groups (mentioned at  \S Introduction). That is, we recall first the \emph{descent of the base field} from \cite[Ch.V,\S 4]{Serre}. Let $H$ be a $\C$-algebraic group, \emph{i.e.}, an algebraic group defined over $\C$. Recall that algebraic groups are quasi-projective by Chevalley's theorem. Let us denote by $\hat{H}$ the algebraic group obtained by means of complex conjugation (this is denoted by $H^\sigma$ in \cite{Serre}).

Now, let $G$ be an $\R$-algebraic group and let $f:G\rightarrow H$ be a biregular isomorphism defined over $\C$. Let $\hat{f}: \hat{G} \rightarrow  \hat{H}$ denote the isomorphism obtained by applying  complex conjugation to $f$. Since $\hat{G}=G$, we get a biregular isomorphism $h:=\hat{f} \circ f^{-1}:H \rightarrow \hat{H}$ defined over $\C$ such that $\text{id}=\hat{h}\circ h$. Reciprocally, given a biregular isomorphism $h:H \rightarrow \hat{H}$ defined over $\C$ such that $\text{id}=\hat{h}\circ h$ there is a unique $\R$-algebraic group $G$ up to $\R$-isomorphism and a biregular isomorphism $f:G\rightarrow H$ such that $h:=\hat{f} \circ f^{-1}$ (see \cite[Ch.V,\S 4, Corollary 2]{Serre}). On the other hand, the map $\tau:= \sigma_H^{-1}\circ h$, where $\sigma_H: H\rightarrow \hat{H}$ is the complex conjugation map,   is clearly an antiholomorphic involution (\emph{i.e.}, $\tau^2=\id$ and $\sigma_H\circ \tau$ is holomorphic). Moreover, it is easy to show that the fixed points of $\tau$ are exactly the real points of the group $G$. In fact, in case that $H$ is projective, if an anti-holomorphic involution $\tau: H\rightarrow H$ is given, then $h:=\sigma_H\circ \tau:H\rightarrow \hat{H}$ is a biregular isomorphism such that $\hat{h}\circ h=\text{id}$.

For example, for the algebraic group $H=\C^*$ we have the biregular isomorphism $h:H\rightarrow \hat{H}=H$ given by the (multiplicative) inverse map. The corresponding algebraic group defined over $\R$ is $G=\text{SO}_2(\R)$. In fact, we can see this using the locally $\C$-Nash category. The map $h$ induces the automorphism $h_1=-\id$ of $(\C,+,\exp)$. On the other hand, we have the isomorphism $\alpha_1:(\C,+,\sin)\rightarrow (\C,+,\exp):u\mapsto iu$, which satisfies $\hat{\alpha}_1\circ \alpha_1^{-1}=h_1$. Since $(\C,+,\sin)$ is the universal covering of $\text{SO}_2(\C)$, from $\alpha_1$ it is easy to compute a biregular isomorphism $\alpha:\text{SO}_2(\C)\rightarrow \C^*$ such that $\hat{\alpha}\circ \alpha^{-1}=h$. Similarly, we could have also an elliptic curve $H$ of lattice ${\langle 1,\omega \rangle}$, where $\omega\in \R i$, and the isomorphism $h:H\rightarrow \hat{H}=H$ given by the inverse map. Again $h$ induces the automorphism $h_1=-\id$ of $(\C,+,\wp_{\langle 1,\omega \rangle})$, and note that the isomorphism $\alpha_1:(\C,+,\wp_{\langle i,\omega i \rangle})\rightarrow (\C,+,\wp_{\langle 1,\omega \rangle}): u \mapsto iu$ satisfies $\hat{\alpha}_1\circ \alpha_1^{-1}=h_1$.  Since $(\C,+,\wp_{\langle i,\omega i \rangle})$ is the universal covering of the elliptic curve $G$ of (invariant) lattice $\wp_{\langle i,\omega i \rangle}$, from $\alpha_1$ it is easy to compute a biregular isomorphism $\alpha:G\rightarrow H$ such that $\hat{\alpha}\circ \alpha^{-1}=h$.

We finish this brief reminder recalling the descent of the base field for extensions of groups. Namely, let $H$ be an algebraic extension of $E$ by $L$,
$$0\rightarrow L \rightarrow H \rightarrow E \rightarrow 0,$$
and let $h:H\rightarrow H^\sigma$ be a biregular isomorphism such that $\hat{h}\circ h=\id$ and $h(L)=\hat{L}$. In particular, the map $h_L:=h|_L:L\rightarrow \hat{L}$ is also biregular with $\hat{h}_L\circ h_L=\id$ and $h$ induces a biregular isomorphism $h_E:E\rightarrow \hat{E}$ such that $\hat{h}_E\circ h_E=\id$. 
By the above discussion, there is an $\R$-algebraic group $G$ and a biregular isomorphism $f:G\rightarrow H$ such that  $h=\hat{f} \circ f^{-1}$. Consider the algebraic subgroup  $L_1:=f^{-1}(L)$ of $G$. From the properties of $h$ and $f$ it is easy to deduce that $L_1$ is defined over $\R$ and 
the biregular isomorphism $f_{L_1}:=f|_{L_1}:L_1\rightarrow L$ satisfies $h_{L}=\hat{f}_{L_1}\circ f_{L_1}^{-1}$. Moreover, for the $\R$-algebraic group $E_1=G/L_1$ we have an induced isomorphism $f_{E_1}:E_1\rightarrow E$ which satisfies $h_{E}=\hat{f}_{E_1}\circ f_{E_1}^{-1}$. Hence, we have shown that $G$ is an algebraic extension of $E_1$ by $L_1$, where $E_1$ and $L_1$ are the $\R$-algebraic groups related to $h_L$ and $h_E$ respectively. 

We now study the above descent of the base field on the case where $H$ is the abelian two-dimensional irreducible algebraic group of Proposition \ref{algebraic groups} which is an extension of an elliptic curve by $\C^*$. In other words, the universal covering of $H$ is $(\mathbb{C}^2,+,(\wp _\omega(u), e^v\widetilde{\sigma} _{\omega,\xi}(u)))$ for some $\omega\in \mathbb{C}\setminus \mathbb{R}$ and $\xi \notin {\langle 1,\omega \rangle}_{\mathbb{Q}}$.  Let $h:H\rightarrow \hat{H}$ be a biregular isomorphism defined over $\C$ such that $\hat{h}\circ h=\text{id}$.
We will show that there are only two possibilities for such an $h$, and for each one we will describe the algebraic group $G$ defined over $\R$ for which there exists an isomorphism $\alpha:G\rightarrow H$ such that $h=\hat{\alpha}\circ \alpha^{-1}$.

Note that the universal covering of $\hat{H}$ is $(\mathbb{C}^2,+,(\wp_{\overline{\omega}}(u), e^v\widetilde{\sigma}_{\overline{\omega},\overline{\xi}} (u)))$, and $h$ induces a locally $\C$-Nash isomorphism
$$h_1:(\mathbb{C}^2,+,(\wp _\omega(u), e^v\widetilde{\sigma} _{\omega,\xi} (u))) \rightarrow (\mathbb{C}^2,+,(\wp_{\overline{\omega}}(u), e^v\widetilde{\sigma}_{\overline{\omega},\overline{\xi}} (u))),$$
such that $\hat{h}_1\circ h_1=\text{id}$, where $\hat{h}_1$ is defined by the complex conjugation of the matrix  associated to $h_1$.
Arguing as in the proof of Theorem \ref{2dim R-classification}, there exist an isomorphism
$$\beta:(\mathbb{C}^2,+,(\wp _\omega(u), e^v\widetilde{\sigma} _{\omega,\xi} (u))) \rightarrow (\mathbb{C}^2,+,(\wp _{\omega_0}(u), e^v\widetilde{\sigma} _{\omega_0,\xi_0} (u)))$$
where $\omega_0\in \R i$ and $\xi_0\in \R\setminus \Q$. We can also consider the automorphism $h_2:=  \hat{\beta}\circ h_1 \circ \beta^{-1}$ of $(\mathbb{C}^2,+,(\wp _{\omega_0}(u), e^{v}\widetilde{\sigma} _{\omega_0,\xi_0} (u)))$, which satisfies $\hat{h}_2\circ h_2=\text{id}$.

All in all, we can assume from the beginning that the universal covering of $H$ is of the form $(\mathbb{C}^2,+,(\wp _{\omega}(u), e^v\widetilde{\sigma} _{\omega,\xi} (u)))$ where $\omega\in \R i$ and $\xi\in \R\setminus \Q$. Thus, $H$ is an algebraic group defined $\R$, the extension of the elliptic curve of lattice ${\langle 1,\omega\rangle}$ by $\C^*$. Now let $h_1$ be the automorphism of $(\mathbb{C}^2,+,(\wp _{\omega}(u), e^v\widetilde{\sigma} _{\omega,\xi} (u))$ induced by the biregular isomorphism $h:H\rightarrow \hat{H}=H$. Since $\hat{h}_1\circ h_1=\id$ we deduce from Proposition \ref{C-automorphisms 2} that $h_1=\id$ or $h_1=-\id$. If $h_1=\id$ then $h=\id$ and the corresponding algebraic group $G$ defined over $\R$ is $H$ again. If $h_1=-\id$ then $h$ is the (multiplicative) inverse of $H$. Using the explicit algebraic description of $H$ provided by \cite{Hindry} it should be possible to compute such a description for the corresponding algebraic group $G$ defined over $\R$ that we obtain in this case. The latter is beyond our interests: we just point out that -- because of the above reminder, concerning the descent of the base field of extensions-- we clearly get that $G$ is an extension of an elliptic curve of lattice ${\langle i,\omega i \rangle}$ by the algebraic group $\text{SO}_2(\C)$.

Hence, we have shown the existence of an algebraic group defined over $\R$ which is the extension of an elliptic curve of invariant lattice by $\text{SO}_2(\C)$. The latter is biregularly $\C$-isomorphic to an algebraic group defined over $\R$ of type (3ac) which is an extension of an elliptic curve of invariant lattice by $\C^*$, but it is not biregularly $\R$-isomorphic.
\begin{prop}\label{realalgebraicgroups}
The groups listed in Theorem\,\emph{\ref{2dim R-classification}} are  universal coverings of the connected components of the real points of abelian irreducible algebraic groups defined over $\R$. For each group of type $(1r)$, the corresponding algebraic group is a direct product of two of the following ones:
  $\C$, $\C^*$, $\text{\emph{SO}}_2(\C)$, and an elliptic curve of invariant lattice.
For groups of type $(2r)$, $(3r)$ and $(4r)$, we get  (specific) extensions of  elliptic curves of invariant lattice by $\C$,
  by $\C^*$, and by $\text{\emph{SO}}_2(\C)$, respectively. For groups of type $(5r)$ we get a simple abelian surface defined over $\R$.
\end{prop}
\begin{proof}The result for the groups (1r)-(3r) and (5r) is a consequence of Proposition \ref{algebraic groups} and Remark \ref{algebraic groups}. For groups of type (4r), it follows by the discussion above.
\end{proof}

We finally state one of the main results of this paper, which follows from Proposition \ref{realalgebraicgroups} and \cite[Prop.\,3.12]{BDOLCN}.
\begin{corollary}\label{real isogenous corollary}
Every two-dimensional abelian irreducible complex algebraic group defined over $\R$ is isogenous to a group of one and only one of the following types:

\emph{(1ar)} A direct product of  two among $\C,\,\C^*$, $\text{\emph{SO}}_2(\C)$ and an elliptic curve of invariant lattice.

\emph{(2ar)} An extension of an elliptic curve of invariant lattice by $\C$.

\emph{(3ar)} An extension of an elliptic curve  of invariant lattice by $\C^*$.

\emph{(4ar)} An extension of an elliptic curve  of invariant lattice by $\text{\emph{SO}}_2(\C)$.

\emph{(5ar)} A simple abelian surface defined over $\R$.
\end{corollary}
\bibliographystyle{plain}
\bibliography{research}
\end{document}